\documentclass[final]{siamonline171218}
\usepackage{fullpage}
\usepackage{color}
\usepackage{mathtools}
\usepackage{url}

\usepackage{amssymb,amsfonts}
\usepackage{algorithmic, algorithm}

\usepackage{mathtools} 
\usepackage[toc,page]{appendix}

\usepackage{subcaption}

\usepackage{cleveref}

\definecolor{Blue}{rgb}{0,0,1}
\definecolor{Red}{rgb}{1,0,0}
\definecolor{Magenta}{rgb}{1,0,1}

\newcommand{\spatialVar}{x}
\newcommand{\spatialVarArg}[1]{x_{#1}}

\newcommand{\RR}[1]{\ensuremath{\mathbb{R}^{ #1 }}}
\newcommand{\natno}[1]{\ensuremath{\mathbb{N}^{ #1 }}}
\newcommand{\natnoZero}[1]{\ensuremath{\mathbb{N}_0^{ #1 }}}
\newcommand{\defeq}{\coloneqq}

\newcommand{\iteration}[2]{#1^{(#2)}}
\newcommand{\relaxation}{\omega}
\newcommand{\relaxationTruth}{\bar\omega}
\newcommand{\nsubsystems}{n}

\newcommand{\samplespace}{\Omega}
\newcommand{\eventset}{\mathcal F}
\newcommand{\probabilityOp}{P}
\newcommand{\probabilitySpace}{(\samplespace,\eventset,\probabilityOp)}

\newcommand{\domain}{\Omega_x}
\newcommand{\subdomainArg}[1]{\Omega_{x}^{#1}}
\newcommand{\neighborSet}{\mathcal N}

\newcommand{\bds}[1]{\boldsymbol{#1}}

\newcommand{\outputssymb}{x}
\newcommand{\outputssymbCap}{X}
\newcommand{\generalRVsymb}{{{w}}}
\newcommand{\generalRV}{{{\mathsf \generalRVsymb}}}
\newcommand{\generalRVsymbVec}{{\boldsymbol{ \generalRVsymb}}}
\newcommand{\generalRVVec}{{\boldsymbol{\mathsf \generalRVsymb}}}
\newcommand{\outputsRV}{{\boldsymbol{\mathsf \outputssymb}}}
\newcommand{\outputsRVArg}[1]{{\outputsRV}_{#1}}
\newcommand{\outputsRVi}{\outputsRVArg{i}}
\newcommand{\noutputsRV}{n_\outputsRV}
\newcommand{\noutputsRVArg}[1]{n_{\outputsRV,#1}}
\newcommand{\noutputsRVi}{\noutputsRVArg{i}}
\newcommand{\outputsRVDummy}{{\outputsRV}}
\newcommand{\outputsRVDummyArg}[1]{\outputsRV_{#1}}

\newcommand{\outputsRVDummyElement}[1]{\outputssymb_{#1}}
\newcommand{\outputsRVTrue}{\outputsRV_\star}
\newcommand{\outputsRVTmp}{\hat\outputsRV}
\newcommand{\outputsTrue}{\outputs_\star}

\newcommand{\AAres}{\boldsymbol {\mathsf f}}
\newcommand{\memory}{m}
\newcommand{\outputs}{{\boldsymbol{ \outputssymb}}}
\newcommand{\outputsArg}[1]{{\outputs}_{#1}}
\newcommand{\outputsi}{\outputsArg{i}}
\newcommand{\noutputs}{n_\outputs}
\newcommand{\noutputsArg}[1]{n_{\outputs,#1}}
\newcommand{\noutputsi}{\noutputsArg{i}}

\newcommand{\outputsTruthRV}{\bar{\boldsymbol{\mathsf \outputssymb}}}
\newcommand{\outputsTruthRVArg}[1]{{\outputsTruthRV}_{#1}}
\newcommand{\outputsTruthRVi}{\outputsTruthRVArg{i}}

\newcommand{\noutputsTruthRVArg}[1]{n_{\outputsTruthRV,#1}}
\newcommand{\noutputsTruthRVi}{\noutputsTruthRVArg{i}}

\newcommand{\outputsTruthRVTrue}{\outputsTruthRV_\star}

\newcommand{\imageoutputsRVArg}[1]{\mathbb \outputssymbCap_{#1}}
\newcommand{\imageoutputsRVi}{\imageoutputsRVArg{i}}
\newcommand{\imageoutputsRV}{\mathbb \outputssymbCap}
\newcommand{\imageoutputsTruthRVArg}[1]{\bar{\mathbb \outputssymbCap}_{#1}}
\newcommand{\imageoutputsTruthRVi}{\imageoutputsTruthRVArg{i}}

\newcommand{\outputsPredictorRV}{\tilde{\boldsymbol{\mathsf x}}}
\newcommand{\outputsPredictorRVArg}[1]{{\outputsPredictorRV}_{#1}}

\newcommand{\inputsendoPredictorRV}{\bar{\boldsymbol{\mathsf y}}}
\newcommand{\inputsendoPredictorRVArg}[1]{{\inputsendoPredictorRV}_{#1}}
\newcommand{\inputsendoPredictorRVi}{\inputsendoPredictorRVArg{i}}

\newcommand{\inputsexosymb}{u}

\newcommand{\inputsexoRV}{{\boldsymbol{\mathsf \inputsexosymb}}}
\newcommand{\inputsexoRVArg}[1]{{\inputsexoRV}_{#1}}
\newcommand{\inputsexoRVi}{\inputsexoRVArg{i}}

\newcommand{\ninputsexoRV}{n_\inputsexoRV}
\newcommand{\ninputsexoRVArg}[1]{n_{\inputsexoRV,#1}}
\newcommand{\ninputsexoRVi}{\ninputsexoRVArg{i}}
\newcommand{\inputsexoRVDummy}{\inputsexoRV}
\newcommand{\inputsexoRVGlobal}{{{\mathsf
\inputsexosymb}}^\text{global}}
\newcommand{\inputsexoRVGlobalArg}[1]{\inputsexoRVGlobal_{#1}}
\newcommand{\inputsexo}{{\boldsymbol{ \inputsexosymb}}}
\newcommand{\inputsexoArg}[1]{{\inputsexo}_{#1}}
\newcommand{\inputsexoi}{\inputsexoArg{i}}
\newcommand{\ninputsexo}{n_\inputsexo}
\newcommand{\ninputsexoArg}[1]{n_{\inputsexo,#1}}
\newcommand{\ninputsexoi}{\ninputsexoArg{i}}

\newcommand{\inputsendosymb}{y}
\newcommand{\inputsendoRV}{{\boldsymbol{\mathsf \inputsendosymb}}}
\newcommand{\inputsendoRVArg}[1]{{\inputsendoRV}_{#1}}
\newcommand{\inputsendoRVi}{\inputsendoRVArg{i}}
\newcommand{\ninputsendoRV}{n_\inputsendoRV}
\newcommand{\ninputsendoRVArg}[1]{n_{\inputsendoRV,#1}}
\newcommand{\ninputsendoRVi}{\ninputsendoRVArg{i}}

\newcommand{\inputsendo}{{\boldsymbol{ \inputsendosymb}}}
\newcommand{\inputsendoArg}[1]{{\inputsendo}_{#1}}
\newcommand{\inputsendoi}{\inputsendoArg{i}}

\newcommand{\ninputsendoArg}[1]{n_{\inputsendo,#1}}
\newcommand{\ninputsendoi}{\ninputsendoArg{i}}

\newcommand{\qoioutputssymb}{z}
\newcommand{\qoioutputsRV}{{\boldsymbol{\mathsf \qoioutputssymb}}}

\newcommand{\nqoioutputsRV}{n_\qoioutputsRV}

\newcommand{\outputsToQoIRV}{\boldsymbol{\mathsf I}_{\outputsRV}^{\qoioutputsRV}}

\newcommand{\RVspaceDim}[1]{\mathbb{V}^{#1}}
\newcommand{\spaceDim}[1]{\mathbb{R}^{#1}}
\newcommand{\relativeResidual}{r}

\newcommand{\propagatorsymb}{f}
\newcommand{\propagatorRV}{{\boldsymbol{\mathsf \propagatorsymb}}}
\newcommand{\propagatorRVArg}[1]{{\propagatorRV}_{#1}}
\newcommand{\propagatorRVVec}{{\boldsymbol{\mathsf \propagatorsymb}_\text{vec}}}

\newcommand{\propagatorRVi}{\propagatorRVArg{i}}
\newcommand{\propagator}{{\boldsymbol{ \propagatorsymb}}}
\newcommand{\propagatorArg}[1]{{\propagator}_{#1}}
\newcommand{\propagatori}{\propagatorArg{i}}

\newcommand{\propagatorTruthRV}{\bar{\boldsymbol{\mathsf
\propagatorsymb}}}
\newcommand{\propagatorTruthRVArg}[1]{\propagatorTruthRV_{#1}}
\newcommand{\propagatorTruthRVi}{\propagatorTruthRVArg{i}}
\newcommand{\propagatorTruthRVVec}{\bar{\boldsymbol{\mathsf
\propagatorsymb}}_\text{vec}}

\newcommand{\image}[1]{\mathrm{Im{(#1)}}}
\newcommand{\imageProp}{\mathbb X}
\newcommand{\imagePropArg}[1]{\imageProp_{#1}}
\newcommand{\imagePropF}{\image{\propagatorRV}}
\newcommand{\imagePropFArg}[1]{\image{\propagatorRVArg{#1}}}

\newcommand{\imagePropFTruth}{\image{\propagatorTruthRV}}
\newcommand{\imagePropFTruthArg}[1]{\image{\propagatorTruthRVArg{#1}}}

\newcommand{\imagePropTruth}{\bar{\mathbb X}}
\newcommand{\imageOrthProp}{\mathbb X^\perp}
\newcommand{\imageOrthPropArg}[1]{\imageOrthProp_{#1}}
\newcommand{\projImageNo}{\boldsymbol{\mathsf P}}
\newcommand{\projImageTruthNo}{\bar{\boldsymbol{\mathsf P}}}
\newcommand{\projImage}[1]{\projImageNo #1}
\newcommand{\projImageArg}[1]{\projImageNo_{#1}}
\newcommand{\projImageTruthArg}[1]{\projImageTruthNo_{#1}}
\newcommand{\projImagePar}[1]{\projImageNo(#1)}
\newcommand{\projOrthImageNo}{\projImageNo^\perp}
\newcommand{\projOrthImage}[1]{\projOrthImageNo #1 }
\newcommand{\projOrthImageArg}[1]{\projOrthImageNo_{#1}}
\newcommand{\constant}{\boldsymbol{\mathsf c}}
\newcommand{\identity}{\boldsymbol{I}}
\newcommand{\identityRV}{\boldsymbol{\mathsf I}}

\newcommand{\propagatorFPsymb}{h}
\newcommand{\propagatorFPRV}{{\boldsymbol{\mathsf \propagatorFPsymb}}}

\newcommand{\propagatorFPTruthRV}{\bar{\boldsymbol{\mathsf
\propagatorFPsymb}}}
\newcommand{\propagatorFPTruthRVPerp}{\bar{\boldsymbol{\mathsf \propagatorFPsymb}}^{\perp}}
\newcommand{\propagatorTruthRVPerp}{\bar{\boldsymbol{\mathsf \propagatorsymb}}^{\perp}}
\newcommand{\propagatorTruthRVPerpArg}[1]{\propagatorTruthRVPerp_{#1}}

\newcommand{\FPoperatorsymb}{h}
\newcommand{\FPoperatorRV}{{\boldsymbol{\mathsf \FPoperatorsymb}}}

\newcommand{\FPoperatorRVJac}{{\FPoperatorRV}_{\mathrm{J}}}
\newcommand{\FPoperatorRVGS}{{\FPoperatorRV}_{\mathrm{GS}}}
\newcommand{\FPoperatorTruthRVJac}{\bar{\FPoperatorRV}_{\mathrm{J}}}
\newcommand{\FPoperatorTruthRVGS}{\bar{\FPoperatorRV}_{\mathrm{GS}}}

\newcommand{\sampleRV}{\bds{\mathsf E}}
\newcommand{\sampleRVArgs}[2]{\sampleRV^{#1}_{#2}}
\newcommand{\sampleinputsendoRVArg}[1]{\sampleRVArgs{\inputsendoRV}{#1}}
\newcommand{\sampleinputsexoRVArg}[1]{\sampleRVArgs{\inputsexoRV}{#1}}
\newcommand{\sampleoutputsRVArg}[1]{\sampleRVArgs{\outputsRV}{#1}}

\newcommand{\outputsToInputsRV}{\boldsymbol{\mathsf I}_{\outputsRV}^{\inputsendoRV}}

\newcommand{\resRV}{{\boldsymbol{\mathsf r}}}
\newcommand{\res}{{\boldsymbol{ r}}}
\newcommand{\resTruthRV}{\bar {\boldsymbol{\mathsf {r}}}}
\newcommand{\zero}{{\mathbf 0}}
\newcommand{\card}[1]{|#1|}

\newcommand{\underlyingRVuniv}{{\xi}}
\newcommand{\underlyingRVunivArg}[1]{\underlyingRVuniv_{#1}}
\newcommand{\underlyingRV}{{\boldsymbol{\underlyingRVuniv}}}
\newcommand{\stochasticDim}{{n_\underlyingRV}}
\newcommand{\PCEbasis}{{\psi}}
\newcommand{\PCEbasisUnivArg}[1]{\tilde{\PCEbasis}_{#1}}
\newcommand{\PCEbasisArg}[1]{\PCEbasis_{#1}}
\newcommand{\PCEbasisArgs}[2]{\PCEbasis_{#1}(#2)}

\newcommand{\PCEdimArgs}[2]{\card{\multiindexSetArgs{#1}{#2}}}
\newcommand{\PCEcoeffsArgs}[3]{\boldsymbol{#1}_{#2,#3}}
\newcommand{\PCEcoeffsArg}[2]{{#1}_{#2}}
\newcommand{\PCEcoeffsGlobalArgs}[3]{{#1}_{#2,#3}^\text{global}}
\newcommand{\PCEcoeffsGlobalFieldArgs}[2]{{#1}_{#2}^\text{global}}
\newcommand{\PCEcoeffsFieldArgs}[3]{{#1}_{#2}^{#3}}
\newcommand{\outputsPCEcoeffsArg}[2]{\PCEcoeffsArgs{\outputssymb}{#1}{#2}}
\newcommand{\outputsTruthPCEcoeffsArg}[2]{\PCEcoeffsArgs{\bar\outputssymb}{#1}{#2}}
\newcommand{\inputsendoPCEcoeffsArg}[2]{\PCEcoeffsArgs{\inputsendosymb}{#1}{#2}}
\newcommand{\inputsexoPCEcoeffsArg}[2]{\PCEcoeffsArgs{\inputsexosymb}{#1}{#2}}
\newcommand{\generalPCECoeffsArg}[1]{\PCEcoeffsArg{\generalRVsymb}{#1}}
\newcommand{\generalVecPCECoeffsArg}[1]{\PCEcoeffsArg{\generalRVsymbVec}{#1}}

\newcommand{\permutationsymb}{p}
\newcommand{\permutationTuple}{\bds{\permutationsymb}}
\newcommand{\permutationTupleTruth}{\bar\permutationTuple}
\newcommand{\permutationArg}[1]{\permutationsymb_{#1}}
\newcommand{\permutationTupleIt}[1]{\permutationTuple^{(#1)}}

\newcommand{\permutationRVMatsymb}{\mathsf P}
\newcommand{\permutationRVMatrixArgs}[2]{\bds{\permutationRVMatsymb}^{#1}_{#2}}

\newcommand{\lowerBlockRV}{\bds{\mathsf{L}}}

\newcommand{\upperBlockRV}{\bds{\mathsf{U}}}
\newcommand{\lowerBlockRVArg}[1]{\lowerBlockRV_{#1}}
\newcommand{\upperBlockRVArg}[1]{\upperBlockRV_{#1}}

\newcommand{\indexSet}{\mathcal I}
\newcommand{\indexSetArg}[1]{\indexSet_{#1}}
\newcommand{\indexSetArgs}[2]{\indexSet_{#1}^{#2}}

\newcommand{\nsequential}{n_\text{seq}}
\newcommand{\nsequentialArg}[1]{n_{\text{seq},#1}}

\newcommand{\DAGalgname}{\texttt{DAG}}
\newcommand{\Jacobialgname}{\texttt{Jacobi}}
\newcommand{\GaussseidelAlgName}{\texttt{GaussSeidel}}
\newcommand{\AAalgname}{\texttt{AndersonAcceleration}}

\newcommand{\normal}[1]{\mathcal N(#1)}
\newcommand{\lipschitzArg}[1]{L_{#1}(\inputsexoRV)}
\newcommand{\lipschitz}[1]{L_{#1}}
\newcommand{\norm}[1]{\|#1\|}

\newcommand{\multiindex}{\boldsymbol j}
\newcommand{\multiindexArg}[1]{ j_{#1}}
\newcommand{\multiindexSet}{\mathcal J}
\newcommand{\multiindexSetArgs}[2]{\multiindexSet_{#1,#2}}
\newcommand{\multiindexGlobalSetArgs}[1]{\multiindexSet_{\text{global}}}
\newcommand{\multiindexSetLow}{\multiindexSet_\text{low}}
\newcommand{\multiindexSetHigh}{\multiindexSet_\text{high}}

\newcommand{\multiindexSetInf}{\natnoZero{\stochasticDim}}
\newcommand{\errorVerb}{\outputsTruthRVTrue- \outputsRVTrue}
\newcommand{\errorInPlaneVerb}{\projImage{\outputsTruthRVTrue}- \outputsRVTrue}
\newcommand{\errorOutOfPlaneVerb}{\outputsTruthRVTrue - \projImage{\outputsTruthRVTrue}}

\newcommand{\error}{\errorVerb}

\newcommand{\errorInPlane}{\errorInPlaneVerb}
\newcommand{\errorOutOfPlane}{\errorOutOfPlaneVerb}

\newcommand{\LinearTruthMat}{{\boldsymbol{\mathsf{A}}}}
\newcommand{\LinearTruthMati}{\LinearTruthMat_{i}}

\newcommand{\captionFirst}{}

\usepackage{enumitem}

\newlist{Assumption}{enumerate}{1}
\setlist[Assumption]{label=A\arabic*}

\newsiamremark{remark}{Remark}
\newsiamremark{hypothesis}{Hypothesis}
\crefname{hypothesis}{Hypothesis}{Hypotheses}
\newsiamthm{claim}{Claim}

\usepackage{xr}
\title{
The network uncertainty quantification method for\\ propagating uncertainties
in component-based systems
}

\author{
Kevin Carlberg\footnotemark[1] \and
     Sofia Guzzetti\footnotemark[2] \and Mohammad Khalil\footnotemark[1]
		 \and Khachik Sargsyan\footnotemark[1]
		 }

\begin{document}
\maketitle
\footnotetext[1]{Sandia National Laboratories, 7011 East Avenue, Livermore, CA
94550.  \email{kevintcarlberg@gmail.com}, \email{mkhalil@sandia.gov}, \email{ksargsy@sandia.gov}.}
\footnotetext[2]{
Emory University, 201 Dowman Dr, Atlanta, GA 30322.  \email{sofia.guzzetti@emory.edu}.}

\begin{abstract}
	This work introduces the network uncertainty quantification (NetUQ) method
	for performing uncertainty propagation in systems composed of interconnected
	components. The method assumes the existence of a collection of components,
	each of which is characterized by exogenous-input random variables (e.g.,
	material properties), endogenous-input random variables (e.g., boundary
	conditions defined by another component), output random variables (e.g.,
	quantities of interest), and a local
	uncertainty-propagation operator (e.g., provided by stochastic collocation)
	that computes output random variables from input random variables. The
	method assembles the full-system network by connecting components, which is
	achieved simply by associating endogenous-input random variables for each
	component with output random variables from other components; no other
	inter-component compatibility conditions are required.  The network
	uncertainty-propagation problem is: Compute output random variables for all
	components given all exogenous-input random variables. To solve this
	problem, the method applies classical relaxation methods (i.e., Jacobi and
	Gauss--Seidel iteration with Anderson acceleration), which require only
	black-box evaluations of component uncertainty-propagation operators.
	Compared with other available methods, this approach is applicable to any
	network topology (e.g., no restriction to feed-forward or two-component
	networks), promotes component independence by enabling components to employ
	tailored uncertainty-propagation operators, supports general functional
	representations of random variables, and requires no offline preprocessing
	stage. Also, because the method propagates functional representations of
	random variables throughout the network (and not, e.g., probability density
	functions), the joint distribution of any set of random variables
	throughout the network can be estimated \textit{a posteriori} in a
	straightforward manner. We perform supporting convergence and error analysis
	and execute numerical experiments that demonstrate the weak- and
	strong-scaling performance of the method.
\end{abstract}



\section{Introduction}

Many systems in science and engineering---ranging from power grids to gas
transfer systems to aircraft---comprise a collection of a large number of
interconnected components. Performing uncertainty propagation in these systems
is often essential for quantifying the influence of uncertainties on the
performance of the full system. Na\"ively, performing full-system uncertainty
propagation in this context requires (1) integrating all component models into
a single deterministic full-system model, and (2) executing uncertainty
propagation using the deterministic full-system model.  The first step is
often challenging or impossible, as components are often characterized by
vastly different physical phenomena and spatiotemporal scales, ensuring
compatibility between the meshes of neighboring components is difficult, and
components are often modeled using different simulation codes that can be
difficult to integrate. If the first step is achievable, then the second step
may still pose an insurmountable challenge due to the large-scale nature of
the deterministic full-system model.

To address these challenges, researchers have developed several classes of
methods that aim to decompose the full-system uncertainty-propagation problem
into tractable subproblems. Ideally, such a method should satisfy several
desiderata. First, the method should \textit{promote component independence}.
For example, each component should be able to use a tailored
uncertainty-propagation method, as it has been shown that different
uncertainty-propagation methods are better suited for different physics (see,
e.g., Ref.~\cite{constantine2009hybrid}); in addition, the approach should
avoid any coupled-component deterministic solves. Second, to be widely
applicable, the method should \textit{not be restricted to systems with
particular component-network topologies}, as systems often comprise complex
assemblies of components with two-way coupling between neighboring components.
Third, the method should characterize uncertainties using \textit{functional
representations of random variables}, and should not be restricted to
particular functional representations.  Characterizing random variables using
functional representations (i.e., as functions acting on the sample space)
allows for the joint distribution of any set of random variables throughout
the system to be estimated straightforwardly, and also facilitates other
post-processing activities (e.g., variance-based decomposition).  Supporting
general representations is also important, as techniques restricted to
polynomial-chaos representations \cite{ghanem2003stochastic,xiu2002wiener},
for example, are not compatible with other widely used random-variable
representations
\cite{wan2006multi,bilionis2012multi,bilionis2012multidimensional}.

The first class of methods replaces component high-fidelity models with
component surrogate models, and subsequently computes inexpensive full-system
samples either by propagating deterministic samples within a feed-forward
network \cite{martin2006methodology,sanson2018uncertainty}, or by using
domain-decomposition methods to converge these full-system samples
\cite{liao2015domain}. We point out that while these methods can significantly
reduce the computational cost of full-system uncertainty propagation, they do
not decompose the uncertainty-propagation task itself; thus, these methods do
not enable each component to use a tailored uncertainty-propagation method.
Relatedly, Ref.~\cite{friedman2018efficient} generates a surrogate model for
the coupling variables across the system to facilitate generating full-system
samples; however, this approach requires generating deterministic full-system
samples to construct the surrogate model.

The second class of methods considers coupled two-component systems and
restricts attention to
polynomial-chaos representations of random variables.  These techniques
focus primarily on reducing the dimensionality of the stochastic space
considered by each component;
Refs.~\cite{arnst2012measure,arnst2012dimension,arnst2014reduced} employ
Gauss--Seidel iteration to update the random variables for each component
(analogous to overlapping domain decomposition), while
Refs.~\cite{constantine2014efficient,mittalFlexible} propose intrusive coupled
formulations (analogous to non-overlapping domain decomposition).
Alternatively, Ref.\ \cite{chen2013flexible} considers linear feed-forward systems
and exploits the fact that---for such systems---the uncertainty-propagation
problems for different PCE coefficients decouple. We note that these
techniques do not consider the case where input random variables may be shared
across components, which may occur, for example if two components share a
boundary condition with uncertainty.

The third class of methods models a given component-based system as a network
of interconnected components, and propagates probability distributions through
the resulting network. The first contribution \cite{amaral2014decomposition}
restricted attention to feed-forward networks, while follow-on work
\cite{ghoreishi2016compositional} considered extensions to coupled
two-component networks.  Critically, because a random variable's marginal
distribution does not completely characterize its functional behavior, these
approaches encounter challenges for certain uncertainty-quantification tasks,
e.g., computing the joint distribution between random variables, compute
Sobol' indices for variance-based decomposition. However,
Ref.~\cite{amaral2014decomposition} provided a clever mechanism for recovering
joint distributions in the case of feed-forward networks.  A related approach
\cite{sankararaman2012likelihood} replaces coupling variables with conditional
probability densities and subsequently generates full-system samples by
propagating samples through the resulting feed-forward network; however, this
approach does not account for joint distributions between system inputs and
coupling variables.

The fourth class of methods stems from the field of multidisciplinary design
optimization, and is typically carried out in a reliability-analysis
context~\cite{yao2011}. Such methods are adopted from related deterministic
approaches via either matching
moments~\cite{gu2006implicit,du2005collaborative,chiralaksanakul2005first} or
reformulating reliability conditions as probabilistic constraints, i.e.,
enforcing interface-matching violation to occur with specified low
probability~\cite{kokkolaras2004design}. In both cases, these approaches
enable full-system uncertainty propagation by introducing first-order, local
perturbation assumptions with no support for general functional
representations, or by introducing strong decoupling assumptions that restrict
the set of supported network topologies~\cite{mahadevan2006}.


To satisfy the aforementioned desiderata and overcome some of the limitations
of previous contributions, this work proposes the network uncertainty
quantification (NetUQ) method, which is inspired by classical relaxation
techniques commonly employed in overlapping domain decomposition. The method simply
assumes the existence of a collection of components, each of which is
characterized by (1) \textit{exogenous-input random variables} (e.g., boundary
conditions, material properties, geometric variables), (2)
\textit{endogenous-input random variables} (e.g., boundary conditions imposed
by a neighboring component), (3) \textit{output random variables} (e.g.,
quantities of interest), and (4) an \textit{uncertainty-propagation
operator} that computes output random variables from input random variables
(e.g., non-intrusive spectral projection, stochastic Galerkin projection).  In
particular, different components can employ different functional
representations of random variables
and different uncertainty-propagation operators.  To construct
the network (i.e., a directed graph) from this
collection of components, the NetUQ method simply associates endogenous-input
random variables for each component with output random variables from other
components; the method makes no other requirements on inter-component
compatibility. This again promotes component independence, because, for
example, the computational meshes employed to discretize neighboring
components need not be perfectly matched. The resulting network
uncertainty-propagation problem becomes: Compute output random variables for
all components given all exogenous-input random variables.

To solve the network uncertainty-propagation problem, NetUQ applies the
classical relaxation techniques of Jacobi and Gauss--Seidel iteration.  Within
each Jacobi iteration, all network edges corresponding to endogenous-input
random variables are ``cut'' such that the endogenous-input random variables
are set to their values at the previous iteration; then, all components
perform local uncertainty propagation in an embarrassingly parallel manner to
compute their output random variables at the current iteration. Within each
Gauss--Seidel iteration, selected network edges are cut to create a
feed-forward network (i.e., a directed acyclic graph); then, the method
performs feed-forward uncertainty propagation to compute output random
variables at the current iteration. We equip each of these methods with
Anderson acceleration \cite{anderson1965iterative,fang2009two} to accelerate
convergence of the resulting fixed-point problem. We provide supporting error
analysis for the method, which quantifies the error incurred by the NetUQ
formulation for the case where the component uncertainty-propagation operators
comprise an approximation of underlying ``truth'' component
uncertainty-propagation operators; this arises, for example, in the case of
truncated polynomial-chaos representations.

The remainder of the paper is outlined as follows.  Section
\ref{sec:formulation} formulates the problem by first describing the local
component uncertainty-propagation problem (Section \ref{sec:localProb}) and
subsequently forming the network uncertainty-propagation problem by
associating endogenous-input random variables for each component with output
random variables from other components (Section \ref{sec:globalProblem}).
Section \ref{sec:relax} describes the two relaxation methods employed by NetUQ
to solve the network uncertainty-propagation problem: the Jacobi method
(Section \ref{sec:Jacobi}) and the Gauss--Seidel method (Section
\ref{sec:gaussseidel}). Subsequently, Section \ref{sec:conv} performs
convergence analysis (Section \ref{sec:convergence}) and introduces Anderson
acceleration as a mechanism to accelerate convergence of the fixed-point
iterations (Section \ref{sec:AA}). Section \ref{sec:error} performs error
analysis, first by deriving \textit{a priori} and \textit{a posteriori} error
bounds when the component uncertainty-propagation operators serve as
approximations of underlying ``truth'' uncertainty-propagation operators
(Section \ref{sec:errorBounds}), and second by performing error analysis in
the case where the component uncertainty-propagation operators restrict the
output random variables to reside in a subspace of a ``truth'' subspace
associated with the truth uncertainty-propagation operators (Section
\ref{sec:errorProj}). Next, Section \ref{sec:numericalExperiments} reports
numerical experiments that demonstrate the performance of NetUQ, namely its
ability (1) to accurately propagate uncertainties in large-scale networks
while promoting component independence, (2) to yield rapid convergence when
deployed with Anderson acceleration, and (3) to generate significant speedups
in both strong-scaling (Section \ref{sec:strong}) and weak-scaling (Section
\ref{sec:weak}) scenarios. Finally, Section \ref{sec:conclusions} concludes
the paper.

\section{Problem formulation}\label{sec:formulation}

We now formulate the problem of interest: uncertainty propagation in systems
composed of interconnected components. Section \ref{sec:localProb} provides
the formulation for uncertainty propagation at the component level, which
comprises the local problem, while Section \ref{sec:globalProblem} presents
the network uncertainty propagation problem, which comprises the global
problem.

\subsection{\textit{Local problem}: component uncertainty propagation}
\label{sec:localProb}

We begin by defining a probability space $\probabilitySpace$ and assuming that
the full system is composed of $\nsubsystems$ interconnected components, the $i$th of which
is characterized by
\begin{itemize}
\item
\textit{exogenous-input random variables}
$\inputsexoRVi:\samplespace\rightarrow\RR{\ninputsexoRVi}$, which correspond
to random variables that are input directly to the component (e.g., material
properties);
\item \textit{endogenous-input random variables}
	$\inputsendoRVi:\samplespace\rightarrow\RR{\ninputsendoRVi}$, which
		correspond to random variables that are input from a neighboring component
		(e.g., boundary conditions imposed by a neighboring component);
\item \textit{output random variables}
$\outputsRVi:\samplespace\rightarrow\RR{\noutputsRVi}$, which can correspond
to quantities of interest or random variables that associate with
		endogenous-input random variables for a neighboring component; and
	\item an
\textit{uncertainty-propagation operator}
\begin{equation} \label{eq:propagatorRViDef}
	\propagatorRVi:(\inputsendoRVi,\inputsexoRVi)\mapsto\outputsRVi
\end{equation}
with $\propagatorRVi:\RVspaceDim{\ninputsendoRVi}\times
\RVspaceDim{\ninputsexoRVi}\rightarrow
\RVspaceDim{\noutputsRVi}$, whose evaluation comprises the component
		uncertainty-propagation problem of computing output random variables from
		input random variables.
\end{itemize}
Here, $\RVspaceDim{}$ denotes the vector space of real-valued random variables
defined on $\probabilitySpace$, i.e., the set
of mappings from the sample space $\samplespace$ to the real numbers $\RR{}$.
We denote the norm on the vector space
$\RVspaceDim{}$ by
$\|\cdot\|$
(e.g., $\|\outputsRVDummyElement{}\| =
[\mathbb{E}[|\outputsRVDummyElement{}|^p]]^{1/p}$ for
$\outputsRVDummyElement{}\in\RVspaceDim{}$ and some $p$).
We also denote the norm on the vector space
$\RVspaceDim{N}$ by
$\|\cdot\|$
(e.g.,
$\|\outputsRVDummy\|=[\sum_{i=1}^N\|\outputsRVDummyElement{i}\|^{q}]^{1/{q}}$ for
$\outputsRVDummy\equiv[\outputsRVDummyElement{1}\ \cdots\
\outputsRVDummyElement{N}]^T\in\RVspaceDim{N}$ and
some $q$), where the use of the respective norms will be apparent from context.
We emphasize that input and output random variables are characterized from the
functional viewpoint, i.e., as functions acting on the sample space
$\samplespace$.

\begin{remark}[Polynomial-chaos expansions]\label{rem:PCE}
	Polynomial-chaos expansions
	(PCE)~\cite{ghanem2003stochastic,xiu2002wiener,xiu2003modeling} comprise a
	widely adopted approach for computing functional representations of random variables. We now describe
	how PCE representations correspond to the general framework of component
	uncertainty propagation outlined above.
	To apply PCE in this context, we
	assume the existence of
	independent ``germ'' random variables $\underlyingRV
	\equiv(\underlyingRVunivArg{1},\ldots,\underlyingRVunivArg{\stochasticDim})
	\in\RVspaceDim{\stochasticDim}$
	with known distribution, and subsequently express the input
	and output random variables as polynomial functions of
	these random variables, i.e.,
	\begin{equation} \label{eq:PCEexpansions}
	\inputsexoRVi
		=\sum_{\multiindex\in\multiindexSetArgs{\inputsexoRV}{i}}\PCEbasisArgs{\multiindex}{\underlyingRV}\inputsexoPCEcoeffsArg{i}{\multiindex},\quad
	\inputsendoRVi
	=\sum_{\multiindex\in\multiindexSetArgs{\inputsendoRV}{i}}\PCEbasisArgs{\multiindex}{\underlyingRV}\inputsendoPCEcoeffsArg{i}{\multiindex},\quad
	\outputsRVi
		=\sum_{\multiindex\in\multiindexSetArgs{\outputsRV}{i}}\PCEbasisArgs{\multiindex}{\underlyingRV}\outputsPCEcoeffsArg{i}{\multiindex},
\end{equation}
where
	$\multiindex\equiv(\multiindexArg{1},\ldots,\multiindexArg{\stochasticDim})$
	denotes a multi-index;
	$\multiindexSetArgs{\inputsexoRV}{i},
	\multiindexSetArgs{\inputsendoRV}{i},
\multiindexSetArgs{\outputsRV}{i}
	\subseteq\natnoZero{\stochasticDim}$
	denote multi-index sets
	 (e.g., defined using total-degree truncation
	such that $
	\multiindexSetArgs{\inputsexoRV}{i}=
	\multiindexSetArgs{\inputsendoRV}{i}=
\multiindexSetArgs{\outputsRV}{i}=
	\{\multiindex\in\natnoZero{\stochasticDim}\,|\,\|\multiindex\|_1\leq p\}$ for some
	$p$);
	$\inputsexoPCEcoeffsArg{i}{\multiindex}\in\RR{\ninputsexoRVi}$ for
	$\multiindex\in\multiindexSetArgs{\inputsexoRV}{i}$,
	$\inputsendoPCEcoeffsArg{i}{\multiindex}\in\RR{\ninputsendoRVi}$ for
	$\multiindex\in\multiindexSetArgs{\inputsendoRV}{i}$,
	and
	$\outputsPCEcoeffsArg{i}{\multiindex}\in\RR{\noutputsRVi}$ for
	$\multiindex\in\multiindexSetArgs{\outputsRV}{i}$ denote PCE
	coefficients; and the PCE basis functions
	$\PCEbasisArg{\multiindex}:\RR{\stochasticDim}\rightarrow\RR{}$
	are orthogonal with respect to the
	probability density function (PDF) of the germ random variables\footnote{For
	example, if the germ random variables
	$\underlyingRVunivArg{i}$, $i=1,\ldots,\stochasticDim$ are independent uniform random variables defined on the
	domain $[-1,
	1]$, the Legendre-polynomial PCE basis functions are employed, while if these
	random variables are independent
	standard normal random variables, Hermite-polynomial PCE basis
	functions are employed.}
	and are defined as a product of $\stochasticDim$ univariate
	polynomials such that
	\begin{equation}
		\PCEbasisArg{\multiindex}:\underlyingRV\mapsto\prod_{k=1}^{\stochasticDim}\PCEbasisUnivArg{\multiindexArg{k}}(\underlyingRVunivArg{k}),
	\end{equation}
	where $\PCEbasisUnivArg{i}:\RR{}\rightarrow\RR{}$ denotes a univariate
	polynomial of degree $i$.


According to the
Cameron--Martin theorem~\cite{cameron1947}, any finite-variance random
variable can be represented as a convergent expansion~\eqref{eq:PCEexpansions}
as the polynomial order and stochastic dimension $\stochasticDim$ grow, given
that the germ random variables $\underlyingRV$
comprise independent Gaussian random variables.
While this result has been conjectured for any general independent-component
germ~\cite{Xiu:2002cmame,arnst2010id}, the authors in
Ref.~\cite{ernst2012conv} prove a necessary and sufficient condition that the
underlying probability measure of the germ be uniquely determined by its
moments; this condition does not hold for a log-normal germ $\underlyingRV$,
for example.


	To apply PCE within a practical NetUQ setting, we first fix the germ random variables $\underlyingRV$
	(e.g., $\underlyingRV\sim\normal{\zero,\identity})$ and
	multi-index sets
	$\multiindexSetArgs{\inputsexoRV}{i}$,
	$\multiindexSetArgs{\inputsendoRV}{i}$, and
	$\multiindexSetArgs{\outputsRV}{i}$, and we subsequently determine the PCE basis functions
	$\PCEbasisArg{\multiindex}$, $\multiindex\in
	\multiindexSetArgs{\inputsexoRV}{i}\cup
	\multiindexSetArgs{\inputsendoRV}{i}\cup
	\multiindexSetArgs{\outputsRV}{i}$
	from the PDF of the germ random variables $\underlyingRV$.  With these
	quantities fixed, the random variables $\inputsexoRVi$, $\inputsendoRVi$,
	and $\outputsRVi$	are completely characterized via
	Eqs.~\eqref{eq:PCEexpansions} by their PCE coefficients
	$\inputsexoi\defeq(\inputsexoPCEcoeffsArg{i}{\multiindex})_{\multiindex\in\multiindexSetArgs{\inputsexoRV}{i}}
	\in\RR{\ninputsexoi} $, $\inputsendoi\defeq
	(\inputsendoPCEcoeffsArg{i}{\multiindex})_{\multiindex\in\multiindexSetArgs{\inputsexoRV}{i}}
	\in\RR{\ninputsendoi}$, and
	$\outputsi\equiv(\outputsPCEcoeffsArg{i}{\multiindex})_{\multiindex\in\multiindexSetArgs{\outputsRV}{i}}\in\RR{\noutputsi}$,
	respectively, where
	$\ninputsexoi \defeq \ninputsexoRVi\PCEdimArgs{\inputsexoRV}{i}$,
$\ninputsendoi \defeq \ninputsendoRVi\PCEdimArgs{\inputsendoRV}{i}$, and
	$\noutputsi \defeq \noutputsRVi\PCEdimArgs{\outputsRV}{i}$.

	Then, the uncertainty-propagation operator $\propagatorRVi$ associated with
	this PCE formulation effectively computes the mapping from the input PCE
	coefficients $\inputsexoi$, $\inputsendoi$ to the output PCE coefficients
	$\outputsi$.  Specifically, the discrete implementation of the
	uncertainty-propagation operator $\propagatorRVi$ in this case is a discrete
	uncertainty-propagation operator
\begin{equation} \label{eq:propagatoriDef}
	\propagatori:(\inputsendoi,\inputsexoi)\mapsto\outputsi
\end{equation}
with
	$\propagatori:\RR{\ninputsendoi\times\ninputsexoi}\rightarrow\RR{\noutputsi}$.

	The most common approaches for defining the discrete uncertainty-propagation
	operator $\propagatori$ are  {non-intrusive stochastic collocation}
	\cite{babuvska2007stochastic,nobile2008sparse},
	{non-intrusive spectral projection} \cite{lemaitre2010book,hosder2006non}, and {intrusive stochastic
	projection} (e.g., stochastic Galerkin
	\cite{deb2001solution,babuska2004galerkin,ghanem2003stochastic}, stochastic
	least-squares Petrov--Galerkin \cite{lee2018stochastic}).
\end{remark}

In the sequel, we present the methodology in the general setting of functional
representations of random variables. The discrete counterpart for PCE
representations (or for any other functional representation defined by a
finite number of deterministic scalar quantities) can be derived as above.

\subsection{\textit{Global problem}: network uncertainty
propagation}\label{sec:globalProblem}

We now describe the uncertainty-propagation problem for the full system
composed of $\nsubsystems$ interconnected components, each of which is
equipped with the component uncertainty-propagation formulation described in
Section \ref{sec:localProb}.
Denote by
$\inputsexoRV\defeq[\inputsexoRVArg{i}^T\ \cdots\
\inputsexoRVArg{\nsubsystems}^T]^T\in\RVspaceDim{\ninputsexoRV}$,
$\inputsendoRV\defeq[\inputsendoRVArg{i}^T\ \cdots\
\inputsendoRVArg{\nsubsystems}^T]^T\in\RVspaceDim{\ninputsendoRV}$, and
$\outputsRV\defeq[\outputsRVArg{i}^T\ \cdots\
\outputsRVArg{\nsubsystems}^T]^T\in\RVspaceDim{\noutputsRV}$
 the vectorization of the
component exogenous-inputs random variables, endogenous-input random
variables, and output random variables, respectively, such that
$\ninputsexoRV\defeq\sum_{i=1}^{\nsubsystems}\ninputsexoRVArg{i}$,
$\ninputsendoRV\defeq\sum_{i=1}^{\nsubsystems}\ninputsendoRVArg{i}$, and
$\noutputsRV\defeq\sum_{i=1}^{\nsubsystems}\noutputsRVArg{i}$. Then, the
full-system uncertainty-propagator operator is
\begin{equation} \label{eq:propagatorRV}
	\propagatorRVVec:(\inputsendoRV,\inputsexoRV)\mapsto\outputsRV,
\end{equation}
where $\propagatorRVVec:\RVspaceDim{\ninputsendoRV}\times
\RVspaceDim{\ninputsexoRV}\rightarrow
\RVspaceDim{\noutputsRV}$ comprises the vectorization of component
uncertainty-propagation operators such that
\begin{equation}
	\propagatorRVVec:(\inputsendoRV,\inputsexoRV) \mapsto
	\sum_{i=1}^\nsubsystems[\sampleoutputsRVArg{i}]^T
	\propagatorRVArg{i}(\sampleinputsendoRVArg{i}\inputsendoRV,
	\sampleinputsexoRVArg{i}
	\inputsexoRV).
\end{equation}
Here, $\sampleoutputsRVArg{i}\in\{0,1\}^{\noutputsRVArg{i}\times\noutputsRV}$,
$\sampleinputsendoRVArg{i}\in\{0,1\}^{\ninputsendoRVArg{i}\times\ninputsendoRV}$, and
$\sampleinputsexoRVArg{i}\in\{0,1\}^{\ninputsexoRVArg{i}\times\ninputsexoRV}$
denote selected rows of the identity matrix that extract quantities associated
with the $i$th component from the network such that
$
\sampleoutputsRVArg{i}\outputsRV\equiv\outputsRVArg{i}
	$,
$\sampleinputsendoRVArg{i}\inputsendoRV\equiv\inputsendoRVArg{i}$, and
$\sampleinputsexoRVArg{i}\inputsexoRV\equiv\inputsexoRVArg{i}$.

Critically, the full system is defined by connecting components such that each
component's endogenous inputs correspond to the outputs of another component.
We encode this relationship by the adjacency matrix
$\outputsToInputsRV\in\{0,1\}^{\ninputsendoRV\times\noutputsRV}$,
which satisfies the
relationship
\begin{equation} \label{eq:adjacencyRV}
\inputsendoRV \equiv \outputsToInputsRV\outputsRV
\end{equation}
such that $[\outputsToInputsRV]_{ij}$ is equal to one if
$[\inputsendoRV]_i\equiv[\outputsRV]_j$ and is zero otherwise.
Note that because
uncertainty propagation within each component is self contained, we do not
require self connections, and thus the diagonal blocks of the adjacency matrix
$\outputsToInputsRV$ are zero, where the $(i,j)$ block is a
$\ninputsendoRVi\times\noutputsRVi$ submatrix. We also admit the possibilities
that a single component output may not correspond to an endogenous input for
any other component (in
which case the associated adjacency-matrix column is zero), or may constitute
the endogenous input for multiple components (in which case the associated
adjacency-matrix column has more than one nonzero element).

Substituting Eq.~\eqref{eq:adjacencyRV} into the full-system
uncertainty-propagator \eqref{eq:propagatorRV} yields the following fixed-point problem:
Given exogenous-input random variables
$\inputsexoRV\in\RVspaceDim{\ninputsexoRV}$, compute output random variables
$\outputsRVTrue\equiv\outputsRVTrue(\inputsexoRV)\in\RVspaceDim{\noutputsRV}$
that satisfy
\begin{equation}\label{eq:globalResRV}
\resRV(\outputsRVTrue,\inputsexoRV) = \zero,
\end{equation}
where $\zero\in\RVspaceDim{\noutputsRV}$ denotes a vector of zero-valued random
variables and
\begin{equation}
\resRV:(\outputsRVDummy,\inputsexoRVDummy)\mapsto\outputsRVDummy -
\propagatorRVVec(\outputsToInputsRV\outputsRVDummy,\inputsexoRVDummy)
\end{equation}
with
$\resRV:\RVspaceDim{\noutputsRV\times\ninputsexoRV}\rightarrow\RVspaceDim{\noutputsRV}$
denoting the fixed-point residual.
To simplify notation, we introduce an alternative version of the full-system
uncertainty-propagation operator
\begin{equation}
	\propagatorRV:(\outputsRV,\inputsexoRV)\mapsto\propagatorRVVec(\outputsToInputsRV\outputsRV,\inputsexoRV),
\end{equation}
where
$\propagatorRV:\RVspaceDim{\noutputsRV}\times\RVspaceDim{\ninputsexoRV}\rightarrow\RVspaceDim{\noutputsRV}$.
Note that the fixed-point residual is equivalently defined as
$\resRV:(\outputsRVDummy,\inputsexoRVDummy)\mapsto\outputsRVDummy -
\propagatorRV(\outputsRVDummy,\inputsexoRVDummy)$.

We note that in many applications, computing specific quantities of interest
(QoI) comprises the ultimate goal of the analysis. In the present context, we
assign QoI random variables $\qoioutputsRV\in\RVspaceDim{\nqoioutputsRV}$ to
be a subset of the network outputs, i.e., there exists an extraction matrix
$\outputsToQoIRV\in\{0,1\}^{\nqoioutputsRV\times\noutputsRV}$ that satisfies
the relationship
\begin{equation} \label{eq:outputsToQoIRV}
	\qoioutputsRV \equiv \outputsToQoIRV\outputsRV
\end{equation}
such that $[\outputsToQoIRV]_{ij}$ is equal to one if
$[\qoioutputsRV]_i\equiv[\outputsRV]_j$ and is zero otherwise.

Figure \ref{fig:networkFormulation} provides a graphical depiction of the
NetUQ formulation, which can be interpreted as performing uncertainty
propagation in networks, wherein each node associates with a component
uncertainty-propagation operator $\propagatorRVi$, and each edge corresponds
to a collection of random variables.

\begin{figure}[htbp]
\centering
	\begin{subfigure}{0.45\textwidth}
		\centering
\includegraphics[width=0.4\textwidth]{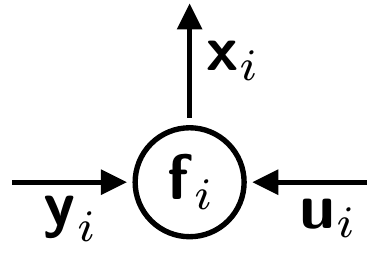}
		\caption{Component uncertainty propagation}\label{fig:subformulation}
	\end{subfigure}
	\begin{subfigure}{0.45\textwidth}
		\centering
\includegraphics[width=0.9\textwidth]{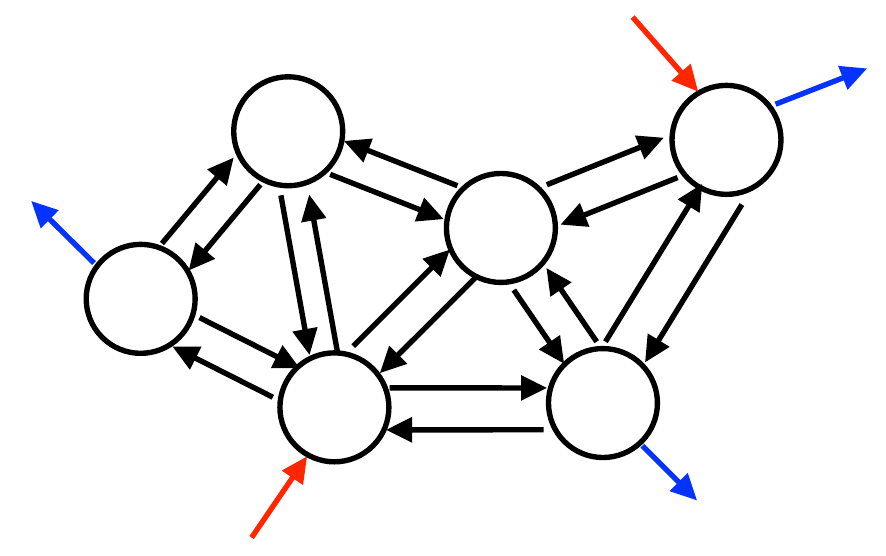}
		\caption{Network uncertainty propagation
	} \label{fig:netformulation}
	\end{subfigure}
	\label{fig:network}
	\caption{\textit{NetUQ formulation}. Figure shows how the network
	uncertainty-propagation problem can be derived from component
	uncertainty-propagation problems
	shown in Fig.\ \ref{fig:subformulation}. In Fig.\ \ref{fig:netformulation}, black edges correspond to endogenous-input random variables $\inputsendoRV$, red
	edges correspond to exogenous-input random variables $\inputsexoRV$, and blue edges
	correspond to QoI random variables $\qoioutputsRV$. Note that output random
	variables $\outputsRV$ are mapped to endogenous-input random
	variables $\inputsendoRV$ and/or $\qoioutputsRV$ via the operators
	$\outputsToInputsRV$ (see Eq.~\eqref{eq:adjacencyRV}) and $\outputsToQoIRV$
	(see Eq.~\eqref{eq:outputsToQoIRV}), respectively. We note that a network is
	defined from component uncertainty-propagation operators
	 $\propagatorRVArg{i}$, $i=1,\ldots,\nsubsystems$, the adjacency matrix
	 $\outputsToInputsRV$, and the extraction matrix $\outputsToQoIRV$.
	}
\label{fig:networkFormulation}
\end{figure}

\begin{remark}[Component independence]
This formulation promotes component independence, as the only formal
	inter-component compatibility requirement is the existence of an adjacency
	matrix $\outputsToInputsRV$ that associates component output random
	variables with endogenous input random variables of other components.  In
	particular, each set of random variables can be represented using different
	functional representations and components can employ completely different
	uncertainty-propagation operators.
\end{remark}

\section{Relaxation methods}
\label{sec:relax}

In principle, the fixed-point system \eqref{eq:globalResRV} can be solved
using a variety of techniques. While Newton's method is often employed for the
solution of systems of nonlinear algebraic equations due to its local
quadratic convergence rate, it relies on the ability to compute the gradient
$\partial \resRV/\partial\outputsRVDummy$. In the present context, this
requires computing the gradient of the output random variables with respect
to the endogenous-input random variables for each component, i.e.,
$\partial\propagatorRVi/\partial\inputsendoRVi$ must be computable for
$i=1,\ldots,\nsubsystems$.  This is frequently impractical, e.g., when the
simulation code used to compute a component uncertainty-propagation operator
is available only as a ``black box''.

As such, we proceed by assuming that only the component
uncertainty-propagation operators $\propagatorRVi$, $i=1,\ldots,\nsubsystems$
themselves are available, and consider classical relaxation methods (i.e.,
Jacobi, Gauss--Seidel) to solve the fixed-point problem \eqref{eq:globalResRV}, as these
methods do not require gradients and promote component independence. If each
node in the network corresponds to subdomain in a
partial-differential-equation problem, then this approach can be considered
an overlapping domain-decomposition strategy; however, the formulation does
not rely on any particular interpretation of the components.

\subsection{Jacobi method}\label{sec:Jacobi}

We first consider applying a variant of the Jacobi method (i.e., additive
Schwarz) to
solve fixed-point problem \eqref{eq:globalResRV}; Algorithm \ref{alg:jacobi} reports the
algorithm. At each iteration, this approach
performs independent, embarrassingly
parallel component uncertainty propagation (Steps
\ref{step:forbeginjacobi}--\ref{step:forendjacobi}), applies a relaxation
update (Step \ref{step:relaxationupdatejacobi}), and subsequently updates the endogenous-input
random variables from the output random variables just computed for
neighboring components (Step \ref{step:updateinputsendojacobi}).
From the network perspective, this approach
is equivalent to ``splitting'' all endogenous-input edges at each iteration and allowing
components to perform independent uncertainty
propagation; see Figure \ref{fig:jacobi}.

At the network level, Jacobi iteration $k$ can be expressed simply as
\begin{gather}\label{eq:jacobiFPone}
	\outputsPredictorRV =
	\propagatorRV(\iteration{\outputsRV}{k},\inputsexoRV)\\
	\label{eq:jacobiFPtwo}\iteration{\outputsRV}{k+1} =\relaxation\outputsPredictorRV +
(1-\relaxation)\iteration{\outputsRV}{k},
\end{gather}
where $\relaxation>0$ is a relaxation factor, which is set to $\relaxation=1$
for the classical Jacobi method. Under-relaxation corresponds to $\relaxation
< 0$, while over-relaxation corresponds to $\relaxation>1$.  This parameter
controls convergence of the algorithm as well as error analysis as will be
discussed in Remarks \ref{rem:controlConverge} and \ref{rem:satisfyAss}.

\begin{figure}[htbp]
\centering
	\begin{subfigure}{0.45\textwidth}
		\centering
\includegraphics[width=0.9\textwidth]{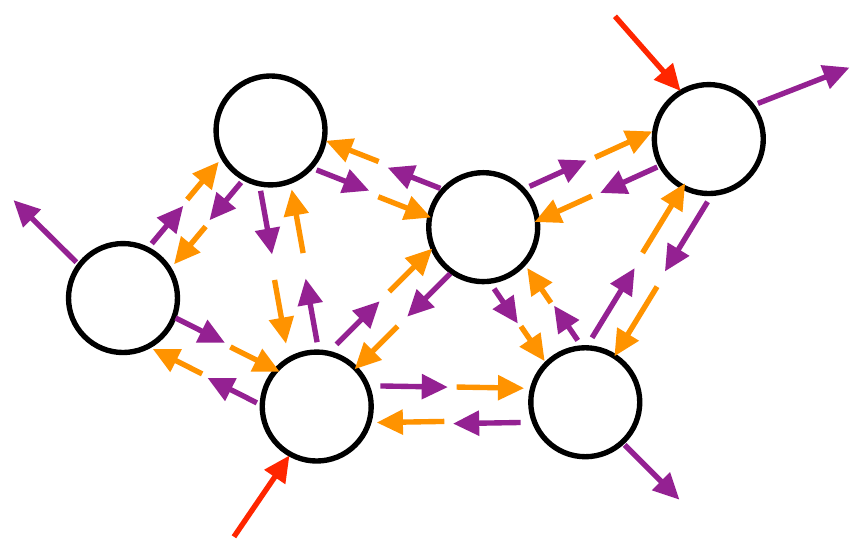}
		\caption{Jacobi method}\label{fig:jacobi}
	\end{subfigure}
	\begin{subfigure}{0.45\textwidth}
		\centering
\includegraphics[width=0.9\textwidth]{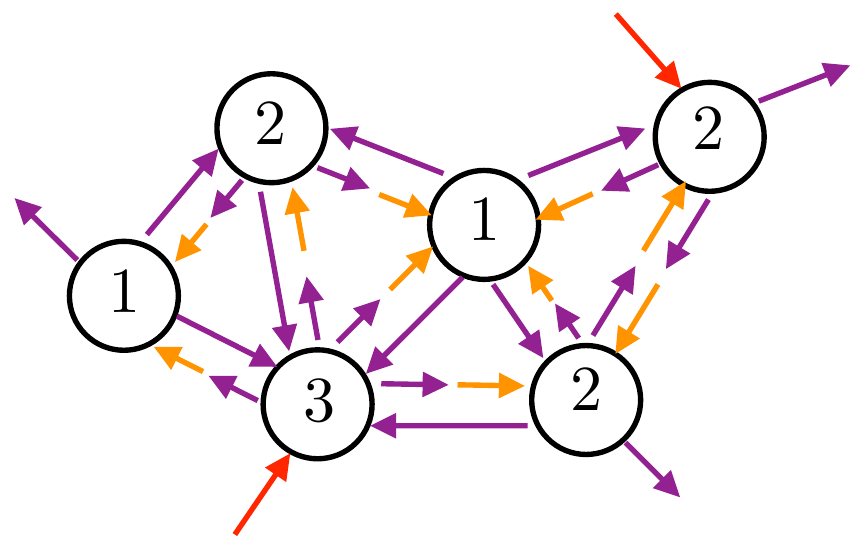}
		\caption{Gauss--Seidel method
	}\label{fig:gs}
	\end{subfigure}
	\caption{\textit{Relaxation methods}. Here, random variables that were
	computed at the previous iteration are shown in orange, random
	variables that are computed at the current iteration are shown in violet,
	and exogenous-input random variables---whose values are fixed---are
	shown in red.
	For the Gauss--Seidel method, the order in which a given node
	is processed is included in Figure \ref{fig:gs}.
	For the Jacobi method, note that all edges associated with
	endogenous-input random variables are ``split'' at each iteration to enable
	embarrassingly parallel component uncertainty propagation. On the other hand,
	for the Gauss--Seidel method, only a
	subset of such edges are ``split'' at each iteration such that a directed
	acyclic graph (DAG) is
	created. As a result, each iteration corresponds to uncertainty propagation
	in a feed-forward network. In this example, the chosen edge splitting yields
	a DAG that incurs three sequential steps per iteration. Note that many such
	DAGs exist, and each associates with a particular permutation of the
	adjacency matrix $\outputsToInputsRV$.}
\label{fig:relaxation}
\end{figure}

\begin{algorithm}[t]
	\caption{\Jacobialgname\ (The Jacobi method for NetUQ)}
	\label{alg:jacobi}
	\textbf{Input:} Component uncertainty-propagation operators
	$\propagatorRVArg{i}$, $i=1,\ldots,\nsubsystems$; adjacency matrix
	$\outputsToInputsRV$; exogenous-input random variables $\inputsexoRVi$,
	$i=1,\ldots,\nsubsystems$; initial guesses for endogenous-input random
	variables $\iteration{\inputsendoRVi}{0}$, $i=1,\ldots,\nsubsystems$;
	relaxation factor $\relaxation$\\
	\textbf{Output:} converged endogenous-input random variables $\iteration{\inputsendoRVi}{K}$;
	converged output random variables $\iteration{\outputsRVi}{K}$
	\begin{algorithmic}[1]
\STATE $k\leftarrow 0$
	\WHILE{not converged}
		\FOR[Execute in
		parallel]{$i=1,\ldots,\nsubsystems$}\label{step:forbeginjacobi}
\STATE $\outputsPredictorRVArg{i} =
	\propagatorRVArg{i}(\iteration{\inputsendoRVi}{k},\inputsexoRVi)$
		\ENDFOR\label{step:forendjacobi}
		\STATE\label{step:relaxationupdatejacobi} $\iteration{\outputsRV}{k+1} =\relaxation\outputsPredictorRV +
(1-\relaxation)\iteration{\outputsRV}{k}$
		\STATE\label{step:updateinputsendojacobi}
	$\iteration{\inputsendoRV}{k+1} =
	\outputsToInputsRV\iteration{\outputsRV}{k+1}$
	\STATE $k\leftarrow k+1$
	\ENDWHILE
	\STATE $K\leftarrow k$
\end{algorithmic}
\end{algorithm}

\subsection{Gauss--Seidel method}\label{sec:gaussseidel}

Naturally, we also consider a variant of the Gauss--Seidel method (i.e.,
multiplicative Schwarz) to solve
fixed-point problem \eqref{eq:globalResRV};
Algorithm \ref{alg:gaussseidel} reports the algorithm.
The benefit of the Gauss--Seidel method with respect to the Jacobi
method is that it enables more updated information to be used within each
iteration, at the expense of reduced parallelism.
To achieve this,
each Gauss--Seidel iteration performs feed-forward uncertainty propagation,
which requires ``splitting'' network edges in a manner that generates a directed
acyclic graph (DAG), and subsequently performing feed-forward uncertainty
propagation within the DAG; see Figure \ref{fig:gs}. In principle, each Gauss--Seidel
iteration can employ a different DAG; for simplicity in exposition, we
restrict consideration to a constant DAG.
\begin{algorithm}[t]
	\caption{\GaussseidelAlgName\ (The Gauss--Seidel method for NetUQ)}
	\label{alg:gaussseidel}
	\textbf{Input:} Component uncertainty-propagation operators
	$\propagatorRVArg{i}$, $i=1,\ldots,\nsubsystems$; adjacency matrix
	$\outputsToInputsRV$; exogenous-input random variables $\inputsexoRVi$,
	$i=1,\ldots,\nsubsystems$; initial guesses for endogenous-input random
	variables $\iteration{\inputsendoRVi}{0}$, $i=0,\ldots,\nsubsystems$;
	relaxation factor $\relaxation$; sequence of permutation matrices
	$\permutationTupleIt{k}$, $k=1,\ldots$\\
	\textbf{Output:} converged endogenous-input random variables $\iteration{\inputsendoRVi}{K}$;
	converged output random variables $\iteration{\outputsRVi}{K}$
	\begin{algorithmic}[1]
\STATE $k\leftarrow 0$
	\WHILE{not converged}
		\STATE $\permutationTuple\leftarrow\permutationTupleIt{k}$
		\STATE\label{step:DAG}$(\permutationRVMatrixArgs{\outputsRV}{\permutationTuple},
		\permutationRVMatrixArgs{\inputsendoRV}{\permutationTuple},\lowerBlockRVArg{\permutationTuple},\upperBlockRVArg{\permutationTuple},\indexSetArg{\permutationTuple}
		,\nsequentialArg{\permutationTuple}) = \DAGalgname(\permutationTuple,\outputsToInputsRV)
		$
		\STATE $\inputsendoPredictorRV\leftarrow\iteration{\inputsendoRV}{k}$
		\FOR[Execute sequentially]{$\ell=1,\ldots,\nsequentialArg{\permutationTuple}$}\label{step:forbegings}
		\FOR[Execute in parallel]{$i\in\indexSetArgs{\permutationTuple}{\ell}$}
		\STATE $\outputsPredictorRVArg{i} =
			\propagatorRVArg{i}(\inputsendoPredictorRVi,\inputsexoRVi)$
		\STATE $\inputsendoPredictorRV\leftarrow\inputsendoPredictorRV +
		\outputsToInputsRV
		[\sampleoutputsRVArg{i}]^T(\outputsPredictorRVArg{i}
		- \iteration{\outputsRVi}{k})
		$
		\ENDFOR
		\ENDFOR\label{step:forendgs}
		\STATE\label{step:relaxationupdategs} $\iteration{\outputsRV}{k+1} =\relaxation\outputsPredictorRV +
(1-\relaxation)\iteration{\outputsRV}{k}$
		\STATE\label{step:updateinputsendogs}
	$\iteration{\inputsendoRV}{k+1} =
	\outputsToInputsRV\iteration{\outputsRV}{k+1}$
	\STATE $k\leftarrow k+1$
	\ENDWHILE
	\STATE $K\leftarrow k$
\end{algorithmic}
\end{algorithm}

To generate a DAG, we introduce $\nsubsystems$-tuple
$\permutationTuple\equiv(\permutationArg{1},\ldots,\permutationArg{\nsubsystems})\in\natno{\nsubsystems}$
that provides a permutation of natural numbers one to $\nsubsystems$.
We then define the block permutation matrices
\begin{equation} \label{eq:permutationRVMatDef}
	\permutationRVMatrixArgs{\outputsRV}{\permutationTuple}\defeq
	\begin{bmatrix}
		\sampleoutputsRVArg{\permutationArg{1}}\\
		\vdots\\
		\sampleoutputsRVArg{\permutationArg{\nsubsystems}}
	\end{bmatrix},\quad
\permutationRVMatrixArgs{\inputsendoRV}{\permutationTuple}\defeq
	\begin{bmatrix}
		\sampleinputsendoRVArg{\permutationArg{1}}\\
		\vdots\\
		\sampleinputsendoRVArg{\permutationArg{\nsubsystems}}
	\end{bmatrix}
\end{equation}
and block decomposition of the permuted adjacency matrix
\begin{equation} \label{eq:adjacencyDecomp}
\permutationRVMatrixArgs{\inputsendoRV}{\permutationTuple}\outputsToInputsRV
[\permutationRVMatrixArgs{\outputsRV}{\permutationTuple}]^T =
\lowerBlockRVArg{\permutationTuple} +
\upperBlockRVArg{\permutationTuple},
\end{equation}
where
$\lowerBlockRVArg{\permutationTuple}\in\{0,1\}^{\ninputsendoRV\times\noutputsRV}$
is strictly block lower triangular
and
$\upperBlockRVArg{\permutationTuple}\in\{0,1\}^{\ninputsendoRV\times\noutputsRV}$
is strictly block upper triangular.

By ``splitting'' all edges in the adjacency matrix $\outputsToInputsRV$ whose
indices match the nonzero elements of
$[\permutationRVMatrixArgs{\inputsendoRV}{\permutationTuple}]^T\upperBlockRVArg{\permutationTuple}
\permutationRVMatrixArgs{\outputsRV}{\permutationTuple}
$, we can create a DAG, as the components
can be processed in sequential order $\permutationArg{i}$,
$i=1,\ldots,\nsubsystems$.
At the network level, Gauss--Seidel iteration $k$ can be expressed
as
\begin{gather}\label{eq:GSFPone}
	\outputsPredictorRV =
	\propagatorRVVec(
	[\permutationRVMatrixArgs{\inputsendoRV}{\permutationTuple}]^T
	\lowerBlockRVArg{\permutationTuple}\permutationRVMatrixArgs{\outputsRV}{\permutationTuple}
\outputsPredictorRV +
	[\permutationRVMatrixArgs{\inputsendoRV}{\permutationTuple}]^T
\upperBlockRVArg{\permutationTuple}
	\permutationRVMatrixArgs{\outputsRV}{\permutationTuple}
	\iteration{\outputsRV}{k},\inputsexoRV)\\
	\label{eq:GSFPTwo}\iteration{\outputsRV}{k+1} =\relaxation\outputsPredictorRV +
(1-\relaxation)\iteration{\outputsRV}{k},
\end{gather}
where again $\relaxation>0$ is a relaxation factor and
the implicit expression \eqref{eq:GSFPone} can be rewritten explicitly
via recursion as
\begin{gather}\label{eq:GSFPonerecurse}
	\outputsPredictorRVArg{\permutationArg{i}}
=
	\propagatorRVArg{\permutationArg{i}}
\left(
\sampleinputsendoRVArg{\permutationArg{i}}
\outputsToInputsRV
\sum_{j=1}^{i-1}
[\sampleoutputsRVArg{\permutationArg{j}}]^T
	\outputsPredictorRVArg{\permutationArg{j}}
+
\sampleinputsendoRVArg{\permutationArg{i}}
\outputsToInputsRV
\sum_{j=i+1}^{\nsubsystems}
[\sampleoutputsRVArg{\permutationArg{j}}]^T
\sampleoutputsRVArg{\permutationArg{j}}
\iteration{\outputsRV}{k}
,\inputsexoRV\right),\quad i=1,\ldots,\nsubsystems.
\end{gather}

By comparing with the Jacobi update
\eqref{eq:jacobiFPone}, it is apparent that
the Gauss--Seidel update to $\outputsPredictorRV$ in Eq.~\eqref{eq:GSFPone} is
implicit rather than explicit, and so uses more updated information than
Jacobi at each iteration; however, this is done at the expense of requiring
sequential processing of the components.

To mitigate the sequential-processing burden, we
observe that additional parallelism may be exposed by examining the sparsity
pattern of $\lowerBlockRVArg{\permutationTuple}$, as this matrix encodes
dependencies among the preserved edges in the (permuted) network.

Algorithm \ref{alg:DAG} describes this process. Given a permutation
tuple
$\permutationTuple$, this algorithm effectively ``splits'' all edges
associated with nonzero elements of $
[\permutationRVMatrixArgs{\inputsendoRV}{\permutationTuple}]^T\upperBlockRVArg{\permutationTuple}
\permutationRVMatrixArgs{\outputsRV}{\permutationTuple} $, and analyzes the
sparsity pattern of the block lower triangular matrix
$\lowerBlockRVArg{\permutationTuple}$ to determine (1) the minimum number of
sequential steps required to propagate uncertainties in the resulting DAG, and
(2) which components can be processed within each sequential step. This
procedure is called in Step \ref{step:DAG} of Algorithm \ref{alg:gaussseidel}
to ensure the Gauss--Seidel iterations employ the fewest number of
sequential steps.
Clearly, different
permutations
$\permutationTuple$ will associate with different
numbers of required sequential steps; graph coloring \cite{jensen2011graph}
provides a mechanism to determine the minimum number of sequential steps for a
given network. In addition to this consideration, different permutations
may lead to different convergence rates as will be discussed in
Remark \ref{rem:controlConverge}.

\begin{algorithm}[t]
	\caption{\DAGalgname\ (Generate a DAG from a permutation)}
	\label{alg:DAG}
	\textbf{Input:} Permutation
	$\permutationTuple\equiv(\permutationArg{1},\ldots,\permutationArg{\nsubsystems})$
	of the natural numbers one to $\nsubsystems$; adjacency matrix
	$\outputsToInputsRV$\\
	\textbf{Output:}
	Permutation matrices
$\permutationRVMatrixArgs{\outputsRV}{\permutationTuple}$ and
	$\permutationRVMatrixArgs{\inputsendoRV}{\permutationTuple}$; block lower
	and upper triangular matrices $\lowerBlockRVArg{\permutationTuple}$ and
	$\upperBlockRVArg{\permutationTuple}$;
	tuple of components in sequential processing order
	$\indexSetArg{\permutationTuple}\defeq(\indexSetArgs{\permutationTuple}{1},\ldots,\indexSetArgs{\permutationTuple}{\nsequentialArg{\permutationTuple}})$; number of
	sequential steps $\nsequentialArg{\permutationTuple}$
	\begin{algorithmic}[1]
		\STATE Compute
$\permutationRVMatrixArgs{\outputsRV}{\permutationTuple}$,
	$\permutationRVMatrixArgs{\inputsendoRV}{\permutationTuple}$,
	$\lowerBlockRVArg{\permutationTuple}$, and
		$\upperBlockRVArg{\permutationTuple}$ from
		Eqs.~\eqref{eq:permutationRVMatDef} and \eqref{eq:adjacencyDecomp}.
\STATE $k\leftarrow 0$; $\indexSet = \emptyset$
		\WHILE{$\indexSet\neq\{1,\ldots,\nsubsystems\}$}
		\STATE $k\leftarrow k+1$
		\STATE $\indexSetArgs{\permutationTuple}{k}\leftarrow \emptyset$
		\FOR{$i=1,\ldots,\nsubsystems$}
		\IF{block $(i,j)$ of $\lowerBlockRVArg{\permutationTuple}$ is zero for all
		$j\in\{1,\ldots,\nsubsystems\}\setminus\indexSet$}
		\STATE
		$\indexSetArgs{\permutationTuple}{k}\leftarrow\indexSetArgs{\permutationTuple}{k}\cup
		\permutationArg{i}$
		\ENDIF
		\ENDFOR
		\STATE $\indexSet\leftarrow\indexSet
		\cup\indexSetArgs{\permutationTuple}{k}$
	\ENDWHILE
		\STATE $\nsequentialArg{\permutationTuple}=  k$
\end{algorithmic}
\end{algorithm}

\section{Convergence analysis and Anderson acceleration}\label{sec:conv}

Both Jacobi and Gauss--Seidel methods can be expressed as fixed-point
iterations
\begin{equation} \label{eq:FPoperatorRV}
\iteration{\outputsRV}{k+1} =
\FPoperatorRV(\iteration{\outputsRV}{k},\inputsexoRVDummy)
\end{equation}
for $k=0,\ldots,K$,
where the mappings
$\outputsRVDummy\mapsto\outputsRV-\FPoperatorRV(\outputsRV,\inputsexoRVDummy)$ and
$\outputsRVDummy\mapsto
\outputsRVDummy -
\propagatorRV(\outputsRVDummy,\inputsexoRVDummy)$
have the same roots for any fixed value of the
exogenous-input random variables $\inputsexoRV\in\RVspaceDim{\ninputsexoRV}$.
In the case of Jacobi iteration, we see from
Eqs.~\eqref{eq:jacobiFPone}--\eqref{eq:jacobiFPtwo} that the
fixed-point operator is $\FPoperatorRV\leftarrow\FPoperatorRVJac$ with
\begin{equation}\label{eq:fpJac}
	\FPoperatorRVJac:(\outputsRVDummy,\inputsexoRVDummy;\relaxation)\mapsto
\relaxation\propagatorRV(\outputsRVDummy,\inputsexoRV)
	+ (1-\relaxation)\outputsRVDummy. 
\end{equation}
In the case of Gauss--Seidel iteration, we see from
Eqs.~\eqref{eq:GSFPone}--\eqref{eq:GSFPonerecurse} that the associated
operator is
$\FPoperatorRV\leftarrow\FPoperatorRVGS$, which is
defined recursively for $i=1,\ldots,\nsubsystems$ as
\begin{align}\label{eq:fpGS}
	\begin{split}
		\sampleoutputsRVArg{\permutationArg{i}}\FPoperatorRVGS:(\outputsRVDummy,\inputsexoRVDummy;\relaxation,\permutationTuple)\mapsto&
\relaxation
	\propagatorRVArg{\permutationArg{i}}
	\left(
\sampleinputsendoRVArg{\permutationArg{i}}
	\outputsToInputsRV
		\sum_{j=1}^{i-1}
	[\sampleoutputsRVArg{\permutationArg{j}}]^T
	\sampleoutputsRVArg{\permutationArg{j}}
\FPoperatorRVGS(\outputsRVDummy,\inputsexoRVDummy;\relaxation,\permutationTuple)
 +
\sampleinputsendoRVArg{\permutationArg{i}}
	\outputsToInputsRV
		\sum_{j=i+1}^{\nsubsystems}
	[\sampleoutputsRVArg{\permutationArg{j}}]^T
\sampleoutputsRVArg{\permutationArg{j}}
\outputsRVDummy
		,\inputsexoRV\right)\\
		&+ (1-\relaxation)\sampleoutputsRVArg{\permutationArg{i}}\outputsRVDummy,
	\end{split}
\end{align}
where we have made its dependence on the relaxation factor $\relaxation$ and
permutation $\permutationTuple$ explicit.

\subsection{Convergence analysis}\label{sec:convergence}

We now adopt a standard result from the theory of fixed-point iterations
to the present context. This result is valid for any (fixed) value of the
exogenous-input random variables $\inputsexoRV\in\RVspaceDim{\ninputsexoRV}$
and relies on the following assumption:
\begin{Assumption}[series=assumption]
\item \label{ass:lipschitz}
The fixed-point operator $\propagatorFPRV$ is a contraction, i.e., it is Lipschitz continuous such that
	\begin{equation}
	\norm{\propagatorFPRV(\outputsRVDummyArg{1},\inputsexoRVDummy)
-\propagatorFPRV(\outputsRVDummyArg{2},\inputsexoRVDummy)
	}\leq\lipschitzArg{\propagatorFPRV}\norm{\outputsRVDummyArg{1}-\outputsRVDummyArg{2}}
	\end{equation}
	for all
	$\outputsRVDummyArg{1},\outputsRVDummyArg{2}\in\RVspaceDim{\noutputsRV}$
	with $\lipschitzArg{\propagatorFPRV}<1$ for any $\inputsexoRV\in\RVspaceDim{\ninputsexoRV}$.
\end{Assumption}
\begin{proposition}\label{prop:converge}
	If Assumption  \ref{ass:lipschitz} holds,
	then the iterations
	\begin{equation}\label{eq:fixedPointIt}
	\iteration{\outputsRV}{k+1}=\FPoperatorRV(
	\iteration{\outputsRV}{k}, \inputsexoRV),\quad k=0,1,\ldots
\end{equation}
converge q-linearly to a fixed point $\outputsRVTrue$ satisfying
	$\outputsRVTrue = \FPoperatorRV(
	\outputsRVTrue, \inputsexoRV)$ with convergence rate
	$\lipschitzArg{\FPoperatorRV}$ for any
	$\inputsexoRV\in\RVspaceDim{\ninputsexoRV}$.
\end{proposition}
\begin{proof}
	Subtracting
	$\outputsRVTrue = \FPoperatorRV(
	\outputsRVTrue, \inputsexoRV)$
	from Eq.~\eqref{eq:fixedPointIt},
	and employing Lipschitz continuity of the fixed-point operator
	$\FPoperatorRV$, yields
	$
	\norm{\iteration{\outputsRV}{k+1} -
	\outputsRVTrue}\leq\lipschitzArg{\FPoperatorRV}\norm{
\iteration{\outputsRV}{k} -
	\outputsRVTrue}
$,
	which leads to the desired result.
\end{proof}

\begin{remark}[Controlling convergence]\label{rem:controlConverge}
	For a fixed network,
	the Jacobi fixed-point operator $\FPoperatorRVJac(\cdot,\cdot;\relaxation)$ is
	parameterized by the relaxation factor $\relaxation$, while the
	Gauss--Seidel fixed-point operator
	$\FPoperatorRVGS(\cdot,\cdot;\relaxation,\permutationTuple)$ is parameterized by
	both the relaxation factor $\relaxation$ and the permutation
	$\permutationTuple$. Thus, these parameters provide mechanisms for
	controlling convergence, as modifying them modifies the Lipschitz
	constants of these fixed-point operators, which in turn affects satisfaction of
	Assumption \ref{ass:lipschitz} and the convergence rate according to
	Proposition \ref{prop:converge}. The numerical experiments vary these
	quantities to numerically assess their effect on performance of NetUQ.
\end{remark}

\subsection{Anderson acceleration}\label{sec:AA}

Proposition \ref{prop:converge} showed that as long as the fixed-point
operator is a contraction, the fixed-point iterations generated by Jacobi or
Gauss--Seidel iterations converge q-linearly. Compared with other methods for
solving systems of nonlinear equations (e.g., Newton's method), this rate of
convergence is rarely competitive; as such, relaxation methods are used
typically only in cases where these types of solvers are not applicable (e.g.,
if the residual Jacobian cannot be exposed easily).

To improve the convergence rate of fixed-point iterations, researchers have
recently rediscovered \cite{fang2009two} Anderson acceleration
\cite{anderson1965iterative}, which provides a practical modification to the
fixed-point-iteration updates and has been shown to substantially improve
convergence. For linear problems, one can demonstrate that
convergence with Anderson acceleration is no slower than the original fixed-point iteration
\cite{toth2015convergence} and is essentially equivalent to
the generalized minimum residual (GMRES) method \cite{walker2011anderson}. For
nonlinear problems, it
can be shown that Anderson acceleration is a multisecant method
\cite{fang2009two}, and similar convergence-rate results exist
\cite{toth2015convergence}. All convergence results hold under the assumption
of a contractive mapping, i.e., under Assumption \ref{ass:lipschitz}.

Algorithm \ref{alg:AA} reports the Anderson acceleration method adopted to the
present context, where the fixed-point operator is either
$\FPoperatorRV\leftarrow\FPoperatorRVJac$ or
$\FPoperatorRV\leftarrow\FPoperatorRVGS$. We note that employing a memory of $\memory=0$ recovers the
original fixed-point iterations \eqref{eq:FPoperatorRV}.  Numerical
experiments performed in Section \ref{sec:numericalExperiments} illustrate the
ability of Anderson acceleration to substantially improve the convergence rate
of the proposed NetUQ method in the case of both Jacobi and Gauss--Seidel
iteration.

\begin{algorithm}[t]
	\caption{\AAalgname\ (Anderson acceleration)}
	\label{alg:AA}
	\textbf{Input:} Fixed-point operator $\FPoperatorRV$;
exogenous-input random variables $\inputsexoRV$;
	initial guess for output random variables
	$\iteration{\outputsRV}{0}$; memory
	$\memory\in\natnoZero{}$ \\
	\textbf{Output:}
	Converged solution $\iteration{\outputsRV}{K}$
	\begin{algorithmic}[1]
		\STATE $k\leftarrow 0 $
	\WHILE{not converged}
		\STATE $\iteration{\outputsRVTmp}{k+1} = \FPoperatorRV(\iteration{\outputsRV}{k};\inputsexoRV)$
		\STATE $\iteration{\AAres}{k} = \iteration{\outputsRVTmp}{k+1} - \iteration{\outputsRV}{k}$
		\STATE $\iteration{\memory}{k} = \min(\memory,k)$
		\STATE Compute
		$(\iteration{\alpha_0}{k},\ldots,\iteration{\alpha_{\iteration{\memory}{k}}}{k})$
		as the solution to
		\begin{align*}
		\begin{split}
			\underset{(\alpha_0,\ldots,\alpha_{\iteration{\memory}{k}})}{\text{minimize}}\
			 \|\sum_{i=0}^{\iteration{\memory}{k}}
			\iteration{\AAres}{k-i}
			\alpha_i \|_2\quad
			\text{subject to} &\sum_{i=0}^{\iteration{\memory}{k}}\alpha_i = 1
		\end{split}
		\end{align*}
		\STATE
		$\iteration{\outputsRV}{k+1}=\sum_{i=0}^{\iteration{\memory}{k}}\iteration{\alpha_i}{k}\iteration{\outputsRVTmp}{k+1-i}$
		\STATE $k\leftarrow k+1 $
	\ENDWHILE
		\STATE $K\leftarrow k$
\end{algorithmic}
\end{algorithm}

\section{Error analysis}\label{sec:error}

We now assess the error incurred by the network formulation when
the full-system uncertainty-propagation operator $\propagatorRV$ constitutes an
approximation of an underlying ``truth'' operator. In
particular, it is common practice for the component uncertainty-propagation
operators $\propagatorRVi$, $i=1,\ldots,\nsubsystems$ to comprise
approximations of underlying ``truth'' component
uncertainty-propagation operators
$\propagatorTruthRVi:\RVspaceDim{\ninputsendoRVi}\times
\RVspaceDim{\ninputsexoRVi}\rightarrow \RVspaceDim{\noutputsRVi}$,
$i=1,\ldots,\nsubsystems$.  For example, this occurs when
component uncertainty-propagation operators $\propagatorRVi$, $i=1,\ldots,\nsubsystems$ associate with computing
low-dimensional polynomial-chaos representation of the output random
variables $\outputsRVi$, $i=1,\ldots,\nsubsystems$, while
$\propagatorTruthRVi$, $i=1,\ldots,\nsubsystems$ might associate with
computing a high-dimensional polynomial-chaos representation of the output
random variables.


We begin by noting that for any function
$\propagatorFPRV:\RVspaceDim{\noutputsRV}\times\RVspaceDim{\ninputsexoRV}\rightarrow
\RVspaceDim{\noutputsRV}
$ for which the mappings
$\outputsRVDummy\mapsto\outputsRV-\FPoperatorRV(\outputsRV,\inputsexoRVDummy)$ and
$\outputsRVDummy\mapsto
\outputsRVDummy -
\propagatorRV(\outputsRVDummy,\inputsexoRVDummy)$
have the same roots for any $\inputsexoRV\in\RVspaceDim{\ninputsexoRV}$,
the solution $\outputsRVTrue$ to the fixed-point system
\eqref{eq:globalResRV} also satisfies the fixed-point system
\begin{equation} \label{eq:FP}
	\outputsRVTrue = \propagatorFPRV(\outputsRVTrue,\inputsexoRV)
\end{equation}
for any $\inputsexoRV\in\RVspaceDim{\ninputsexoRV}$.
For example,
$\propagatorFPRV(\cdot,\cdot)=\FPoperatorRVJac(\cdot,\cdot;\relaxation)$
and
$\propagatorFPRV(\cdot,\cdot)=\FPoperatorRVGS(\cdot,\cdot;\relaxation,\permutationTuple)$
are the Jacobi and Gauss--Seidel fixed-point operators,
respectively, which satisfy this property.

We also define the truth full-system uncertainty-propagator as
$\propagatorTruthRVVec:\RVspaceDim{\ninputsendoRV}\times
\RVspaceDim{\ninputsexoRV}\rightarrow
\RVspaceDim{\noutputsRV}$, which comprises the vectorization of truth component
uncertainty-propagation operators such that
\begin{equation}\label{eq:propagatorTruthRV}
	\propagatorTruthRVVec:(\inputsendoRV,\inputsexoRV) \mapsto
	\sum_{i=1}^\nsubsystems[\sampleoutputsRVArg{i}]^T
	\propagatorTruthRVArg{i}(\sampleinputsendoRVArg{i}\inputsendoRV,
	\sampleinputsexoRVArg{i}
	\inputsexoRV).
\end{equation}
Substituting the relationship between outputs and endogenous
inputs~\eqref{eq:adjacencyRV} into the truth full-system
uncertainty-propagator \eqref{eq:propagatorTruthRV} yields the following truth fixed-point problem:
Given exogenous-input random variables
$\inputsexoRV\in\RVspaceDim{\ninputsexoRV}$, compute truth output random variables
$\outputsTruthRVTrue\equiv\outputsTruthRVTrue(\inputsexoRV)\in\RVspaceDim{\noutputsRV}$
that satisfy
\begin{equation}\label{eq:globalResTruthRV}
\resTruthRV(\outputsTruthRVTrue,\inputsexoRV) = \zero,
\end{equation}
where
\begin{equation}
	\resTruthRV :(\outputsRVDummy,\inputsexoRVDummy)\mapsto\outputsRVDummy
- \propagatorTruthRVVec(\outputsToInputsRV\outputsRVDummy,\inputsexoRVDummy)
\end{equation}
with
$\resTruthRV:\RVspaceDim{\noutputsRV}\times\RVspaceDim{\ninputsexoRV}\rightarrow\RVspaceDim{\noutputsRV}$
denotes the truth fixed-point residual.
As before, to simplify notation, we introduce an alternative version of the
truth full-system uncertainty-propagation operator
\begin{equation}
	\propagatorTruthRV:(\outputsRV,\inputsexoRV)\mapsto\propagatorTruthRVVec(\outputsToInputsRV\outputsRV,\inputsexoRV),
\end{equation}
where
$\propagatorTruthRV:\RVspaceDim{\noutputsRV}\times\RVspaceDim{\ninputsexoRV}\rightarrow\RVspaceDim{\noutputsRV}$
Note that the truth fixed-point residual is equivalently defined as
$\resTruthRV:(\outputsRVDummy,\inputsexoRVDummy)\mapsto\outputsRVDummy -
\propagatorTruthRV(\outputsRVDummy,\inputsexoRVDummy)$.

As above, we note that for any function
$\propagatorFPTruthRV:\RVspaceDim{\noutputsRV}\times\RVspaceDim{\ninputsexoRV}\rightarrow
\RVspaceDim{\noutputsRV}
$ for which the mappings
$\outputsRVDummy\mapsto\outputsRV-\propagatorFPTruthRV(\outputsRV,\inputsexoRVDummy)$ and
$\outputsRVDummy\mapsto
\outputsRVDummy
- \propagatorTruthRV(\outputsRVDummy,\inputsexoRVDummy) $
have the same roots for any  $\inputsexoRV\in\RVspaceDim{\ninputsexoRV}$,
the solution $\outputsTruthRVTrue$ to the fixed-point system
\eqref{eq:globalResTruthRV} also satisfies the fixed-point system
\begin{equation}\label{eq:FPtrue}
	\outputsTruthRVTrue =
	\propagatorFPTruthRV(\outputsTruthRVTrue,\inputsexoRV).
\end{equation}
We refer to such an operator $\propagatorFPTruthRV$ as
the truth fixed-point operator.
For example, one could employ
$\propagatorFPTruthRV(\cdot,\cdot)=\FPoperatorTruthRVJac(\cdot,\cdot;\relaxationTruth)$
with $\FPoperatorTruthRVJac:
(\outputsRVDummy,\inputsexoRVDummy;\relaxationTruth)\mapsto
\relaxationTruth\propagatorRV(\outputsRVDummy,\inputsexoRV)
	+ (1-\relaxationTruth)\outputsRVDummy
$
as a Jacobi fixed-point operator parameterized by the relaxation
factor $\relaxationTruth>0$,
or
$\propagatorFPRV(\cdot,\cdot)=\FPoperatorTruthRVGS(\cdot,\cdot;\relaxationTruth,\permutationTupleTruth)$
as a Gauss--Seidel fixed-point operator defined analogously to
$\FPoperatorRVGS$ and parameterized by the
relaxation factor $\relaxationTruth$ and permutation
$\permutationTupleTruth\in\natno{\nsubsystems}$.

In the remainder of this section, we drop dependence of all quantities on the
exogenous-input random variables $\inputsexoRVDummy$ for notational
simplicity. The results can be interpreted as holding
for any (fixed) value of $\inputsexoRV\in\RVspaceDim{\ninputsexoRV}$.

\subsection{Error bounds}\label{sec:errorBounds}
We first derive \textit{a priori} and \textit{a posteriori} bounds for the
error $\norm{\outputsTruthRVTrue- \outputsRVTrue}$.
We begin by introducing the following assumption, which will be used to derive
the \textit{a posteriori} error bound:
\begin{Assumption}[resume=assumption]
\item\label{ass:lipschitzTruth} The truth fixed-point operator $\propagatorFPTruthRV$ is Lipschitz continuous such that
	$\norm{\propagatorFPTruthRV(\outputsRVDummyArg{1})
-\propagatorFPTruthRV(\outputsRVDummyArg{2})
	}\leq\lipschitz{\propagatorFPTruthRV}\norm{\outputsRVDummyArg{1}-\outputsRVDummyArg{2}}$
	for all
	$\outputsRVDummyArg{1},\outputsRVDummyArg{2}\in\RVspaceDim{\noutputsRV}$
	with $\lipschitz{\propagatorFPTruthRV}<1$.
\end{Assumption}
\begin{remark}[Practical satisfaction of Assumptions \ref{ass:lipschitz} and
	\ref{ass:lipschitzTruth}]\label{rem:satisfyAss}
	Analogously to Remark \ref{rem:controlConverge}, we note that the
	flexibility in the definitions of the fixed-point operators
	$\propagatorFPRV$ and $\propagatorFPTruthRV$ allows them to be defined such
	that
	Assumptions \ref{ass:lipschitz} and \ref{ass:lipschitzTruth} are satisfied.
	For example, if
	$\propagatorFPRV(\cdot,\cdot)=\FPoperatorRVJac(\cdot,\cdot;\relaxation)$ (resp.\
	$\propagatorFPTruthRV(\cdot,\cdot)=\FPoperatorTruthRVJac(\cdot,\cdot;\relaxationTruth)$), then the relaxation factor
	$\relaxation$ (resp.\ $\relaxationTruth$) can often be chosen to satisfy
	Assumption \ref{ass:lipschitz} (resp.\ \ref{ass:lipschitzTruth}).
	Alternatively,
	if
	$\propagatorFPRV(\cdot,\cdot)=\FPoperatorRVGS(\cdot,\cdot;\relaxation,\permutationTuple)$ (resp.\
	$\propagatorFPTruthRV(\cdot,\cdot)=\FPoperatorTruthRVGS(\cdot,\cdot;\relaxationTruth,\permutationTupleTruth)$), then the relaxation factor
	$\relaxation$ and permutation $\permutationTuple$ (resp.\ $\relaxationTruth$ and
	$\permutationTupleTruth$) can often be chosen to satisfy
	Assumption \ref{ass:lipschitz} (resp.\ \ref{ass:lipschitzTruth}).
	Indeed, ensuring that Assumptions \ref{ass:lipschitz} and
	\ref{ass:lipschitzTruth} are satisfied is required to ensure the associated
	fixed-point iterations converge according to Proposition
	\ref{prop:converge}.
\end{remark}

We proceed by deriving the error bounds, which rely on Assumptions
\ref{ass:lipschitz} and \ref{ass:lipschitzTruth}.
\begin{proposition}[\textit{A priori} error bound for $\norm{\outputsTruthRVTrue- \outputsRVTrue}$] \label{prop:apriori}
	If Assumption \ref{ass:lipschitz} holds,
	then
	\begin{equation} \label{eq:aprioribound}
	\norm{\outputsTruthRVTrue- \outputsRVTrue} \leq
	 \frac{1}{1-\lipschitz{\propagatorFPRV}}\norm{\outputsTruthRVTrue-\propagatorFPRV(\outputsTruthRVTrue)}.
\end{equation}
\end{proposition}
\begin{proof}
	We subtract Eq.~\eqref{eq:FP} from \eqref{eq:FPtrue} and add and subtract
	$\propagatorFPRV(\outputsTruthRVTrue)$ to obtain
\begin{equation}
	\outputsTruthRVTrue- \outputsRVTrue =
	\propagatorFPTruthRV(\outputsTruthRVTrue) - \propagatorFPRV(\outputsTruthRVTrue)
	+ \propagatorFPRV(\outputsTruthRVTrue) -\propagatorFPRV(\outputsRVTrue)
	.
\end{equation}
Applying the triangle inequality and Lipschitz continuity of
	$\propagatorFPRV$ gives
\begin{equation}
	\norm{\outputsTruthRVTrue- \outputsRVTrue} \leq
\norm{\propagatorFPTruthRV(\outputsTruthRVTrue) - \propagatorFPRV(\outputsTruthRVTrue)}
+
	\lipschitz{\propagatorFPRV}\norm{\outputsTruthRVTrue- \outputsRVTrue}.
\end{equation}
Finally, using $\lipschitz{\propagatorFPRV}<1$ and
	$\propagatorFPTruthRV(\outputsTruthRVTrue) = \outputsTruthRVTrue$ (from
	Eq.~\eqref{eq:FPtrue}) yields the desired result.
\end{proof}
We note that because
the mapping
$
\outputsRVDummy\mapsto\outputsRVDummy - \propagatorFPRV(\outputsRVDummy)
$
corresponds to the residual of the fixed-point system \eqref{eq:FP}
and the (generally unknown) truth solution $\outputsTruthRVTrue$ appears in the bound,
inequality \eqref{eq:aprioribound} can be considered a residual-based
\textit{a priori} error bound.

\begin{proposition}[\textit{A posteriori} error bound for $\norm{\outputsTruthRVTrue- \outputsRVTrue}$] \label{prop:apost}
	If Assumption \ref{ass:lipschitzTruth} holds, then
	\begin{equation} \label{eq:aposterioribound}
	\norm{\outputsTruthRVTrue- \outputsRVTrue} \leq
	 \frac{1}{1-\lipschitz{\propagatorFPTruthRV}}\norm{\outputsRVTrue-\propagatorFPTruthRV(\outputsRVTrue)}.
\end{equation}
\end{proposition}
\begin{proof}
	We subtract Eq.~\eqref{eq:FP} from \eqref{eq:FPtrue} and add and subtract
	$\propagatorFPTruthRV(\outputsRVTrue)$ to obtain
\begin{equation}
	\outputsTruthRVTrue- \outputsRVTrue =
	\propagatorFPTruthRV(\outputsTruthRVTrue) - \propagatorFPTruthRV(\outputsRVTrue)
	+ \propagatorFPTruthRV(\outputsRVTrue) -\propagatorFPRV(\outputsRVTrue)
	.
\end{equation}
As before, applying the triangle inequality and Lipschitz continuity of
	$\propagatorFPTruthRV$ gives
\begin{equation}
	\norm{\outputsTruthRVTrue- \outputsRVTrue} \leq
	\lipschitz{\propagatorFPTruthRV}\norm{\outputsTruthRVTrue- \outputsRVTrue}
	+ \norm{\propagatorFPTruthRV(\outputsRVTrue)
	-\propagatorFPRV(\outputsRVTrue)}
	.
\end{equation}
Now, using $\lipschitz{\propagatorFPTruthRV}<1$ and
	$\propagatorFPRV(\outputsRVTrue) = \outputsRVTrue$ (from
	Eq.~\eqref{eq:FP}) yields the desired result.
\end{proof}
We can interpret this result similarly to Proposition \ref{prop:apriori}.
In particular, because
the mapping
$
\outputsRVDummy\mapsto\outputsRVDummy - \propagatorFPTruthRV(\outputsRVDummy)
$ corresponds to the residual of the truth fixed-point system \eqref{eq:FPtrue}
and the computed solution $\outputsRVTrue$ appears in the bound,
inequality \eqref{eq:aposterioribound} can be considered a residual-based
\textit{a posteriori} error bound.
The right-hand side norm corresponds to the difference between
the computed solution $\outputsRVTrue$ and the result of applying one iteration with
the truth fixed-point operator $\propagatorFPTruthRV$ to the computed solution $\outputsRVTrue$.
Thus, the right-hand side can be practically computed (resp.\
bounded from above) if the
Lipschitz constant can be computed (resp.\ bounded from above) by applying the
truth fixed-point operator to the computed solution.

\subsection{In-plane error analysis}\label{sec:errorProj}

We now consider the particular case where the uncertainty-propagation operator
$\propagatorRV$ restricts solutions to lie in a subspace of the space
considered by the truth uncertainty-propagation operator $\propagatorTruthRV$.
Then, we perform analysis that quantifies the difference between the computed
solution $\outputsRVTrue$ and the orthogonal projection of the truth solution
$\outputsTruthRVTrue$ onto the considered subspace. This is often referred to
as the ``in-plane error'', as it represents how close the computed solution is
to the orthogonal projection of the truth solution onto the subspace (i.e., ``plane'') of
interest.

We begin by introducing the following assumption, which will be employed in
all results within Section \ref{sec:errorProj}:
\begin{Assumption}[resume=assumption]
\item\label{ass:nested}
The component uncertainty-propagation operators satisfy
	$\imagePropFArg{i}\subseteq\imageoutputsRVArg{i}$, while the truth
	component uncertainty-propagation operators satisfy
	$\imagePropFTruthArg{i}\subseteq\imageoutputsTruthRVArg{i}$ with
$\imageoutputsRVArg{i}$ and
$\imageoutputsTruthRVArg{i}$
	linear subspaces of $\RVspaceDim{\noutputsRVArg{i}}$ satisfying
$\imageoutputsRVArg{i}\subseteq\imageoutputsTruthRVArg{i}\subseteq\RVspaceDim{\noutputsRVArg{i}}$
for
$i=1,\ldots,\nsubsystems$.
\end{Assumption}

We note that under Assumption \ref{ass:nested}, it follows trivially that the
full-system uncertainty-propagation operator satisfies
$\imagePropF\subseteq\imageProp\equiv\imageoutputsRVArg{1}\times\cdots\times\imageoutputsRVArg{\nsubsystems}$,
while the truth full-system uncertainty-propagation operator satisfies
$\imagePropFTruth\subseteq\imagePropTruth\equiv\imageoutputsTruthRVArg{1}\times\cdots\times\imageoutputsTruthRVArg{\nsubsystems}
$ with $\imagePropF$ and $\imagePropFTruth$ linear subspaces of
$\RVspaceDim{\noutputsRV}$ satisfying
$\imageProp\subseteq\imagePropTruth\subseteq\RVspaceDim{\noutputsRV}$.

\begin{remark}[Polynomial-chaos expansions satisfy Assumption
	\ref{ass:nested}]\label{rem:PCEAss3}
	Assumption \ref{ass:nested} holds if $\propagatorRVi$,
	$i=1,\ldots,\nsubsystems$ associate with computing low-order
	polynomial-chaos outputs and $\propagatorTruthRVi$,
	$i=1,\ldots,\nsubsystems$ associate with computing high-order
	polynomial-chaos outputs.  In this case,
	$\imageoutputsRVi=
	\{
\sum_{\multiindex\in\multiindexSetLow}\PCEbasisArgs{\multiindex}{\underlyingRV}\outputsPCEcoeffsArg{i}{\multiindex}\,|\,\outputsPCEcoeffsArg{i}{\multiindex}\in\RR{\noutputsRVi},\
	\multiindex\in\multiindexSetLow
\}\subseteq
	\RVspaceDim{\noutputsRVi}
	$, $i=1,\ldots,\nsubsystems$  with
$\multiindexSetLow \defeq
	\{\multiindex\in\natnoZero{\stochasticDim}\,|\,\|\multiindex\|_1\leq
	p_\text{low}\}$,
	while
$\imageoutputsTruthRVi=
	 \{
\sum_{\multiindex\in\multiindexSetHigh}\PCEbasisArgs{\multiindex}{\underlyingRV}\outputsPCEcoeffsArg{i}{\multiindex}\,|\,\outputsPCEcoeffsArg{i}{\multiindex}\in\RR{\noutputsTruthRVi},\
	\multiindex\in\multiindexSetHigh
\}\subseteq
	\RVspaceDim{\noutputsTruthRVi}
	$, $i=1,\ldots,\nsubsystems$ with
$\multiindexSetHigh \defeq
	\{\multiindex\in\natnoZero{\stochasticDim}\,|\,\|\multiindex\|_1\leq
	p_\text{high}\}$,
and $p_\text{high}\geq p_\text{low}$.
The conditions
	of Assumption \ref{ass:nested} can be trivially verified in this case.
\label{rem:pcnested}
\end{remark}

Under Assumption \ref{ass:nested}, the
orthogonal projector $\projImageArg{i}$ onto $\imageoutputsRVArg{i}$, $i=1,\ldots,\nsubsystems$
is a linear operator and
satisfies
\begin{equation}
	\projImageArg{i}\outputsRVDummyArg{i}=
	\underset{\tilde\outputsRVDummy\in\imageoutputsRVArg{i}}{\arg\min}\|
	\outputsRVDummyArg{i}- \tilde \outputsRVDummy
	\|,
\end{equation}
while the
orthogonal projector $\projImage$ onto $\imageProp$ is a linear operator and
satisfies
\begin{equation}
	\projImage\outputsRVDummy=
	\underset{\tilde\outputsRVDummy\in\imageoutputsRV}{\arg\min}\|
	\outputsRVDummy- \tilde \outputsRVDummy
	\|
\end{equation}
with
$
	\projImage \equiv
	\sum_{i=1}^{\nsubsystems}[\sampleoutputsRVArg{\permutationArg{i}}]^T\projImageArg{i}\sampleoutputsRVArg{\permutationArg{i}}$.
In this case, the error can be decomposed into orthogonal \textit{in-plane}
and \textit{out-of-plane} components
\begin{equation}
	\error =
	\underbrace{\errorOutOfPlane}_\text{out-of-plane error} +
	\underbrace{\errorInPlane}_\text{in-plane error}
	\end{equation}
	with
$\norm{\error}^2 = \norm{\errorOutOfPlane}^2 + \norm{\errorInPlane}^2$;
the out-of-plane error satisfies $\errorOutOfPlane\in \imageOrthProp$, where
$\imageOrthProp$ denotes the orthogonal complement of $\imageProp$
	in $\RVspaceDim{\noutputsRV}$; and the in-plane error satisfies
	$\errorInPlane\in \imageProp$. Because the out-of-plane error is determined
	completely from the truth solution $\outputsTruthRVTrue$ and the subspace
$\imageProp$, the in-plane error is the error component that is informative
of the accuracy of the computed solution $\outputsRVTrue$ given the
truth solution $\outputsTruthRVTrue$ and the considered subspace
$\imageProp$.

Before stating error bounds, we first derive conditions under which the
in-plane error is zero, which implies that the computed
solution satisfies $\outputsRVTrue=\projImageNo\outputsTruthRVTrue$ and thus
can be considered the optimal approximation of $\outputsTruthRVTrue$ in
$\imageProp$.  This result relies on the following assumptions:
\begin{Assumption}[resume=assumption]
\item\label{ass:specialForm}
	Assumption \ref{ass:nested} holds and the component uncertainty-propagation operators satisfy
		\begin{equation}\label{eq:requiredFormArg}
			\propagatorTruthRVArg{i}(\sampleinputsendoRVArg{i}\outputsToInputsRV\outputsRVDummy) = \propagatorRVArg{i}(\sampleinputsendoRVArg{i}\outputsToInputsRV\projImage{\outputsRVDummy})
			+
			\propagatorTruthRVPerpArg{i}(\sampleinputsendoRVArg{i}\outputsToInputsRV\outputsRVDummy),\quad\forall
			\outputsRVDummy\in\RVspaceDim{\noutputsRV}
		\end{equation}
		for $i=1,\ldots,\nsubsystems$.
\item \label{ass:uniqueFP} There is a unique solution $\outputsRVTrue$
	to the fixed-point system \eqref{eq:globalResRV}.
\end{Assumption}
It is straightforward to show that Assumption~\ref{ass:specialForm} leads to a
truth full-system uncertainty-propagation operator of the form
		\begin{equation}\label{eq:requiredForm}
		\propagatorTruthRV:\outputsRVDummy\mapsto\propagatorRV(\projImage{\outputsRVDummy})
	+ \propagatorTruthRVPerp(\outputsRVDummy),
		\end{equation}
		where
\begin{equation}
\propagatorTruthRVPerp:\outputsRV \mapsto
	\sum_{i=1}^\nsubsystems[\sampleoutputsRVArg{i}]^T
	\propagatorTruthRVPerpArg{i}(\sampleinputsendoRVArg{i}\outputsToInputsRV\outputsRVDummy)
\end{equation}
		with $\propagatorTruthRVPerp:\RVspaceDim{\noutputsRV}\rightarrow\RVspaceDim{\noutputsRV}$
	 and
	$\image{\propagatorTruthRVPerp}
	\subseteq
	\imageOrthProp\equiv\imageOrthPropArg{1}\times\cdots\times\imageOrthPropArg{\nsubsystems}$.

Note that Assumption \ref{ass:specialForm} is equivalent to assuming the
component uncertainty-propagation operators to satisfy
		$
			\projImageArg{i}\propagatorTruthRVArg{i}(\sampleinputsendoRVArg{i}\outputsToInputsRV\outputsRVDummy)
			=
			\propagatorRVArg{i}(\sampleinputsendoRVArg{i}\outputsToInputsRV\projImage{\outputsRVDummy})$,
			$\forall
			\outputsRVDummy\in\RVspaceDim{\noutputsRV}$,
which can be obtained
by premultiplying Eq.~\eqref{eq:requiredFormArg} by $\projImageArg{i}$.


\begin{remark}[Polynomial-chaos expansions satisfy Assumption
	\ref{ass:specialForm} for linear problems]\label{rem:PCELHSzero} Assumption
	\ref{ass:specialForm} holds if $\propagatorRVi$, $i=1,\ldots,\nsubsystems$
	are linear operators that project the outputs onto the space spanned by a
	low-order PCE basis and $\propagatorTruthRVi$, $i=1,\ldots,\nsubsystems$ are
	the same linear operators, but project the outputs onto the space spanned by
	a high-order PCE basis.

To demonstrate this, we employ the second moment as the norm-squared on the vector space
$\RVspaceDim{}$, i.e.,
	$\|\generalRV\| \defeq \sqrt{\mathbb{E}[\generalRV^2]} \equiv
		\sqrt{\int_\samplespace \generalRV^2 \mathrm{d}\probabilityOp(\theta)}
		$.
	Because any finite-variance random variable $\generalRV\in\RVspaceDim{}$ can
	be represented as a convergent PCE expansion, we can express such random
	variables as
	$\generalRV=\sum_{\multiindex\in\multiindexSetInf}\PCEbasisArgs{\multiindex}{\underlyingRV}\generalPCECoeffsArg{\multiindex}
	$
with 
$
\generalPCECoeffsArg{\multiindex} =
	\mathbb{E}[\generalRV\PCEbasisArg{\multiindex}]/\|\PCEbasisArg{\multiindex}\|^2
$, in which case the second moment is
	merely the weighted Euclidean norm of the PCE coefficients, i.e.,
	$\|\generalRV\| =
	\sqrt{{\sum_{\multiindex\in\multiindexSetInf}\generalPCECoeffsArg{\multiindex}^2}
	||\PCEbasisArg{\multiindex}||^2}$.

	In this case, the output random variables associated with the
	uncertainty-propagation operators $\propagatorRVi$ and $\propagatorTruthRVi$
	are
	\begin{equation}
		\outputsRVi
	=
		\sum_{\multiindex\in\multiindexSetLow}\PCEbasisArgs{\multiindex}{\underlyingRV}\outputsPCEcoeffsArg{i}{\multiindex}\qquad\text{and}\qquad
\outputsTruthRVi
	=\sum_{\multiindex\in\multiindexSetHigh}\PCEbasisArgs{\multiindex}{\underlyingRV}\outputsTruthPCEcoeffsArg{i}{\multiindex},
	\end{equation}
	respectively,
	for $i=1,\ldots,\nsubsystems$,
	with $\multiindexSetLow$ and $\multiindexSetHigh$ defined in Remark
	\ref{rem:PCEAss3}. Note in particular that we assume the same total-degree
	truncation is employed for the output random variables of all components. Given the choice the norm,
	any vector of finite-variance random variables
	$\generalRVVec\in\RVspaceDim{\noutputsRVi}$ can be expressed
	as
$\generalRVVec=\sum_{\multiindex\in\multiindexSetInf}\PCEbasisArgs{\multiindex}{\underlyingRV}\generalVecPCECoeffsArg{\multiindex}
	$ with its projection satisfying
\begin{equation}
\projImageArg{i}\generalRVVec
	=\sum_{\multiindex\in\multiindexSetLow}\PCEbasisArgs{\multiindex}{\underlyingRV}\generalVecPCECoeffsArg{\multiindex},\qquad
\projImageTruthArg{i}\generalRVVec
	=\sum_{\multiindex\in\multiindexSetHigh}\PCEbasisArgs{\multiindex}{\underlyingRV}\generalVecPCECoeffsArg{\multiindex}
\end{equation}
for $i=1,\ldots,\nsubsystems$,
where we have introduced
the
orthogonal projector $\projImageTruthArg{i}$ onto $\imageoutputsTruthRVArg{i}$, $i=1,\ldots,\nsubsystems$
as the linear operator satisfying
\begin{equation}
	\projImageTruthArg{i}\outputsRVDummyArg{i}=
	\underset{\tilde\outputsRVDummy\in\imageoutputsTruthRVArg{i}}{\arg\min}\|
	\outputsRVDummyArg{i}- \tilde \outputsRVDummy
	\|.
\end{equation}

Thus, in this case, assuming the endogenous-input random variables
	$\inputsendoRVArg{i}$ have finite variance such that
	$\inputsendoRVArg{i}=\sum_{\multiindex\in\multiindexSetInf}\PCEbasisArgs{\multiindex}{\underlyingRV}\inputsendoPCEcoeffsArg{i}{\multiindex}
	$, we have
	\begin{align}
		\propagatorRVArg{i} (\inputsendoRVArg{i})&=
		\projImageArg{i}
		\LinearTruthMati
		\inputsendoRVArg{i} =
\projImageArg{i}
\sum_{\multiindex\in\multiindexSetInf}\PCEbasisArgs{\multiindex}{\underlyingRV}
		\LinearTruthMati
		\inputsendoPCEcoeffsArg{i}{\multiindex}
		=
\sum_{\multiindex\in\multiindexSetLow}\PCEbasisArgs{\multiindex}{\underlyingRV}
		\LinearTruthMati
		\inputsendoPCEcoeffsArg{i}{\multiindex}
= \propagatorRVArg{i} (\sampleinputsendoRVArg{i}\outputsToInputsRV\projImage{\outputsRVDummy})
		\\
	\begin{split}
		\propagatorTruthRVArg{i} (\inputsendoRVArg{i})&=
		\projImageTruthArg{i}
		\LinearTruthMati
		\inputsendoRVArg{i} =
\projImageTruthArg{i}
\sum_{\multiindex\in\multiindexSetInf}\PCEbasisArgs{\multiindex}{\underlyingRV}
		\LinearTruthMati
		\inputsendoPCEcoeffsArg{i}{\multiindex}
		=
\sum_{\multiindex\in\multiindexSetHigh}\PCEbasisArgs{\multiindex}{\underlyingRV}
		\LinearTruthMati
		\inputsendoPCEcoeffsArg{i}{\multiindex}\\
		 &=
		\underbrace{\sum_{\multiindex\in\multiindexSetLow}\PCEbasisArgs{\multiindex}{\underlyingRV}
		\LinearTruthMati
		\inputsendoPCEcoeffsArg{i}{\multiindex}}_{\propagatorRVArg{i} (\sampleinputsendoRVArg{i}\outputsToInputsRV\projImage{\outputsRVDummy})}
		 +
		\underbrace{\sum_{\multiindex\in\multiindexSetHigh\setminus\multiindexSetLow}\PCEbasisArgs{\multiindex}{\underlyingRV}
		\LinearTruthMati
		\inputsendoPCEcoeffsArg{i}{\multiindex}}_{\propagatorTruthRVPerpArg{i}(\inputsendoRVi)},
	\end{split}
	\end{align}
where $\LinearTruthMati \in\RR{\noutputsRVArg{i}\times\ninputsendoRVArg{i}}$,
$i=1,\ldots,\nsubsystems$ denote matrices defining the linear
uncertainty-propagation operators.

This decoupling implies that higher-order terms of $\outputsRVDummyArg{i}$ do
not affect lower-order terms of
$\propagatorTruthRVArg{i}(\outputsRVDummyArg{i})$.  We note that this is
similar to \cite[Theorem 1]{chen2013flexible}.
\end{remark}

We now derive conditions under which the in-plane error is zero.
\begin{proposition}[Conditions for zero in-plane error]\label{cor:LHSzero}
	If Assumptions \ref{ass:specialForm} and \ref{ass:uniqueFP} hold, then
	\begin{equation} \label{eq:desiredZeroError}
\outputsRVTrue=\projImage{\outputsTruthRVTrue},
	  \end{equation}
		and thus the in-plane error is zero.
\end{proposition}
\begin{proof}
	Substituting \eqref{eq:requiredForm}, which derives from Assumption
	\ref{ass:specialForm}, in
	the fixed-point system \eqref{eq:globalResTruthRV} yields
\begin{equation}\label{eq:globalResTruthRVrequiredForm}
	\outputsTruthRVTrue =
\propagatorRV(\projImage{\outputsTruthRVTrue})
	+ \propagatorTruthRVPerp(\outputsTruthRVTrue).
\end{equation}
	Applying the (linear) operator $\projImageNo$ to both sides
	of Eq.~\eqref{eq:globalResTruthRVrequiredForm} yields
\begin{equation}\label{eq:FPtrueRequiredFormProj}
	\projImage{\outputsTruthRVTrue} =
	\propagatorRV(\projImage{\outputsTruthRVTrue}),
\end{equation}
where we have used
	$\image{\propagatorTruthRVPerp} \subseteq
	\imageOrthProp$. Eq.~\eqref{eq:FPtrueRequiredFormProj} shows that
	$\projImage{\outputsTruthRVTrue} $ satisfies
$\resRV(\projImage{\outputsTruthRVTrue} ) = \zero$, and thus
	the desired result holds
	by Assumption \ref{ass:uniqueFP}.
\end{proof}

We now introduce
an \textit{a priori} error bound, whose proof follows similar steps to
that of Proposition \ref{prop:apriori}.
\begin{proposition}[\textit{A priori} in-plane error
	bound]\label{prop:alternativeApriori}
	If Assumptions \ref{ass:lipschitz} and \ref{ass:nested} hold, then the
	in-plane error can be bounded as
	\begin{equation} \label{eq:aprioriboundTwo}
		\norm{\projImage{\outputsTruthRVTrue}- \outputsRVTrue} \leq
		\frac{1}{1-\lipschitz{\propagatorFPRV}}\norm{\projImage{\outputsTruthRVTrue}-\propagatorFPRV(\projImage{\outputsTruthRVTrue})}.
\end{equation}
\end{proposition}
\begin{proof}
	We subtract Eq.~\eqref{eq:FP} from
	$\projImagePar{\outputsTruthRVTrue -
	\propagatorFPTruthRV(\outputsTruthRVTrue)} = \zero
	$ (which holds from Eq.~\eqref{eq:FPtrue} and linearity of
	$\projImageNo$) and add and subtract
	$\propagatorFPRV(\projImage{\outputsTruthRVTrue})$ to obtain
\begin{equation}
	\projImage{\outputsTruthRVTrue}- \outputsRVTrue =
	\projImage{\propagatorFPTruthRV(\outputsTruthRVTrue)} -
	\propagatorFPRV(\projImage{\outputsTruthRVTrue})
	+ \propagatorFPRV(\projImage{\outputsTruthRVTrue}) -\propagatorFPRV(\outputsRVTrue)
	.
\end{equation}
Applying the triangle inequality and Lipschitz continuity of
	$\propagatorFPRV$ gives
\begin{equation}
	\norm{\projImage{\outputsTruthRVTrue}- \outputsRVTrue} \leq
\norm{\projImage{\propagatorFPTruthRV(\outputsTruthRVTrue)} -
	\propagatorFPRV(\projImage{\outputsTruthRVTrue})}
	+
	\lipschitz{\propagatorFPRV}\norm{\projImage{\outputsTruthRVTrue}- \outputsRVTrue}
	.
\end{equation}
Finally, using $\lipschitz{\propagatorFPRV}<1$ and
	$\propagatorFPTruthRV(\outputsTruthRVTrue) = \outputsTruthRVTrue$ (from
	Eq.~\eqref{eq:FPtrue}) yields the desired result.
\end{proof}
We note that this bound is identical to the bound in Proposition
\ref{prop:apriori}, with $\projImage{\outputsTruthRVTrue}$ replacing
$\outputsTruthRVTrue$; this result was achievable through the introduction of
Assumption \ref{ass:nested}.


We now introduce the following assumption, which will be used to derive an
\textit{a posteriori} bound:
\begin{Assumption}[resume=assumption]
\item\label{ass:lipschitzTruthRestrict}
	Assumption \ref{ass:nested} holds and
	the projected truth fixed-point operator satisfies
\begin{equation} \label{eq:lipschitzTruthRestrict}
	\|\projImage{\propagatorFPTruthRV(\outputsRVDummyArg{1})}-
	\projImage{\propagatorFPTruthRV(\outputsRVDummyArg{2})}\|\leq
	\lipschitz{\projImage{\propagatorFPTruthRV}}\|\projImage{\outputsRVDummyArg{1}}-\projImage{\outputsRVDummyArg{2}}\|,\quad \forall
	\outputsRVDummyArg{1},\outputsRVDummyArg{2}\in\RVspaceDim{\noutputsRV}
\end{equation}
	with $\lipschitz{\projImage{\propagatorFPTruthRV}} < 1$.
\end{Assumption}

We now introduce the notion of an $\imageOrthProp$-invariant operator, which
will be used to demonstrate conditions under which Assumption \ref{ass:lipschitzTruthRestrict} holds.
\begin{definition}[$\imageOrthProp$-invariant truth fixed-point
	operators]\label{def:inv}
We deem a truth fixed-point operator
$\propagatorFPTruthRV$ to be
$\imageOrthProp$-invariant  if it satisfies
	\begin{equation} \label{eq:requirement}
		\projImage{\propagatorFPTruthRV(\outputsRVDummy)} =
		\constant,\quad\forall\outputsRVDummy\in\imageOrthProp,
\end{equation}
	where $\constant\in\RVspaceDim{\noutputsRV}$ is a constant random
	vector.
\end{definition}
Intuitively, $\imageOrthProp$-invariant truth fixed-point operators preserve
the decomposition of $\RVspaceDim{\noutputsRV}$  into
$\imageProp$ and $\imageOrthProp$, as the output components of these operators in the
space $\imageProp$ are unaffected by input components in the space
$\imageOrthProp$.  $\imageOrthProp$-invariance is a more general notion than
that introduced in Assumption \ref{ass:specialForm}. We now demonstrate that
this type of operator is necessary to satisfy Assumptions
\ref{ass:specialForm} and \ref{ass:lipschitzTruthRestrict}.

\begin{proposition}[$\imageOrthProp$-invariance is a necessary condition for
	Assumptions \ref{ass:specialForm} and \ref{ass:lipschitzTruthRestrict}]
	\label{prop:inv}
	If (1) Assumption \ref{ass:specialForm} holds and either Jacobi or
	Gauss--Seidel iteration defines the truth fixed-point operator
	$\propagatorFPTruthRV$, or (2) Assumption \ref{ass:lipschitzTruthRestrict}
	holds, then $\propagatorFPTruthRV$ is an $\imageOrthProp$-invariant
	operator.
\end{proposition}
\begin{proof}
	\textbf{Case 1}.
If Assumption \ref{ass:specialForm} holds and Jacobi iteration is applied, then
		from definition \eqref{eq:fpJac}, the truth fixed-point operator becomes
	\begin{equation}\label{eq:propFormInvariant}
		\propagatorFPTruthRV:\outputsRVDummy\mapsto\underbrace{\relaxation\propagatorRV(\projImage\outputsRVDummy)
		+ (1-\relaxation)\projImage\outputsRVDummy
		}_{\in\imageProp} +
		\underbrace{\relaxation\propagatorTruthRVPerp(\outputsRVDummy) +
		(1-\relaxation)\projOrthImage\outputsRVDummy}_{\in\imageOrthProp},
	\end{equation}
where $\projOrthImageNo=\identityRV -
		\projImage{}$,
	from which condition \eqref{eq:requirement} can be trivially verified with
$\constant=
	\relaxation\propagatorRV(\zero)$.

		Alternatively, if Assumption \ref{ass:specialForm} holds and Gauss--Seidel
		iteration is applied, then
		from definition \eqref{eq:fpGS},
		the truth fixed-point operator becomes
 \begin{align} \label{eq:propFormInvariantGS}
 \begin{split}
	\sampleoutputsRVArg{\permutationArg{i}}
	 \propagatorFPTruthRV:\outputsRVDummy\mapsto
	 &\underbrace{
		 \relaxation
	 \propagatorRVArg{i}
	 \left(\sampleinputsendoRVArg{i}\outputsToInputsRV
		\sum_{j=1}^{i-1}
	[\sampleoutputsRVArg{\permutationArg{j}}]^T
	 \projImageArg{\permutationArg{j}}
	\sampleoutputsRVArg{\permutationArg{j}}
\propagatorFPTruthRV(\outputsRVDummy)
 +
	\sampleinputsendoRVArg{i}\outputsToInputsRV	\sum_{j=i+1}^{\nsubsystems}
	[\sampleoutputsRVArg{\permutationArg{j}}]^T
	 \projImageArg{\permutationArg{j}}
\sampleoutputsRVArg{\permutationArg{j}}
\outputsRVDummy
		\right)
	 +
	 (1-\relaxation)
	 \projImageArg{\permutationArg{i}}
	 \sampleoutputsRVArg{\permutationArg{i}}\outputsRVDummy}_{\in\imagePropArg{i}}\\
	 &
+
	 \underbrace{
	 \relaxation
	 \propagatorTruthRVPerpArg{i}
	\left(
		\sampleinputsendoRVArg{i}\outputsToInputsRV\sum_{j=1}^{i-1}
	[\sampleoutputsRVArg{\permutationArg{j}}]^T
	\sampleoutputsRVArg{\permutationArg{j}}
\propagatorFPTruthRV(\outputsRVDummy)
 +
	\sampleinputsendoRVArg{i}\outputsToInputsRV\sum_{j=i+1}^{\nsubsystems}
	[\sampleoutputsRVArg{\permutationArg{j}}]^T
\sampleoutputsRVArg{\permutationArg{j}}
\outputsRVDummy
		\right)
	 + (1-\relaxation)\projOrthImageArg{\permutationArg{i}}\sampleoutputsRVArg{\permutationArg{i}}\outputsRVDummy
	 }_{\in\imageOrthPropArg{i}}
	  \end{split}
	  \end{align}
		for $i=1,\ldots,\nsubsystems$,
		where
$\projOrthImageArg{i}=\identityRV -
		\projImageArg{i}$ and
		$\projOrthImage\equiv\projOrthImageArg{1}\times\cdots\times\projOrthImageArg{\nsubsystems}$,
from which condition \eqref{eq:requirement} can be verified by induction with
\begin{equation}
\sampleoutputsRVArg{\permutationArg{i}}\constant =
\relaxation
	 \propagatorRVArg{i}
\left(\sampleinputsendoRVArg{i}\outputsToInputsRV\projImage
		\sum_{j=1}^{i-1}
	[\sampleoutputsRVArg{\permutationArg{j}}]^T
\sampleoutputsRVArg{\permutationArg{j}}\constant
		\right)
\end{equation}
and recalling $
	\projImage \equiv
	\sum_{i=1}^{\nsubsystems}[\sampleoutputsRVArg{\permutationArg{i}}]^T\projImageArg{i}\sampleoutputsRVArg{\permutationArg{i}}$.

	\textbf{Case 2}.
	We provide a proof by contradiction. Given any
$\outputsRVDummyArg{1},\outputsRVDummyArg{2}
	\in\imageOrthProp\subseteq
	\RVspaceDim{\noutputsRV}$, we have
$\projImage{\outputsRVDummyArg{1}} =
	\projImage{\outputsRVDummyArg{2}} =
	\zero$
	such that the right-hand-side of inequality \eqref{eq:lipschitzTruthRestrict}
	is zero. If Eq.~\eqref{eq:requirement} does not hold, then there will exist
	some $\outputsRVDummyArg{1},\outputsRVDummyArg{2}
	\in\imageOrthProp$ for which
	$\|\projImage{\propagatorFPTruthRV(\outputsRVDummyArg{1})}-
	\projImage{\propagatorFPTruthRV(\outputsRVDummyArg{2})}\|>0$, which violates
	the inequality.
\end{proof}
Recall that Remark \ref{rem:PCELHSzero} demonstrated that Assumption
\ref{ass:specialForm} holds if $\propagatorTruthRV$ is a linear operator
associated with a polynomial-chaos expansion, and so these conditions also
imply $\imageOrthProp$-invariance of the truth fixed-point operator
$\propagatorFPTruthRV$ associated with either Jacobi or Gauss--Seidel
iteration.

We now derive an \textit{a posteriori} bound on the error that leverages
Assumption \ref{ass:lipschitzTruthRestrict}.
\begin{proposition}[\textit{A posteriori} in-plane error bound] \label{prop:alternativeApost}
	If Assumption \ref{ass:lipschitzTruthRestrict} holds, then the in-plane
	error can be bounded as
	\begin{equation}\label{eq:aposterioriInPlane}
	\norm{\projImage{\outputsTruthRVTrue} - \outputsRVTrue
	}\leq
	\frac{1}{1-\lipschitz{\projImage{\propagatorFPTruthRV}}}\norm{
\outputsRVTrue-
		\projImage{\propagatorFPTruthRV(\outputsRVTrue)}
		}.
\end{equation}
\end{proposition}
\begin{proof}
	As in  Proposition \ref{prop:alternativeApriori}, we first
subtract Eq.~\eqref{eq:FP} from
	$\projImagePar{\outputsTruthRVTrue -
	\propagatorFPTruthRV(\outputsTruthRVTrue)} = \zero
	$;
	however, we add and subtract
	$\projImage{\propagatorFPTruthRV(\outputsRVTrue)}$ to obtain
\begin{equation}
	\projImage{\outputsTruthRVTrue}- \outputsRVTrue =
	\projImage{\propagatorFPTruthRV(\outputsTruthRVTrue)} -
	\projImage{\propagatorFPTruthRV(\outputsRVTrue)}
	+ \projImage{\propagatorFPTruthRV(\outputsRVTrue)} -\propagatorFPRV(\outputsRVTrue)
	.
\end{equation}
	Applying the triangle inequality, Assumption
	\ref{ass:lipschitzTruthRestrict}, and noting $\projImage{\outputsRVTrue}=\outputsRVTrue$ gives
\begin{equation}
	\norm{\projImage{\outputsTruthRVTrue}- \outputsRVTrue} \leq
	\lipschitz{\projImage{\propagatorFPTruthRV}}\norm{\projImage{\outputsTruthRVTrue}- \outputsRVTrue}
	+ \norm{\projImage{\propagatorFPTruthRV(\outputsRVTrue)} -\propagatorFPRV(\outputsRVTrue)}
	.
\end{equation}
	Finally, using $\lipschitz{\projImage{\propagatorFPTruthRV}}<1$ and
	$\propagatorFPRV(\outputsRVTrue) = \outputsRVTrue$ yields the
	desired result.
\end{proof}

Proposition \ref{cor:LHSzero} derived conditions for zero in-plane error by
proving conditions under which $\projImage{\outputsTruthRVTrue} $
is a fixed point of $\propagatorFPRV$. From Proposition
\ref{prop:alternativeApost}, we now derive conditions for zero in-plane error
by considering the case where
$\outputsRVTrue$ is a fixed point of
$\projImage{\propagatorFPTruthRV(\outputsRVTrue)}$, which makes the right-hand
side---and thus the left-hand side---of inequality \eqref{eq:aposterioriInPlane}
zero.

\begin{proposition}[Additional conditions for zero in-plane
	error]\label{cor:LHSzeroSecond}
	If Assumptions
	\ref{ass:specialForm} and
	\ref{ass:lipschitzTruthRestrict} hold, then
	\begin{equation} \label{eq:desiredZeroErrorSecond}
\outputsRVTrue=\projImage{\outputsTruthRVTrue},
	  \end{equation}
		and thus the in-plane error is zero.
\end{proposition}
\begin{proof}
	Under Assumption \ref{ass:specialForm}, the truth fixed-point operator takes
	the form
		\begin{equation}\label{eq:requiredFormFPTwo}
		\propagatorFPTruthRV:\outputsRVDummy\mapsto\propagatorFPRV(\projImage\outputsRVDummy)
	+ \propagatorFPTruthRVPerp(\outputsRVDummy)
		\end{equation}
		with $\image{\propagatorFPTruthRVPerp}
	\subseteq \imageOrthProp$, where
\begin{equation}
\propagatorFPTruthRVPerp:\outputsRVDummy\mapsto\relaxation\propagatorTruthRVPerp(\outputsRVDummy) +
		(1-\relaxation)\projOrthImage\outputsRVDummy
\end{equation}
in the case of Jacobi iteration, while
\begin{equation}
\propagatorFPTruthRVPerp:\outputsRVDummy\mapsto
\relaxation
	 \propagatorTruthRVPerpArg{i}
	\left(
		\sampleinputsendoRVArg{i}\outputsToInputsRV\sum_{j=1}^{i-1}
	[\sampleoutputsRVArg{\permutationArg{j}}]^T
	\sampleoutputsRVArg{\permutationArg{j}}
\propagatorFPTruthRV(\outputsRVDummy)
 +
	\sampleinputsendoRVArg{i}\outputsToInputsRV\sum_{j=i+1}^{\nsubsystems}
	[\sampleoutputsRVArg{\permutationArg{j}}]^T
\sampleoutputsRVArg{\permutationArg{j}}
\outputsRVDummy
		\right)
	 + (1-\relaxation)\projOrthImageArg{\permutationArg{i}}\sampleoutputsRVArg{\permutationArg{i}}\outputsRVDummy
\end{equation}
in the case of Gauss--Seidel iteration.
	Applying projection $\projImage$ to \eqref{eq:FP} and adding
	$\projImage\propagatorFPTruthRVPerp(\outputsRVTrue)(=\zero)$ yields
\begin{equation} \label{eq:FPsecondary}
	 \projImage\propagatorFPRV(\outputsRVTrue) + \projImage\propagatorFPTruthRVPerp(\outputsRVTrue)-\projImage\outputsRVTrue  =
	\zero.
\end{equation}
	Then, using \eqref{eq:requiredFormFPTwo}
	with $\outputsRVTrue=\projImage\outputsRVTrue$ yields
\begin{equation} \label{eq:FPsecondaryTwo}
	\projImage\propagatorFPTruthRV(\outputsRVTrue)-\outputsRVTrue =
	\zero.
\end{equation}
	Thus, the right-hand side of inequality \eqref{eq:aposterioriInPlane}---which is valid under
	Assumption \ref{ass:lipschitzTruthRestrict}---is zero, which yields the desired result.
\end{proof}

\section{Numerical experiments}\label{sec:numericalExperiments}

We now assess the performance of the proposed NetUQ method on a
two-dimensional stationary nonlinear diffusion equation with uncertainties
arising in the diffusion coefficient and boundary conditions. Here, we
construct the network formulation by decomposing the physical domain into
(overlapping) subdomains, and associating each subdomain with a network
component.  Thus---for this problem---the NetUQ method is tantamount to an
overlapping domain decomposition method for uncertainty propagation. We first
define the global uncertainty-propagation problem in Section
\ref{sec:globUPprob}, and subsequently define the resulting component and network
uncertainty-propagation problems in Section \ref{sec:netUPproblem}. Then,
Sections \ref{sec:strong} and \ref{sec:weak} perform strong- and weak-scaling
studies, respectively.

\subsection{Global uncertainty-propagation problem}\label{sec:globUPprob}

We express the deterministic boundary value problem (BVP) as the partial
differential equation (adopted from Ref.~\cite{grepl2007efficient})
\begin{align}\label{eq:PDE}
\begin{split}
	- \frac{\partial^2 v}{\partial \spatialVarArg{1}\partial\spatialVarArg{2}} + \left(e^{\mu v} - 1 \right)  &= 10 \sin \left(2 \pi
	\spatialVarArg{1} \right) \sin \left(2 \pi x_2 \right) \ , \quad
	\spatialVar\equiv(\spatialVarArg{1},\spatialVarArg{2}) \in
	\domain = \left(0, 1\right)^2 \\
	v(x) &= v_\Gamma,\quad x\in\partial \domain.
\end{split}
\end{align}
To formulate the uncertainty-propagation problem, we assume the boundary condition
and diffusion coefficient are ``global'' input random variables, which we model
using polynomial-chaos expansions as
\begin{equation}
\label{eq:globalInputPCE}
	v_\Gamma  =\inputsexoRVGlobalArg{1} =
	\sum_{\multiindex\in\multiindexGlobalSetArgs{\inputsexoRV}}
	\PCEbasisArgs{\multiindex}{\underlyingRV}\PCEcoeffsGlobalArgs{\inputsexosymb}{
		1 }{\multiindex} , \qquad
\mu  =  \inputsexoRVGlobalArg{2}  =
	\sum_{\multiindex\in\multiindexGlobalSetArgs{\inputsexoRV}} \PCEbasisArgs{\multiindex}{\underlyingRV}\PCEcoeffsGlobalArgs{\inputsexosymb}{ 2 }{\multiindex}
\end{equation}
with
$\stochasticDim=2$ germ random variables
$\underlyingRV \equiv \left( \underlyingRVunivArg{1},\underlyingRVunivArg{2}
\right)\sim\normal{\zero,\identity}$. Note that the PCE coefficients
$\PCEcoeffsGlobalArgs{\inputsexosymb}{ 1 }{\multiindex}$ and
$\PCEcoeffsGlobalArgs{\inputsexosymb}{2 }{\multiindex}$,
$\multiindex\in\multiindexGlobalSetArgs{\inputsexoRV}$
completely characterize the
global input random variables $\inputsexoRVGlobalArg{1}\in\RVspaceDim{}$ and
$\inputsexoRVGlobalArg{2}\in\RVspaceDim{}$, respectively. We consider the case of
independent global input random variables such that
$\inputsexoRVGlobalArg{1}$ and $\inputsexoRVGlobalArg{2}$
 are
expressed strictly in terms of $\underlyingRVunivArg{1}$ and
$\underlyingRVunivArg{2}$, respectively.
As described in Remark \ref{rem:PCE}, we define the multi-index set
via total-degree truncation as $\multiindexGlobalSetArgs{\inputsexoRV} =
\{\multiindex\in\natnoZero{\stochasticDim}\,|\,\|\multiindex\|_1\leq p\}$ for
$p=3$ and employ Hermite polynomials such that
\begin{gather*}
\PCEbasisArgs{(0,0)}{\underlyingRV}= 1 ,\\
\PCEbasisArgs{(1,0)}{\underlyingRV}=\xi_1 ,\  \ \
\PCEbasisArgs{(0,1)}{\underlyingRV}=\xi_2  ,\ \\
\PCEbasisArgs{(2,0)}{\underlyingRV}=\xi_1^2 - 1 ,\ \ \
\PCEbasisArgs{(1,1)}{\underlyingRV}=\xi_1\xi_2 , \ \ \
\PCEbasisArgs{(0,2)}{\underlyingRV}=\xi_2^2 - 1 , \ \\
\PCEbasisArgs{(3,0)}{\underlyingRV}=\xi_1^3 - 3 \xi_1 , \ \
\PCEbasisArgs{(2,1)}{\underlyingRV}=\xi_2\left( \xi_1^2 - 1 \right)   ,\ \ \
\PCEbasisArgs{(1,2)}{\underlyingRV}=\xi_1 \left(\xi_2^2 - 1 \right)  , \ \ \
\PCEbasisArgs{(0,3)}{\underlyingRV}=\xi_2^3 - 3 \xi_2.
\end{gather*}

Table \ref{tab:PCEglobal} reports the polynomial-chaos coefficients used to
represent the two non-Gaussian, independent global input random variables.
Note that while this characterization yields independent global input
random variables, the NetUQ formulation supports dependent input random
variables.
 \begin{table}[ht]
 \centering
 \begin{tabular}{|c|c|c|c|c|c|c|c|c|c|c|}
  \hline
$i$ &
$\PCEcoeffsGlobalArgs{\inputsexosymb}{ i }{(0,0)}$&
$\PCEcoeffsGlobalArgs{\inputsexosymb}{ i }{(1,0)}$&
$\PCEcoeffsGlobalArgs{\inputsexosymb}{ i }{(0,1)}$&
$\PCEcoeffsGlobalArgs{\inputsexosymb}{ i }{(2,0)}$&
$\PCEcoeffsGlobalArgs{\inputsexosymb}{ i }{(1,1)}$&
$\PCEcoeffsGlobalArgs{\inputsexosymb}{ i }{(0,2)}$&
$\PCEcoeffsGlobalArgs{\inputsexosymb}{ i }{(3,0)}$&
$\PCEcoeffsGlobalArgs{\inputsexosymb}{ i }{(2,1)}$&
$\PCEcoeffsGlobalArgs{\inputsexosymb}{ i }{(1,2)}$&
$\PCEcoeffsGlobalArgs{\inputsexosymb}{ i }{(0,3)}$\\
  \hline
$1$ &
$1.0   $ &
$0.2   $ &
$0.0   $ &
$0.02   $ &
$0.0  $ &
$0.0   $ &
$0.002 $ &
$0.0   $ &
$0.0   $ &
$0.0   $ \\
2 &
$1.0$ &
$0.0$ &
$0.2$ &
$0.0$ &
$0.0$ &
$0.0$ &
$0.02$ &
$0.0$ &
$0.0$ &
$0.002$ \\
\hline
  \end{tabular}
  \caption{Polynomial-chaos coefficients for global input random variables.}
  \label{tab:PCEglobal}
  \end{table}

The global uncertainty-propagation problem is now: Given global input random
variables $\inputsexoRVGlobalArg{1}$ and $\inputsexoRVGlobalArg{2}$
characterized using polynomial-chaos expansions
\eqref{eq:globalInputPCE} with coefficients provided by Table
\ref{tab:PCEglobal}, compute the random field $\mathsf v:\domain\rightarrow\RVspaceDim{}$,
where $\mathsf v(x)$ is the random variable associated with
the variable $v(x)$ for $x\in\domain$ due to the randomness in the diffusion
coefficient and boundary condition. Solving this global
uncertainty-propagation with non-intrusive spectral projection (NISP)
\cite{lemaitre2010book,hosder2006non}
using a
16-point Gauss--Hermite quadrature rule (derived using a full tensor
product of the 1D rule with 4 points per dimension) yields a PCE representation for the random
field. Each of the 16 quadrature points yields one instance of the
global deterministic problem \eqref{eq:PDE}. We discretize this (nonlinear) PDE using
the finite-element method (FEM) with rectangular linear elements, and we solve the
resulting system of nonlinear equations using Newton's method.  The initial
guess for the Newton solver is the zero solution.

For simplicity, we truncate the polynomial-chaos expansion of the random field
at the same level as the global inputs such that
\begin{equation}
\mathsf v(x)  =
	\sum_{\multiindex\in\multiindexGlobalSetArgs{\inputsexoRV}}
	\PCEbasisArgs{\multiindex}{\underlyingRV}\PCEcoeffsGlobalFieldArgs{\mathsf
	v}{\multiindex}(x),\quad x\in\domain.
\end{equation}
Figure \ref{fig:2dheat_soln} plots the resulting PCE coefficients
$
\PCEcoeffsGlobalFieldArgs{\mathsf v}{\multiindex}(x)
$, $\multiindex\in\multiindexGlobalSetArgs{\inputsexoRV}$, $x\in\domain$.

\begin{figure}[htbp]
\centering
	\begin{subfigure}{0.45\textwidth}
		\centering
\includegraphics[width=0.6\textwidth]{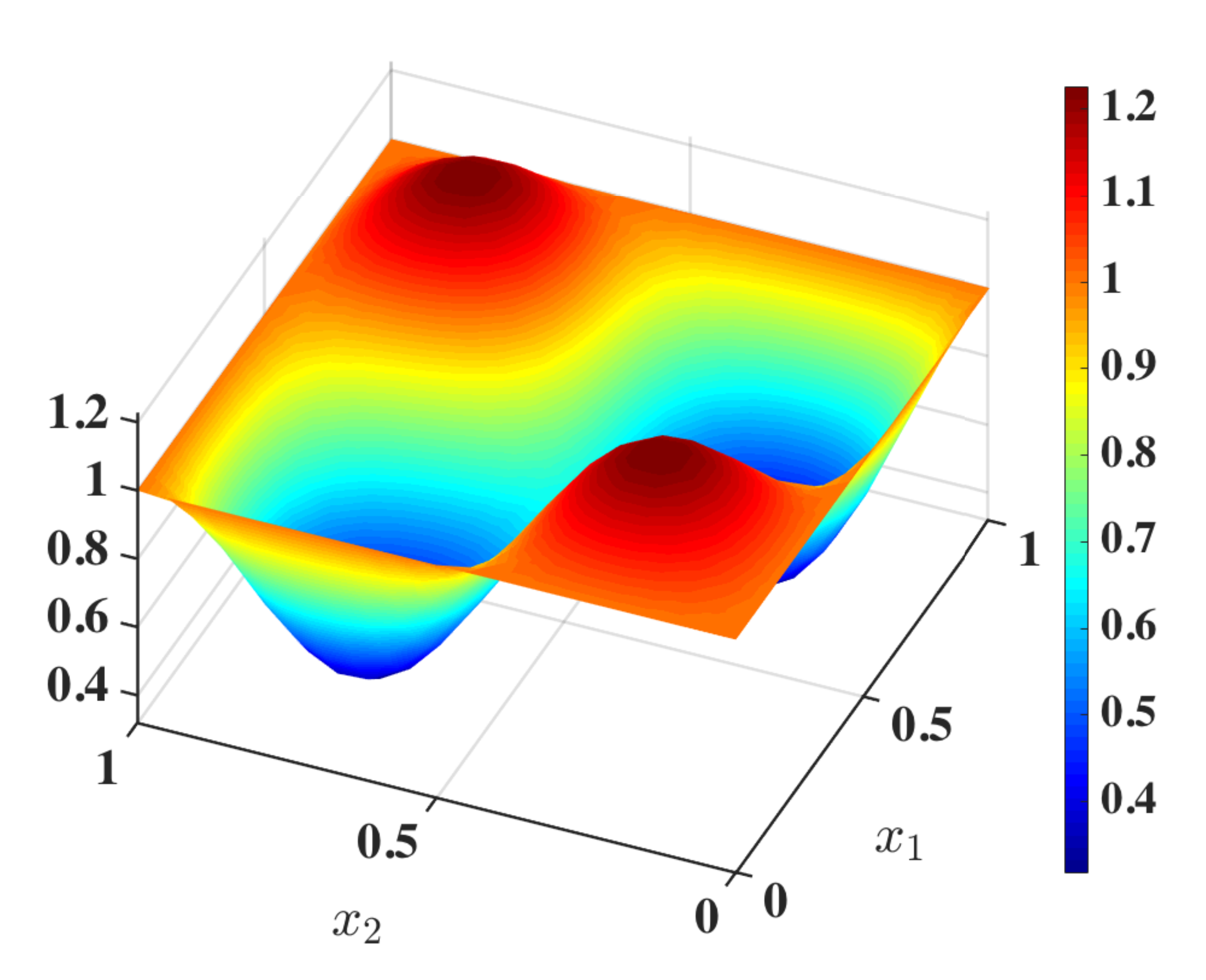}
		\caption{Output PCE coefficient
		$
		\PCEcoeffsGlobalFieldArgs{\mathsf v}{(0,0)}(x)
$
		}
	\end{subfigure}
	\begin{subfigure}{0.45\textwidth}
		\centering
\includegraphics[width=0.6\textwidth]{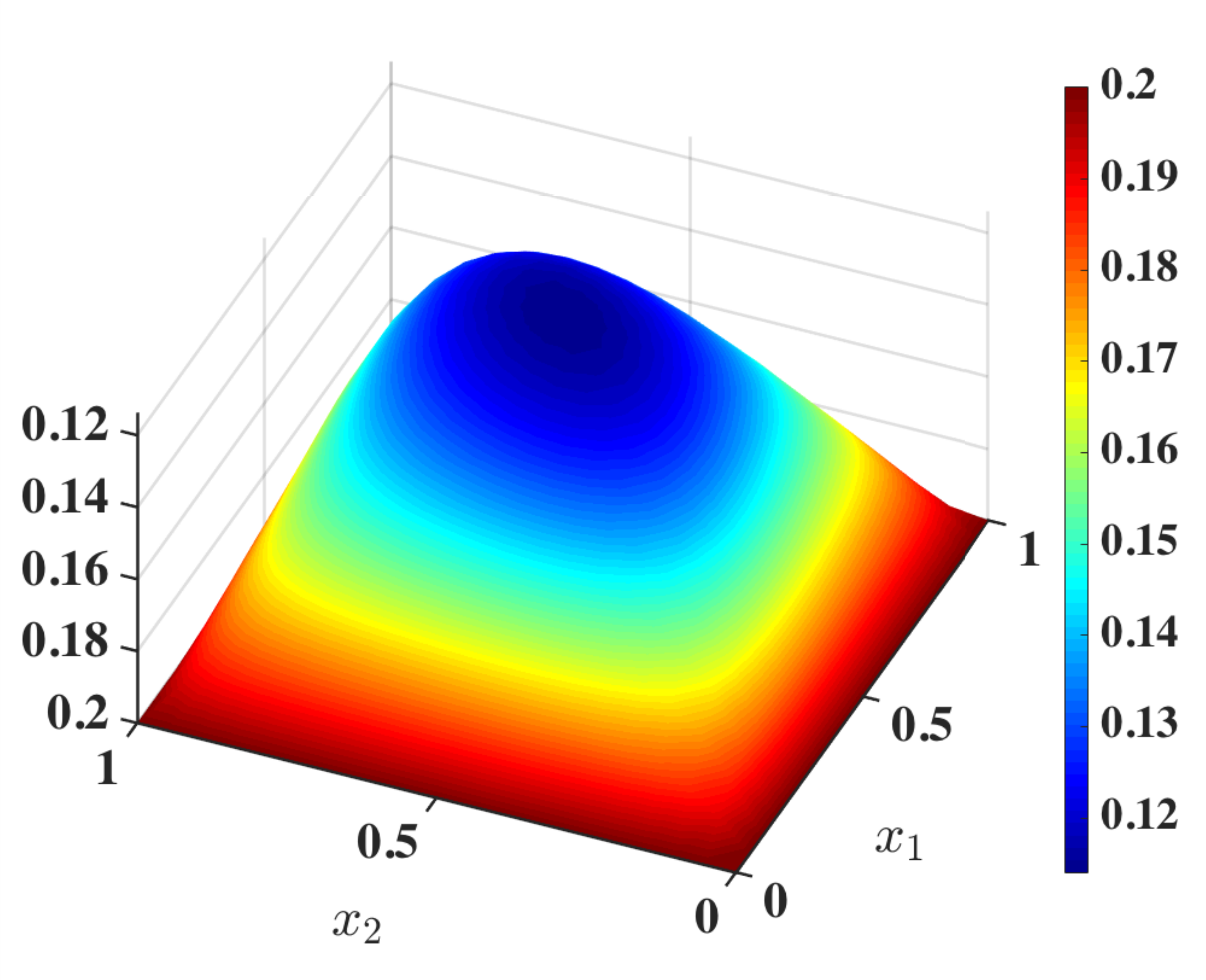}
		\caption{Output PCE coefficient
		$
		\PCEcoeffsGlobalFieldArgs{\mathsf v}{(1,0)}(x)
$
		}
	\end{subfigure}
	\begin{subfigure}{0.45\textwidth}
		\centering
\includegraphics[width=0.6\textwidth]{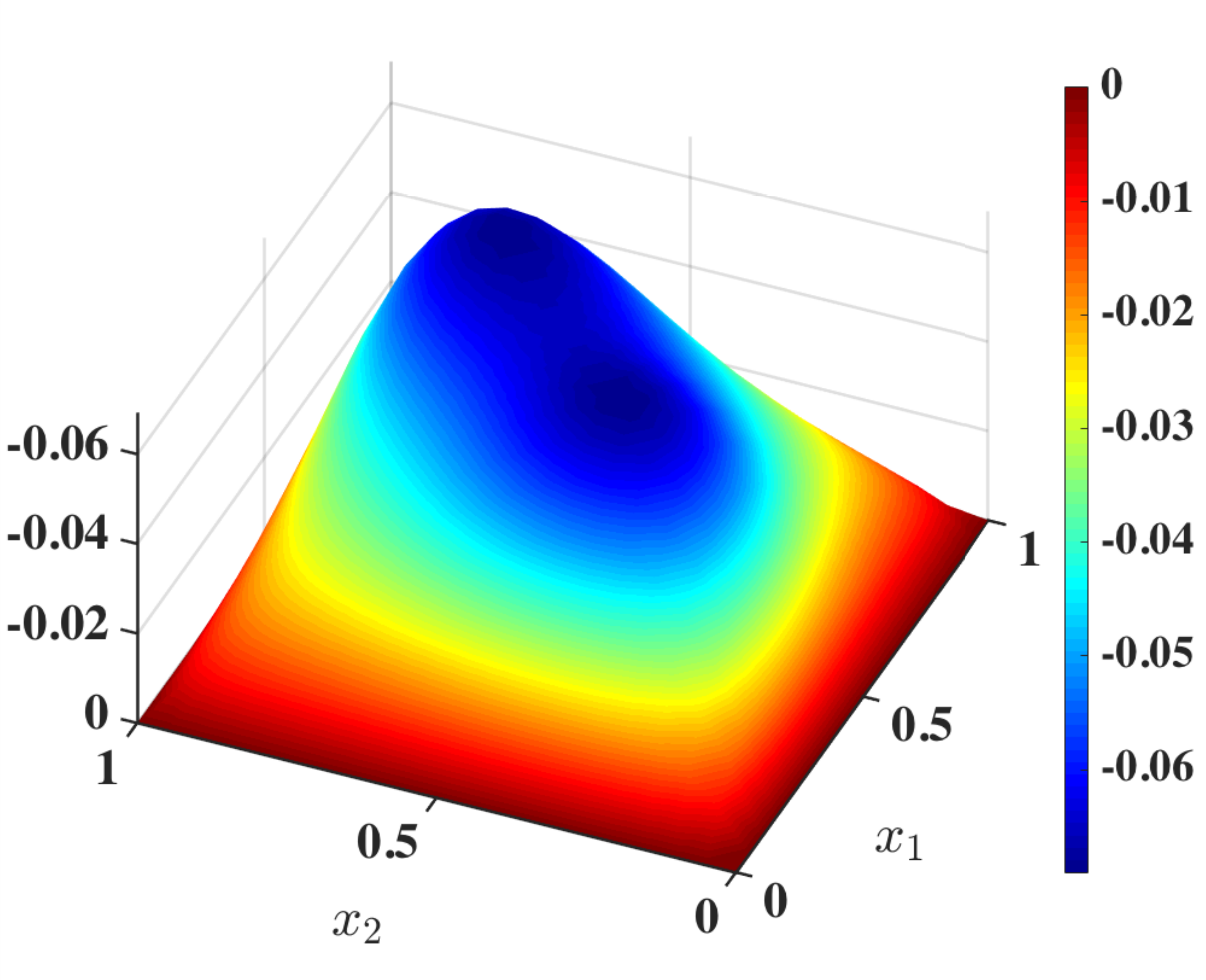}
		\caption{Output PCE coefficient $
		\PCEcoeffsGlobalFieldArgs{\mathsf v}{(0,1)}(x)
$}
	\end{subfigure}
	\begin{subfigure}{0.45\textwidth}
		\centering
\includegraphics[width=0.6\textwidth]{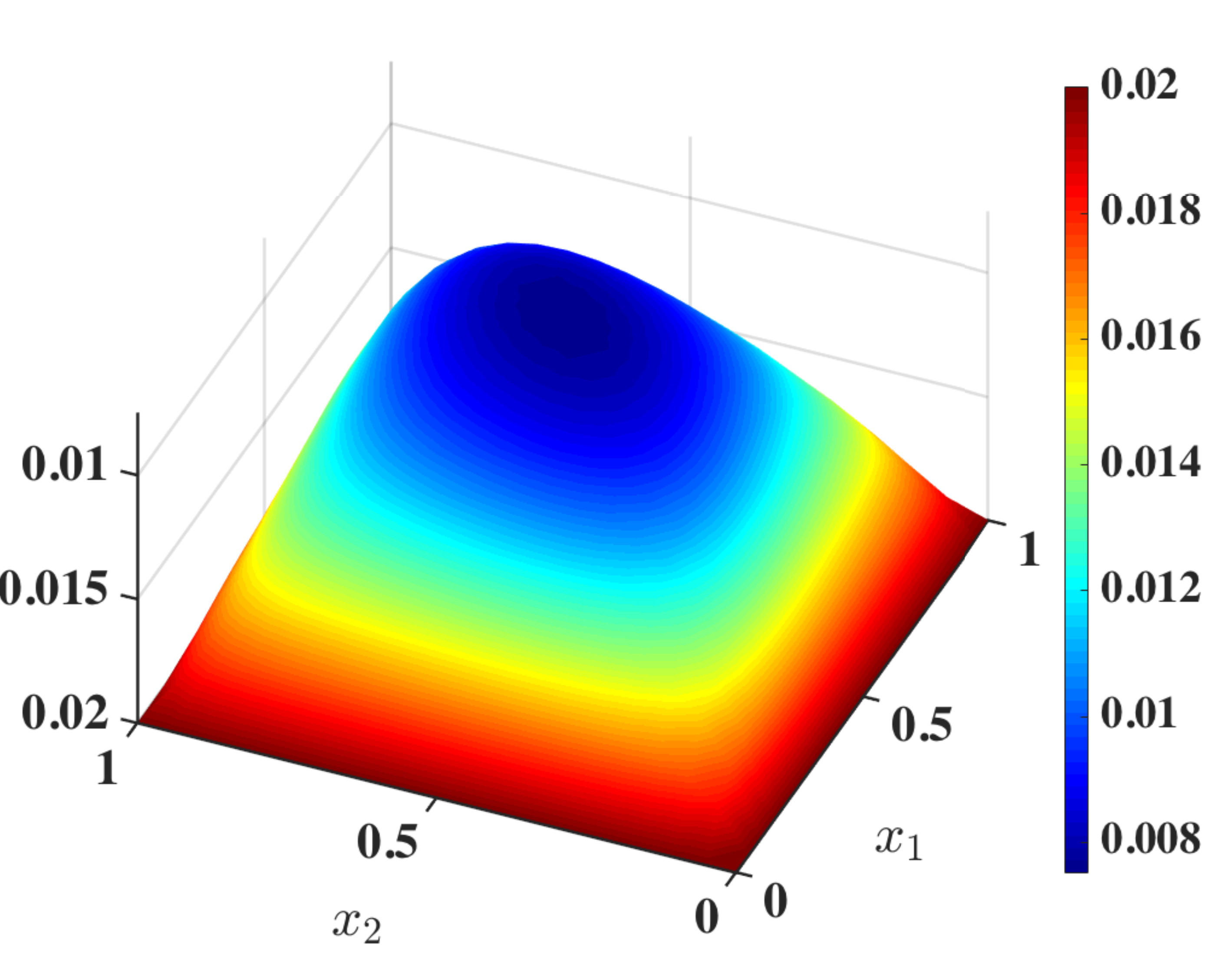}
		\caption{Output PCE coefficient $
		\PCEcoeffsGlobalFieldArgs{\mathsf v}{(2,0)}(x)
$}
	\end{subfigure}
	\begin{subfigure}{0.45\textwidth}
		\centering
\includegraphics[width=0.6\textwidth]{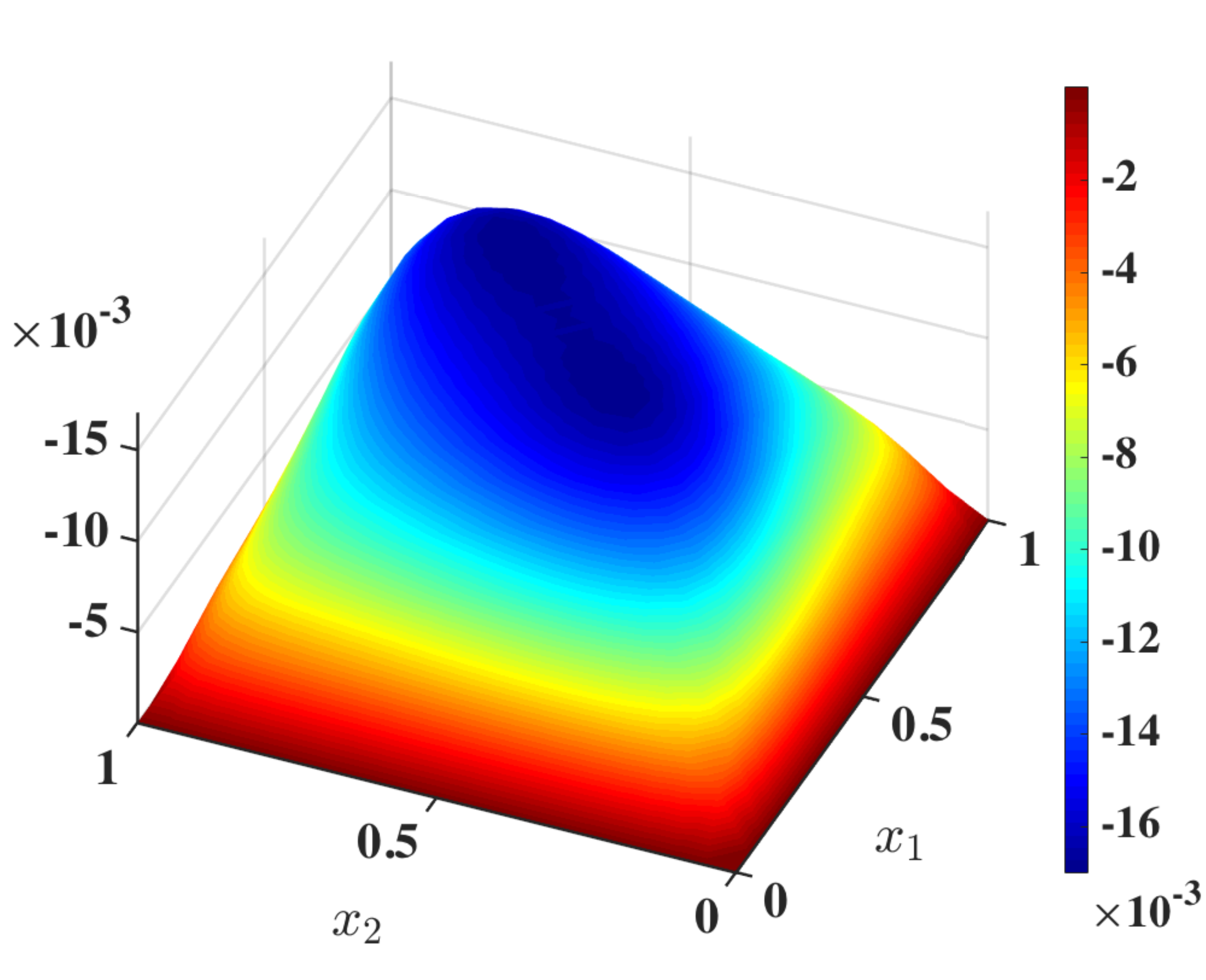}
		\caption{Output PCE coefficient $
		\PCEcoeffsGlobalFieldArgs{\mathsf v}{(1,1)}(x)
$}
	\end{subfigure}
	\begin{subfigure}{0.45\textwidth}
		\centering
\includegraphics[width=0.6\textwidth]{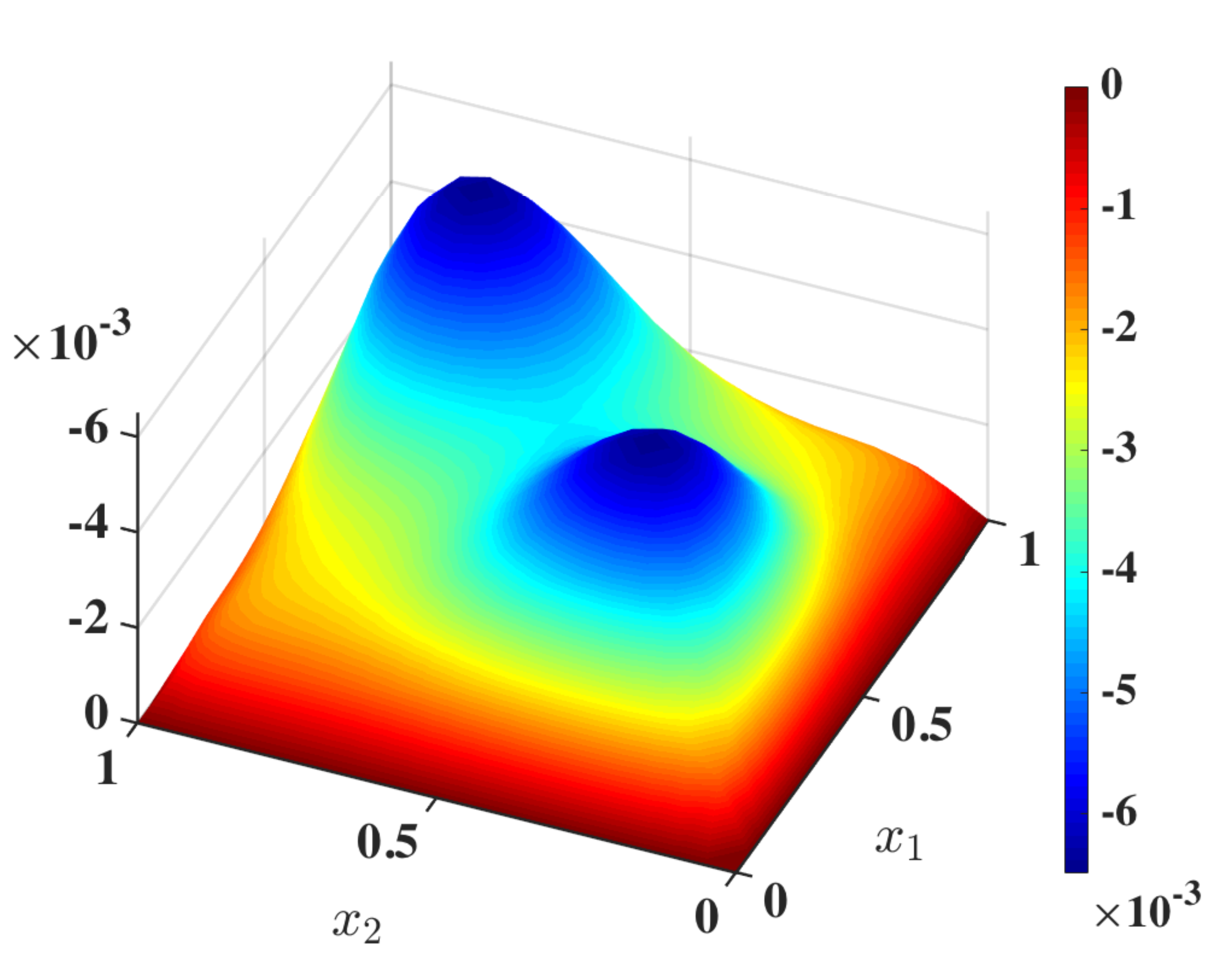}
		\caption{Output PCE coefficient $
		\PCEcoeffsGlobalFieldArgs{\mathsf v}{(0,2)}(x)
$}
	\end{subfigure}
	\begin{subfigure}{0.45\textwidth}
		\centering
\includegraphics[width=0.6\textwidth]{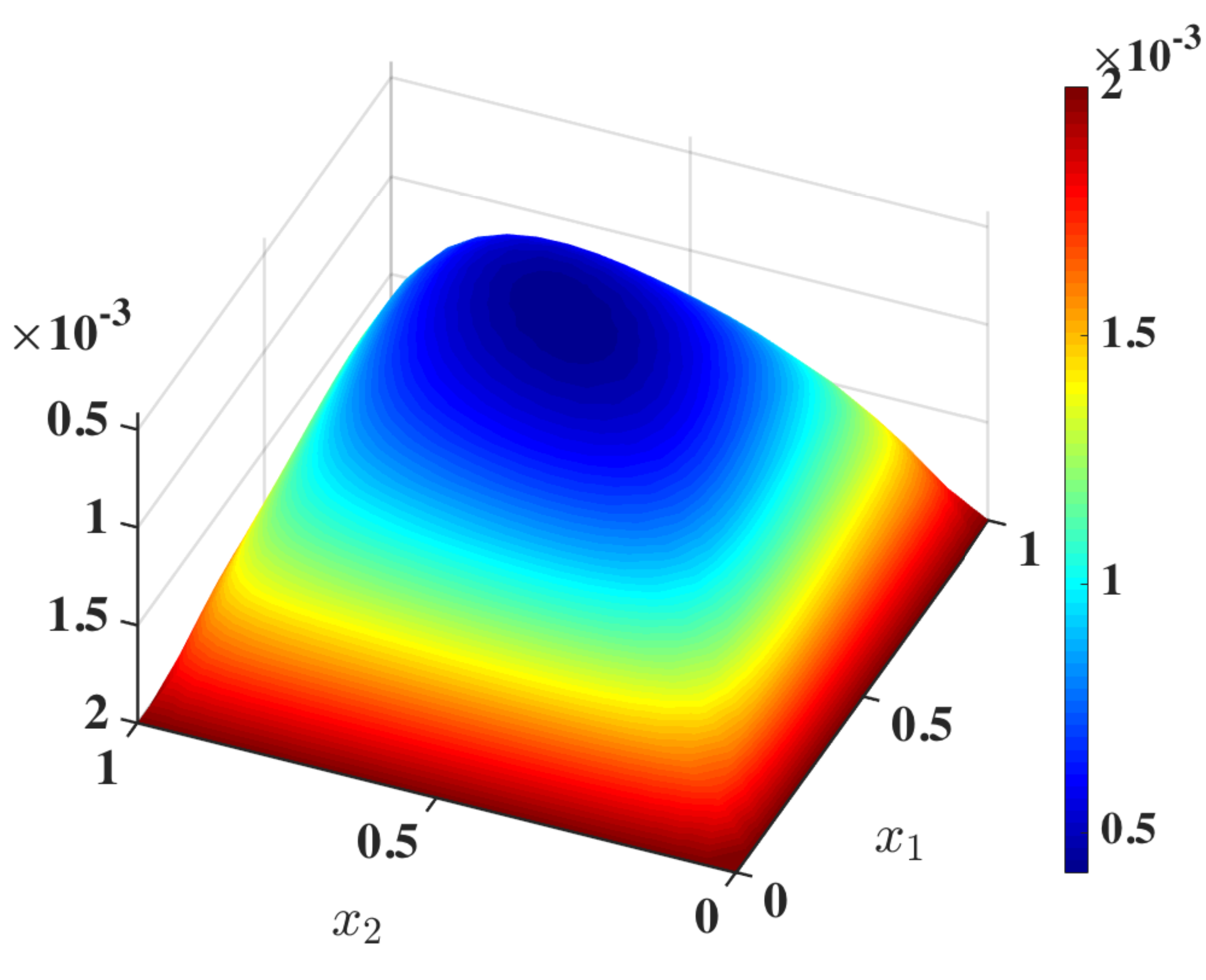}
		\caption{Output PCE coefficient $
		\PCEcoeffsGlobalFieldArgs{\mathsf v}{(3,0)}(x)
$}
	\end{subfigure}
	\begin{subfigure}{0.45\textwidth}
		\centering
\includegraphics[width=0.6\textwidth]{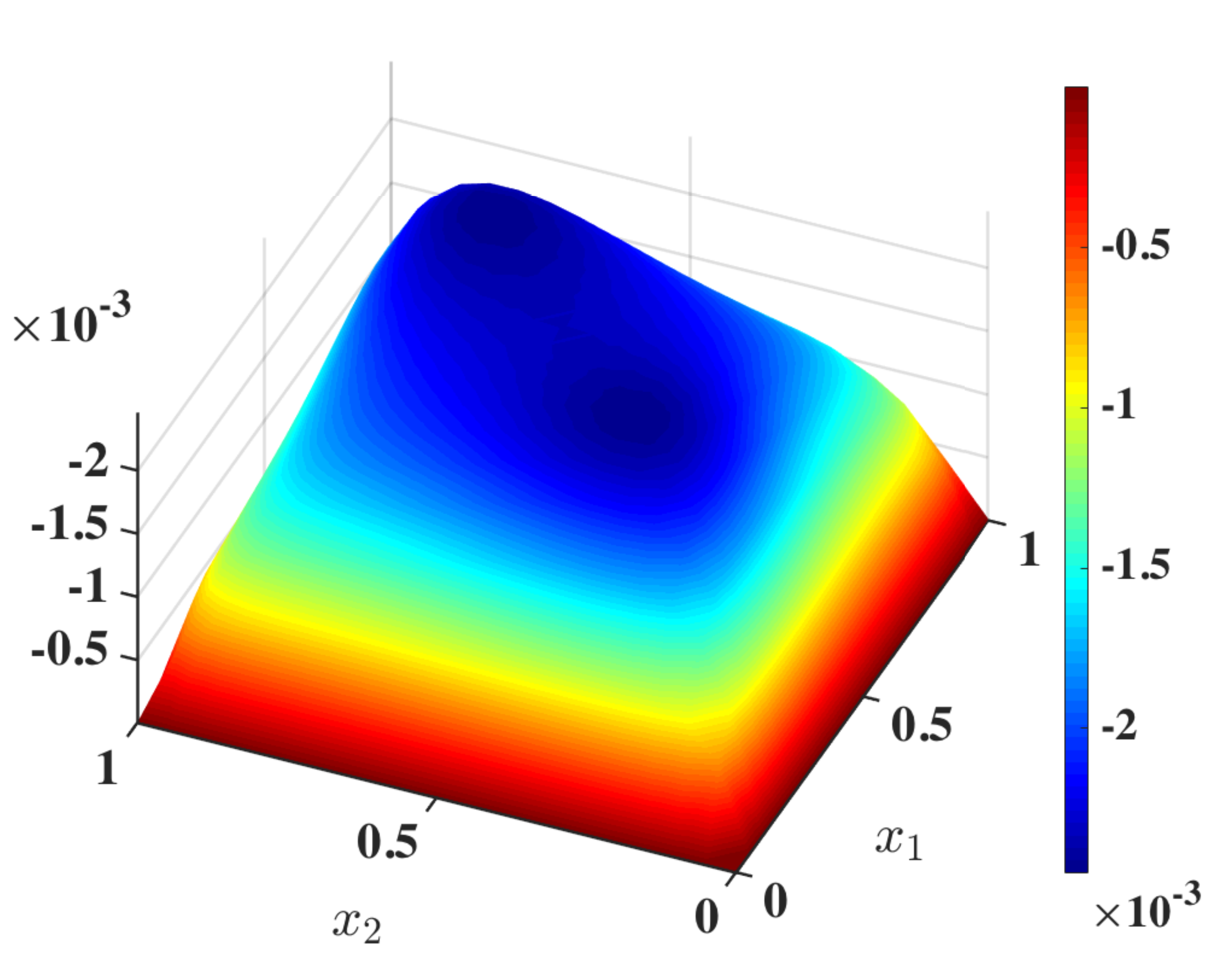}
		\caption{Output PCE coefficient $
		\PCEcoeffsGlobalFieldArgs{\mathsf v}{(2,1)}(x)
$}
	\end{subfigure}
	\begin{subfigure}{0.45\textwidth}
		\centering
\includegraphics[width=0.6\textwidth]{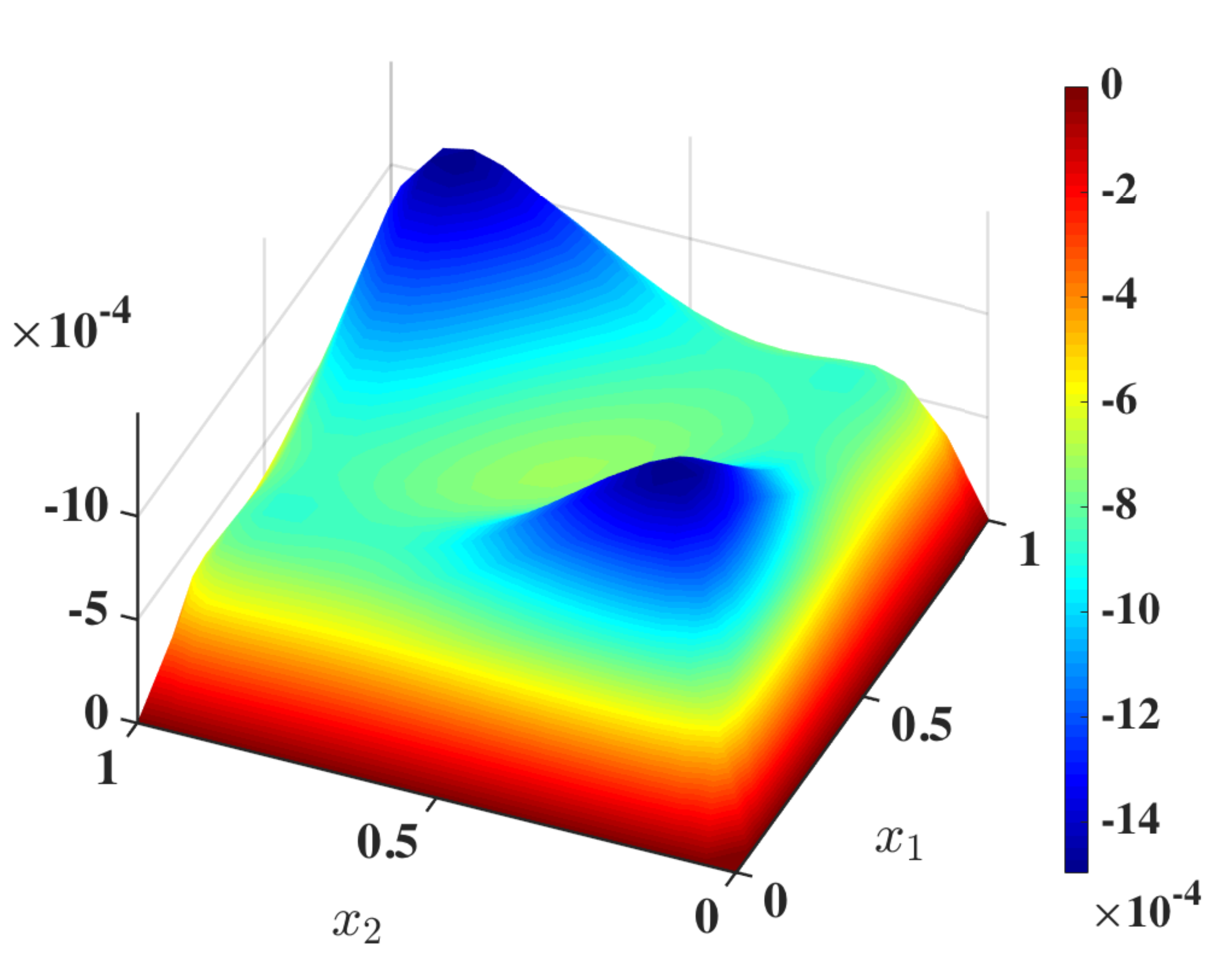}
		\caption{Output PCE coefficient $
		\PCEcoeffsGlobalFieldArgs{\mathsf v}{(1,2)}(x)
$}
	\end{subfigure}
	\begin{subfigure}{0.45\textwidth}
		\centering
\includegraphics[width=0.6\textwidth]{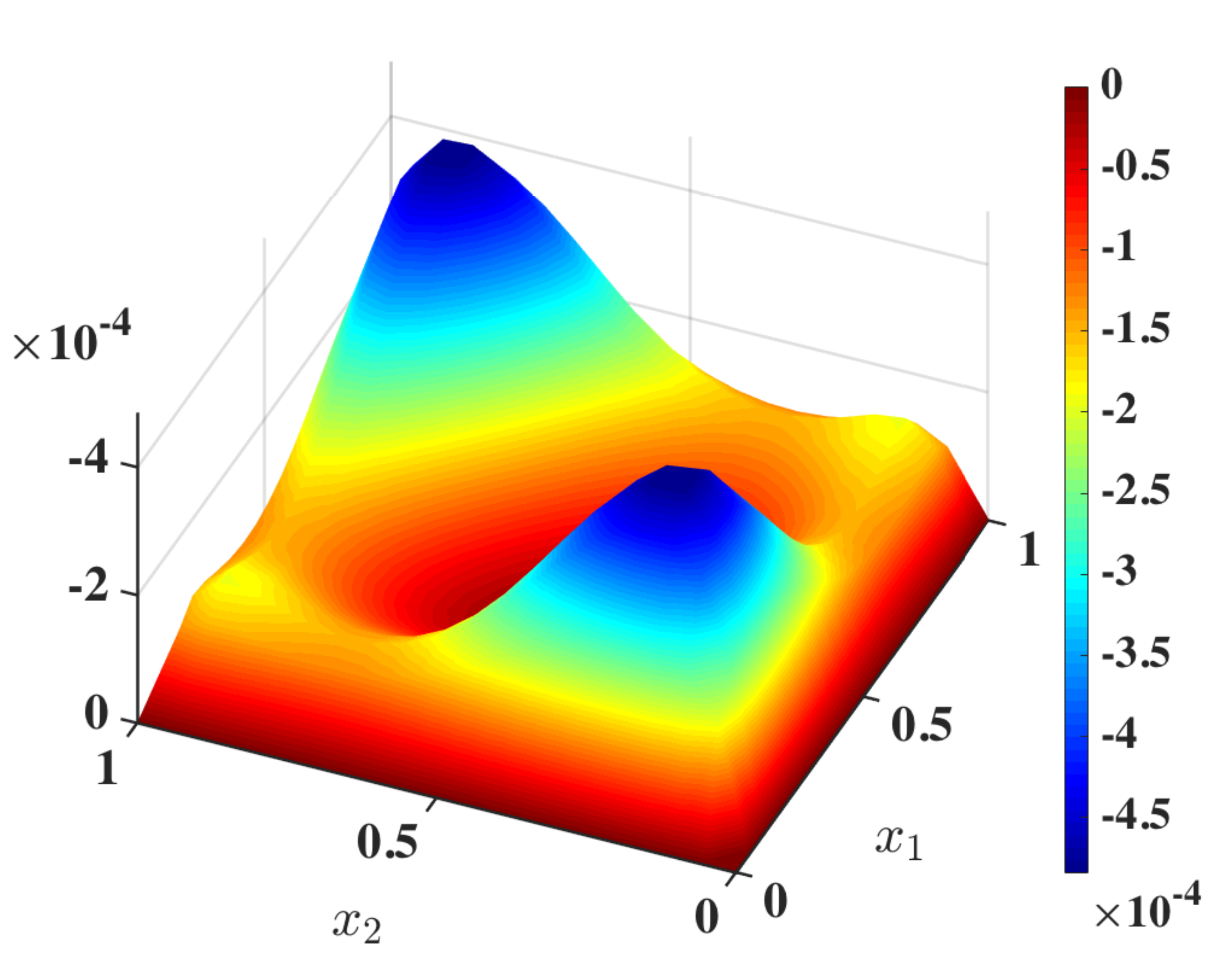}
		\caption{Output PCE coefficient $
		\PCEcoeffsGlobalFieldArgs{\mathsf v}{(0,3)}(x)
$}
	\end{subfigure}
	\caption{\captionFirst PCE coefficients for
	the random field $\mathsf v(x)$ computed via NISP
	using a
16-point, 4-level Gauss--Hermite quadrature rule.}
\label{fig:2dheat_soln}
\end{figure}

\subsection{Component and network uncertainty-propagation
problems}\label{sec:netUPproblem}

We construct the network uncertainty-propagation problem by decomposing the
global problem into overlapping subdomains
$\{\subdomainArg{i}\}_{i=1}^{\nsubsystems}$ satisfying
$\cup_{i=1}^{\nsubsystems}\subdomainArg{i}=\domain$, and treating each subdomain as a
component.  Then, the component uncertainty-propagation problem is defined
as uncertainty propagation performed on the associated subdomain.  In
particular, the deterministic BVP for the $i$th component is
\begin{align}\label{eq:PDEsub}
	\begin{split}
	&- \frac{\partial^2 v^i}{\partial \spatialVarArg{1}\partial\spatialVarArg{2}} + \left(e^{\mu v} - 1 \right)  = 10 \sin \left(2 \pi
	\spatialVarArg{1} \right) \sin \left(2 \pi x_2 \right) \ , \quad
	\spatialVar \in
	\subdomainArg{i} \subseteq\domain\\
	&v^i(x) = v_\Gamma,\ x\in\partial \subdomainArg{i}\cap\partial
	\domain,\qquad
	v^i(x) = v^j(x),\ x\in\partial
		\subdomainArg{i}\cap\subdomainArg{j}\setminus\partial
\subdomainArg{j},\
		j\in\neighborSet
	\end{split}
\end{align}
for $i=1,\ldots,\nsubsystems$, where
$\neighborSet\subset\{1,\ldots,\nsubsystems\}$ denotes the components
neighboring the $i$th component.

The $i$th component uncertainty-propagation problem is now: Given global input
random variables $\inputsexoRVGlobalArg{1}$ and $\inputsexoRVGlobalArg{2}$
characterized using polynomial-chaos expansions \eqref{eq:globalInputPCE} with
coefficients provided by Table \ref{tab:PCEglobal}, as well as
random fields
$\mathsf v^j:
\partial
		\subdomainArg{i}\cap\subdomainArg{j}\setminus\partial
\subdomainArg{j}
$ for $j\in\neighborSet$, where
$\mathsf v^j(x)$ is the random variable associated with the variable $v^j(x)$
on the boundary $x\in\partial
		\subdomainArg{i}\cap\subdomainArg{j}\setminus\partial
\subdomainArg{j}$,
compute the random field
$\mathsf v^i:\subdomainArg{j}\rightarrow\RVspaceDim{}$, where $\mathsf v^i(x)$ is the
random variable associated with the variable $v^i(x)$ for
$x\in\subdomainArg{i}$.

As with the global uncertainty-propagation problem, we
solve this component uncertainty-propagation with non-intrusive spectral
projection (NISP) using a 16-point, 4-level Gauss--Hermite quadrature rule, which yields a PCE
representation for the random field. Each of the 16 quadrature points
yields one instance of the component deterministic problem \eqref{eq:PDEsub}. We
use the same spatial discretization as the global problem, and we solve
the resulting system of nonlinear equations using Newton's method. The initial
guess for the Newton solver is the solution of the problem at the previous
fixed-point iteration; at the first iteration, the initial guess is zero.

Again, we truncate the polynomial-chaos expansion of the $i$th component's
random field at the same level as the global inputs such that
\begin{equation}
\mathsf v^i(x)  =
	\sum_{\multiindex\in\multiindexGlobalSetArgs{\inputsexoRV}}
	\PCEbasisArgs{\multiindex}{\underlyingRV}\PCEcoeffsFieldArgs{\mathsf
	v}{\multiindex}{i}(x),\quad x\in\subdomainArg{i}.
\end{equation}

Abstractly, we view the mapping from global input random variables
$\inputsexoRVGlobalArg{1}$ and $\inputsexoRVGlobalArg{2}$
and neighbor field random variables
$
v^j(x),\ x\in\partial
		\subdomainArg{i}\cap\subdomainArg{j}\setminus\partial
\subdomainArg{j},\
		j\in\neighborSet $
to the field random variables
$
v^i(x),\ x\in
\partial
		\subdomainArg{j}\cap\subdomainArg{i}\setminus\partial
\subdomainArg{i},\
		j\in\neighborSet $ computed using NISP in this manner as defining the
		component uncertainty-propagation problem defined in Section
		\ref{sec:localProb}. In particular,
	the exogenous-input random variables correspond to the
	component global input random variables such that $\inputsexoRVi =
	(\inputsexoRVGlobalArg{1},\inputsexoRVGlobalArg{2})$, the endogenous-input
	random variables correspond to the neighbor field random variables
	such that
$
\inputsendoRVi
=
(v^j(x),\ x\in\partial
		\subdomainArg{i}\cap\subdomainArg{j}\setminus\partial
\subdomainArg{j})_{
	j\in\neighborSet}
$, and the output random variables correspond to the field random variables on
the current subdomain that are used to define boundary conditions on
neighboring subdomains such that
$
\outputsRVi= (
v^i(x),\ x\in
\partial
		\subdomainArg{j}\cap\subdomainArg{i}\setminus\partial
\subdomainArg{i})_{j\in\neighborSet}
$. Because we are employing a finite-element spatial discretization, all field
variables are represented in the finite-element trial space such that the
endogenous-input random variables $\inputsendoRVi$ and output random variables
$\outputsRVi$ defined above are also finite dimensional.

As described in
Remark \ref{rem:PCE}, because we are employing PCE representations of random
variables, we can equivalently express the component uncertainty-propagation
problem in terms of the PCE coefficients themselves via
Eq.~\eqref{eq:propagatoriDef}. The associated network uncertainty-propagation
problem in terms of PCE coefficients is derived analogously to
\eqref{eq:globalResRV} and is
\begin{equation}\label{eq:globalRes}
\res(\outputsTrue,\inputsexo) = \zero,
\end{equation}
where
$\inputsexo\defeq[\inputsexoArg{i}^T\ \cdots\
\inputsexoArg{\nsubsystems}^T]^T\in\spaceDim{\ninputsexo}$ and
$\outputs\defeq[\outputsArg{i}^T\ \cdots\
\outputsArg{\nsubsystems}^T]^T\in\spaceDim{\noutputs}$
\ such that
$\ninputsexo\defeq\sum_{i=1}^{\nsubsystems}\ninputsexoArg{i}$
and
$\noutputs\defeq\sum_{i=1}^{\nsubsystems}\noutputsArg{i}$.
We denote the value of the outputs that satisfies the fixed point problem
by
$\outputsTrue\equiv\outputsTrue(\inputsexo)\in\spaceDim{\noutputs}$.

Figure~\ref{fig:2dheat_dd} illustrates construction of the network
uncertainty-propagation problem
a decomposition involving 4 subdomains defined on a $2 \times 2$ grid.
Note that we employ an overlap region spanning one element between neighboring subdomains.
Analogously, Fig.~\ref{fig:2dheat_network_4x4} illustrates the network
connectivity for a decomposition involving 16 subdomains defined on a $4
\times 4$ grid.

\begin{figure}[h!]
	\centering
	\begin{subfigure}{0.2\textwidth}
		\centering
		\includegraphics[width=0.9\textwidth]{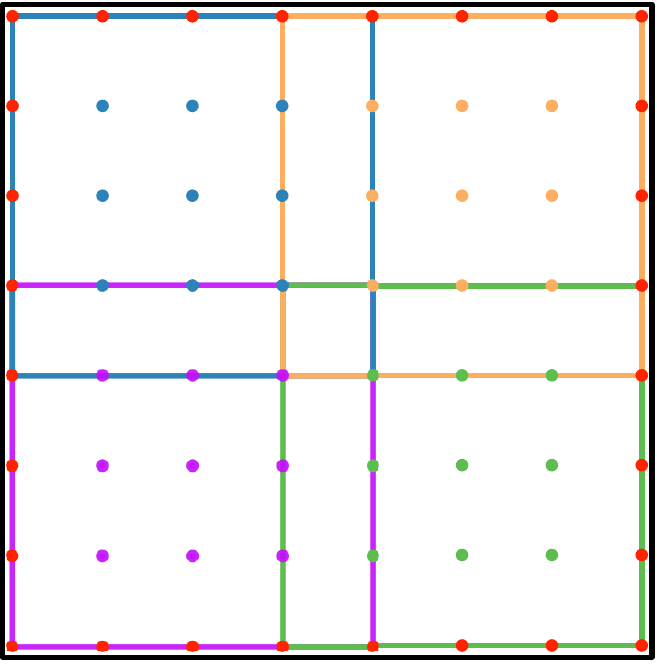}
		\caption{Decomposition of original domain into 4 overlapping subdomains or
		components}
	\end{subfigure}
	\begin{subfigure}{0.55\textwidth}
		\centering
		\includegraphics[width=0.9\textwidth]{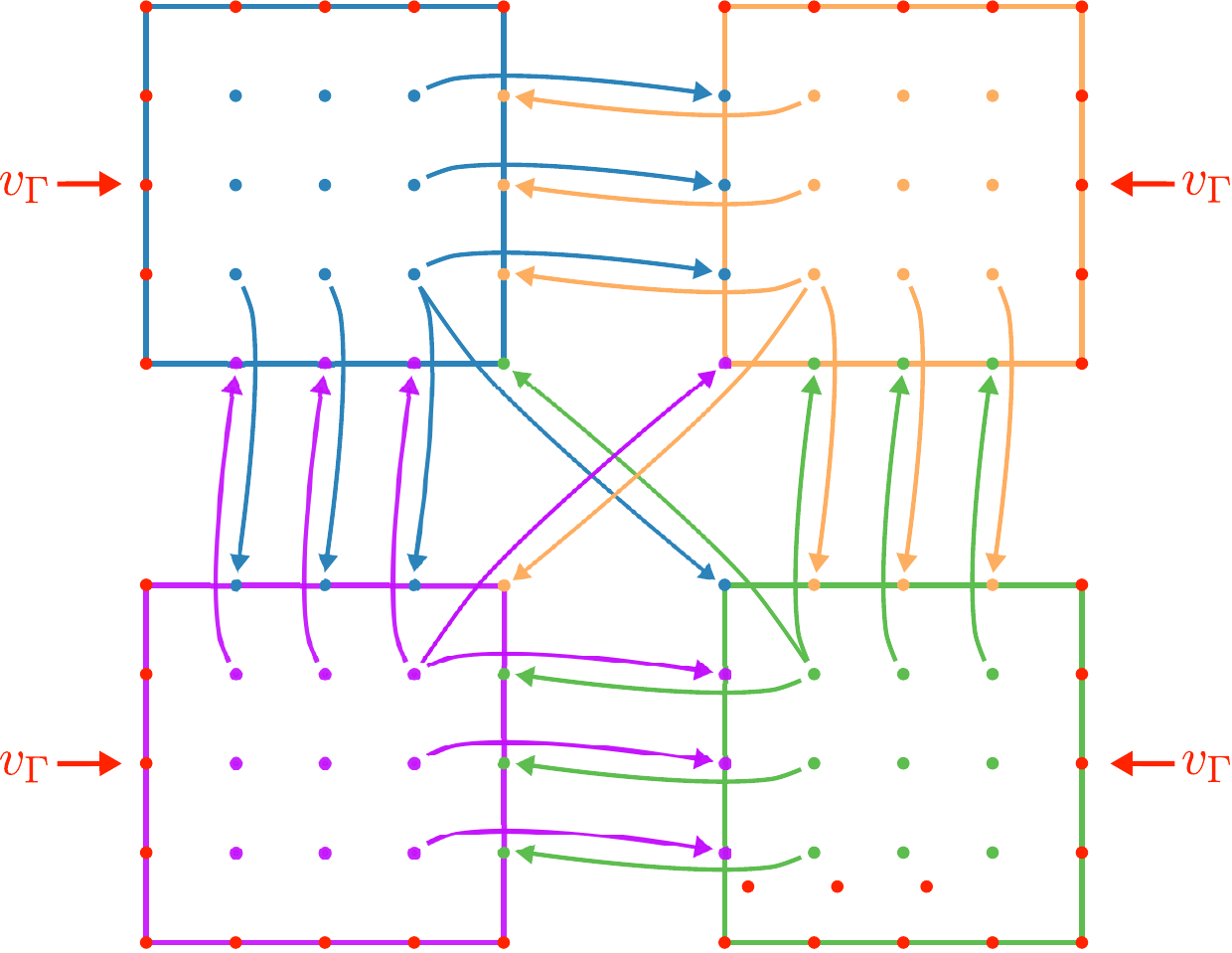}
		\caption{Inter-component connectivity}
	\end{subfigure}
	\begin{subfigure}{0.2\textwidth}
		\centering
		\includegraphics[width=0.9\textwidth]{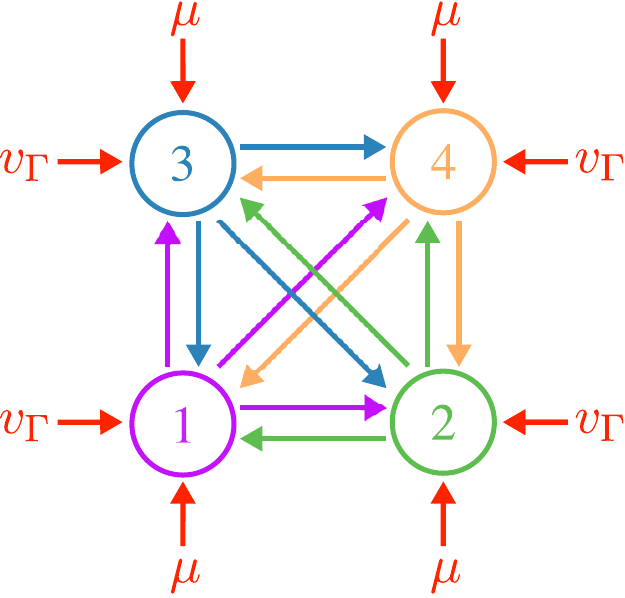}
		\caption{Resulting network}
	\end{subfigure}
	\caption{\captionFirst Network formulation for the $2\times 2$ component
	case.}
	\label{fig:2dheat_dd}
\end{figure}

\begin{figure}[h!]
	\centering
	\includegraphics[width=0.35\textwidth]{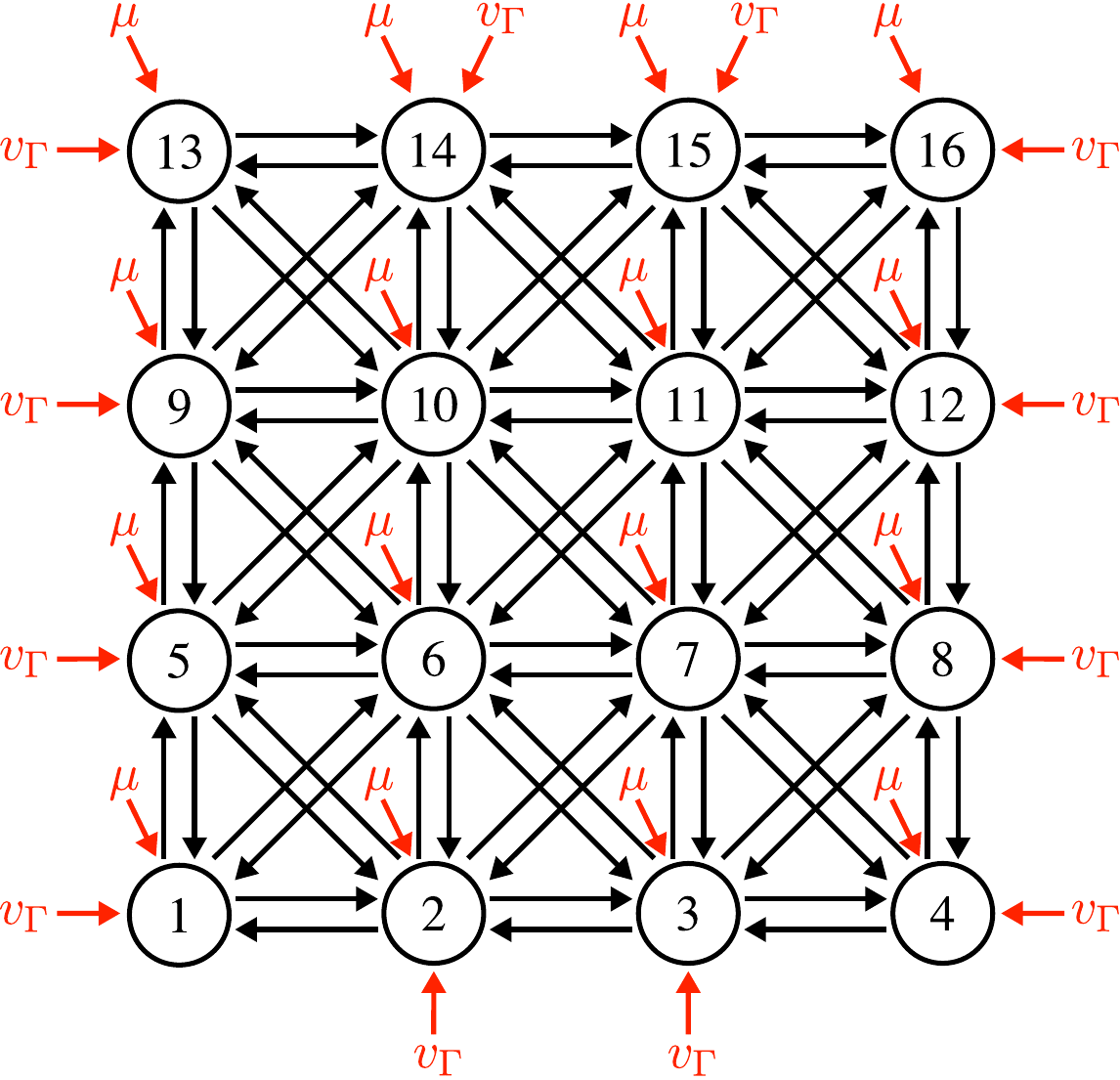}
	\caption{\captionFirst Network formulation $4\times 4$ component case.
	Figure shows coupling between nodes for the network uncertainty-propagation
	problem, with black edges corresponding to endogenous-input random variables
	and red edges corresponding to exogenous-input random variables.}
	\label{fig:2dheat_network_4x4}
\end{figure}

To solve the resulting network uncertainty-propagation problem, we investigate
both the Jacobi and Gauss--Seidel NetUQ methods, with two values of the relaxation
factor, with and without Anderson acceleration. In the case of Gauss--Seidel,
we also assess the effect of different permutations $\permutationTuple$
on performance. The initial guess for the relaxation methods is the zero
solution (i.e., all PCE coefficients for output random variables are zero).
To assess convergence of the relaxation methods at iteration $k$, we compute
the relative residual $\relativeResidual(\outputs^{(k)},\inputsexo)$, where
\begin{equation}
	\relativeResidual:(\outputs,\inputsexo) \mapsto \|\res(\outputs,\inputsexo)\|_2/
\|\res(\zero,\inputsexo)\|_2.
\end{equation}
All timings are obtained by performing calculations on an Intel(R) Xeon(R)
Core(TM) i7-5557U CPU @ 3.10GHz with 16 GB RAM. The NetUQ software is written
in Matlab; it wraps around the Uncertainty Quantification Toolkit (UQTk)
\cite{debusschere2004numerical,debusschere2016uncertainty}, which performs
component uncertainty propagation.

\subsection{Strong-scaling study}\label{sec:strong}

Here we investigate the strong scaling performance of NetUQ. In the context of
NetUQ, strong scaling associates with the case where a fixed system can be
broken down into smaller and smaller components or assemblies on which
uncertainty-propagation can be performed. To perform strong
scaling, we fix the finite-element discretization of the global problem using
1681 nodes arising from a mesh characterized by a uniform $41 \times 41$ grid.
Subsequently, we increase the number of subdomains used to define the network
on this fixed global finite-element discretization, always employing an
overlap region spanning one element, and always employing decompositions that
are uniform in $\spatialVarArg{1}$ and $\spatialVarArg{2}$.  In
particular, we consider $2\times 2$, $4\times 4$, and $8\times 8$
decompositions.  Note that the size of each component deterministic problem
increases as the number of components decreases, and employing
$\nsubsystems=1$ corresponds to the global uncertainty-propagation problem
described in Section \ref{sec:globUPprob}.

Fig.~\ref{fig:2dheat_strong_conv} reports convergence results for this study,
where we have employed a permutation $\permutationTuple$ for Gauss--Seidel that
yields the minimum number of sequential steps per iteration of $\nsequential=4$ .
This figure elucidates several trends. First, we note that convergence is
faster
for smaller networks; this is sensible, as more iterations are required
for information to propagate throughout the domain when the method uses a
larger number of (smaller) subdomains. Indeed, in the limiting case of
$\nsubsystems=1$, the NetUQ formulation is equivalent to the global
uncertainty-propagation problem, which---by definition---converges in a single
iteration.
Second, we note that Gauss--Seidel yields faster convergence than Jacobi when
measured as a function of the number of iterations. As discussed in Section
\ref{sec:gaussseidel}, this
occurs because Gauss--Seidel employs more updated information within each
iteration. However, because Gauss--Seidel employs more
sequential steps per iteration than the Jacobi method, each Gauss--Seidel
iteration incurs a larger wall time, and thus these results are
insufficient for assessing which method yields the best parallel wall-time performance.
Third, we observe that the classical iterations (i.e., employing a relaxation
factor of $\relaxation=1$) yield faster convergence than
under-relaxation performed with a
relaxation factor $\relaxation = 2/3$. However, over-relaxation
(i.e., employing a relaxation
factor of $\relaxation>1$) sometimes resulted in divergence; thus, we have
omitted these results (see Remark \ref{rem:controlConverge}).
Fourth, we note that employing Anderson acceleration yields substantially
faster convergence.

\begin{figure}[h!]
\centering
	\begin{subfigure}{0.4\textwidth}
	\centering
	\includegraphics[width=0.95\textwidth]{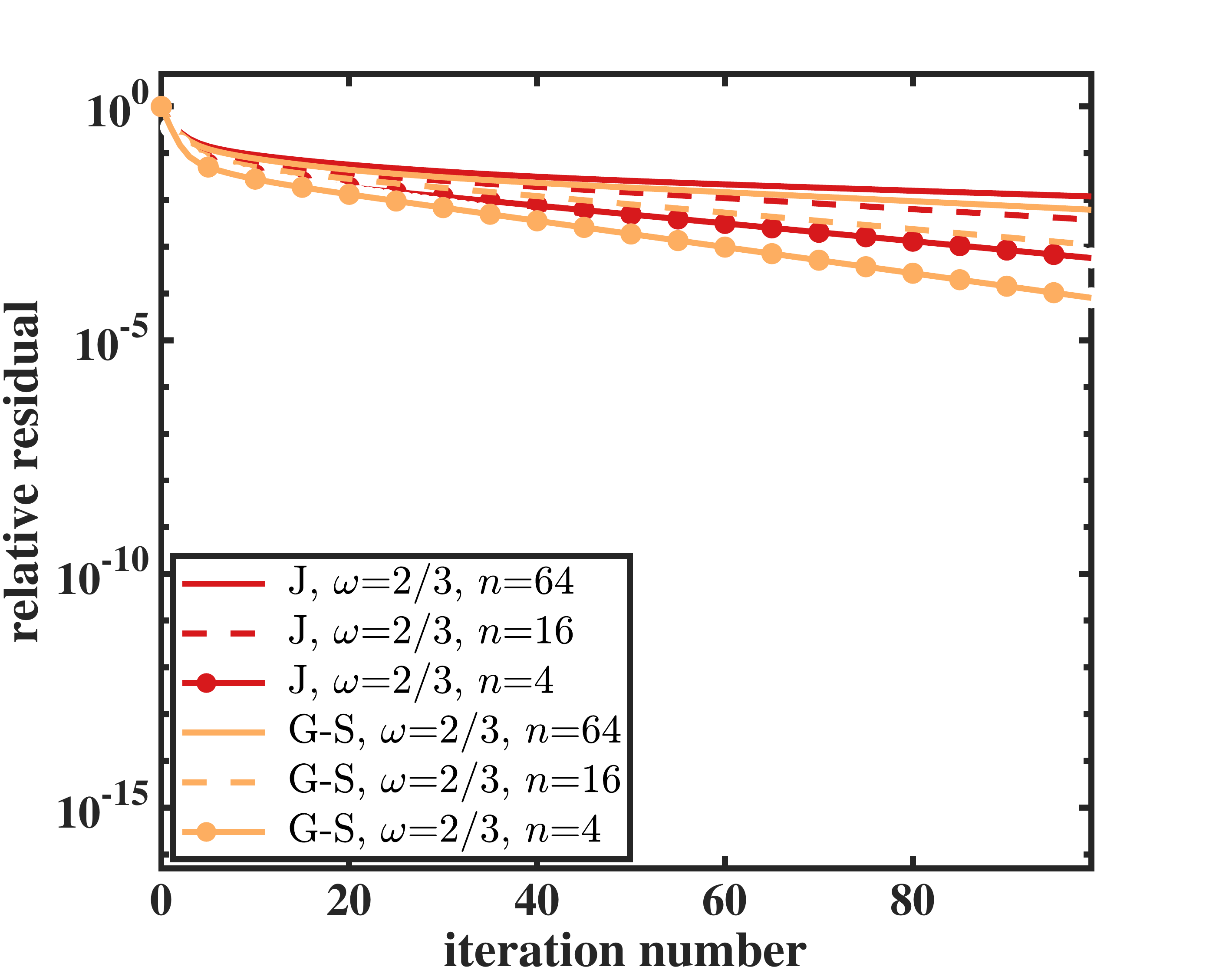}
	\caption{$\relaxation = 2/3$; no Anderson acceleration}
\end{subfigure}
\begin{subfigure}{0.4\textwidth}
	\centering
	\includegraphics[width=0.95\textwidth]{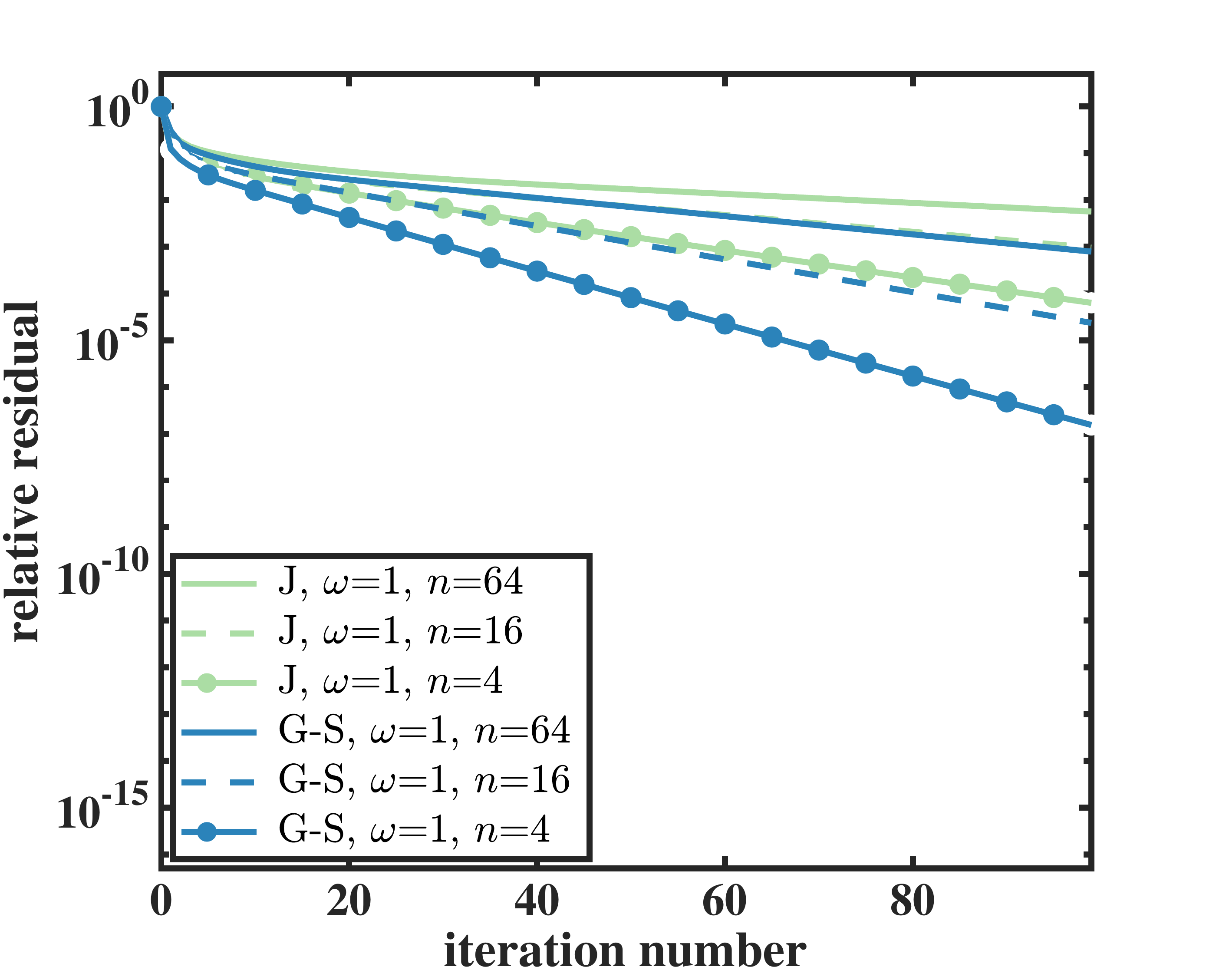}
	\caption{$\relaxation = 1$; no Anderson acceleration}
\end{subfigure}
\begin{subfigure}{0.4\textwidth}
\centering
\includegraphics[width=0.95\textwidth]{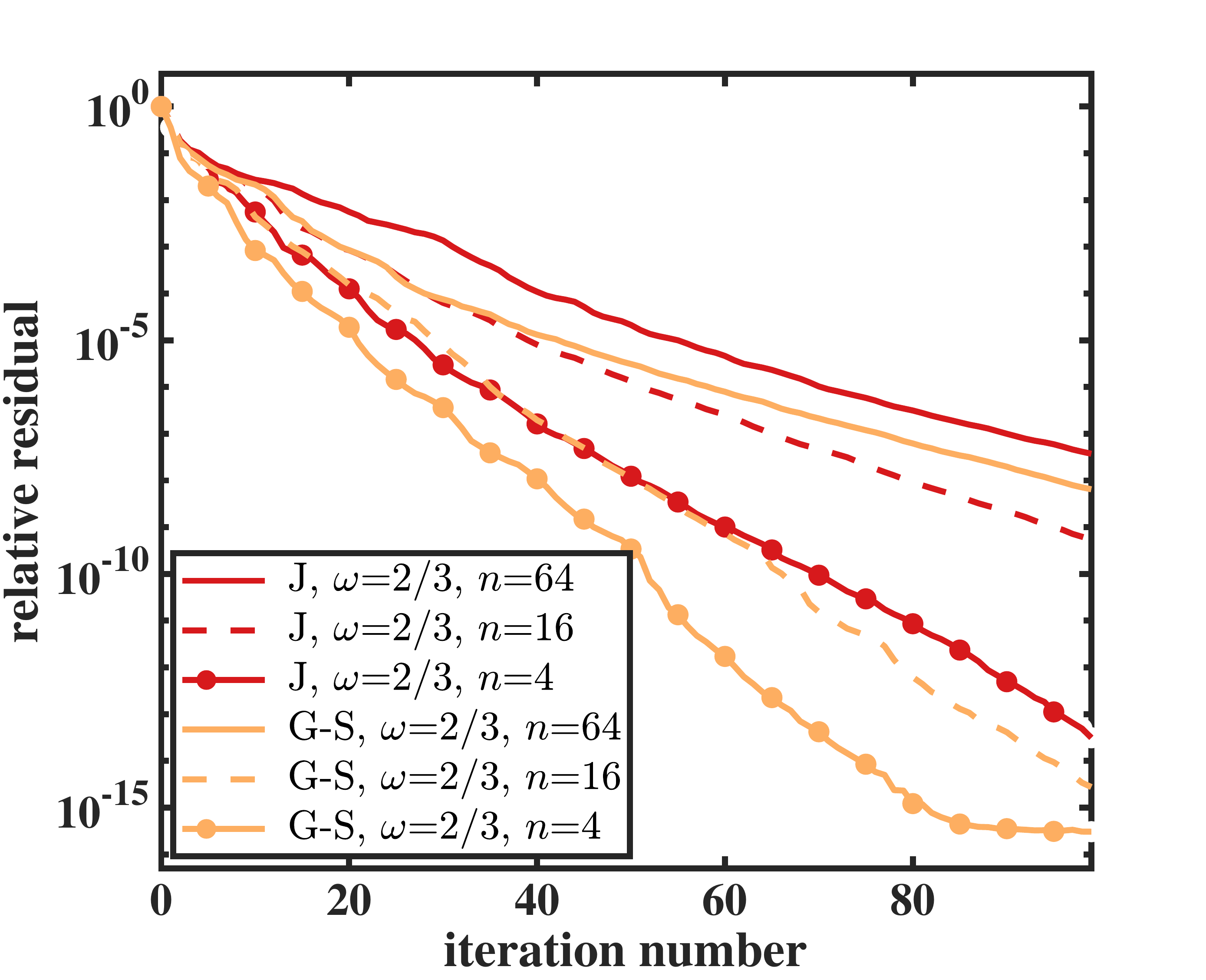}
\caption{$\relaxation = 2/3$; Anderson acceleration, $\memory=5$}
\end{subfigure}
\begin{subfigure}{0.4\textwidth}
\centering
\includegraphics[width=0.95\textwidth]{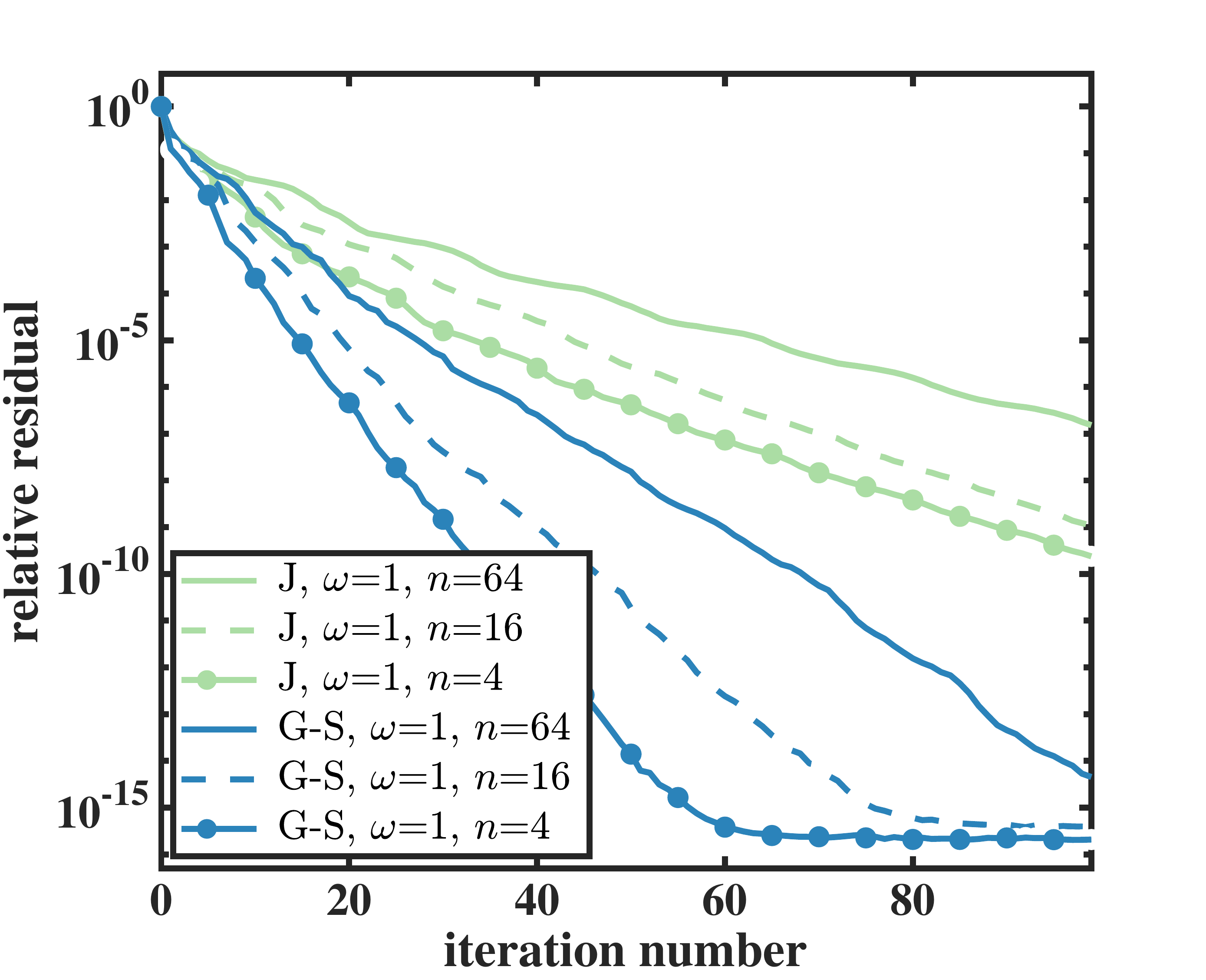}
\caption{$\relaxation = 1$; Anderson acceleration, $\memory=5$}
\end{subfigure}
\caption{\captionFirst Strong-scaling study: convergence results. J and G-S
	labels refer to the Jacobi and Gauss--Seidel methods, respectively. Note
	that convergence is faster when NetUQ employs smaller networks,
	Gauss--Seidel (rather than Jacobi), $\relaxation=1$ (rather than
	$\relaxation=2/3$), and Anderson acceleration.}
\label{fig:2dheat_strong_conv}
\end{figure}

Next, Fig.~\ref{fig:2dheat_strong_scalability} reports the dependence of
several performance quantities as a function of the number of components.
This figure reports results for NetUQ without Anderson acceleration, and for
the case where iterations are terminated when the relative residual reaches a
value of $10^{-3}$.  Consistent with Fig.~\ref{fig:2dheat_strong_conv},
Fig.~\ref{strongNoAAIt} shows that the lowest iteration count is achieved for
smaller networks, Gauss--Seidel, and a relaxation factor of $\relaxation=1$.

Fig.~\ref{fig:2dheat_strong_scalability_b} reports the associated parallel
wall times, employing a one-to-one component-to-processor map.  These results show that the trends suggested by
Fig.~\ref{fig:2dheat_strong_scalability} can be reversed when accounting for
the lower computational cost of solving the component deterministic problems
\eqref{eq:PDEsub} for smaller subdomains and the larger number of sequential
steps per iteration for the Gauss--Seidel method. Specifically, this figure
shows that the smallest parallel wall times are achieved for Jacobi iteration
and larger network sizes.  This suggests that the embarrassingly parallel
nature of each Jacobi iteration outweighs the larger number of iterations it
requires for convergence relative to Gauss--Seidel, and the lower
computational cost of solving component deterministic problems
\eqref{eq:PDEsub} with smaller subdomains outweighs the larger number of
iterations needed for convergence.


Next, Fig.~\ref{fig:2dheat_strong_scalability_c} reports the speedup as a
function of the number of components, where the speedup is defined as the
ratio of serial execution time to parallel execution time. Here, we consider
the serial execution time to be the wall time required to solve the global
uncertainty-propagation problem described in Section \ref{sec:globUPprob}
using one computing core. We observe that---for four components---the
relatively large number of iterations required for convergence outweighs gains
achieved through parallelization, resulting in a speedup less than one.
Speedups greater than one are achieved only for $\nsubsystems=16$ in the case
of Jacobi with $\relaxation=1$ and for $\nsubsystems=64$ in the case of Jacobi
and Gauss--Seidel with $\relaxation=1$. However, we emphasize that the purpose
of NetUQ is not necessarily to achieve speedups in this particular context,
i.e., where solving the global uncertainty-propagation problem is feasible. As
described in the introduction, the true objective of NetUQ is to enable
full-system UQ in scenarios where it would be otherwise impossible due to
challenges in integrating component models, for example. Nonetheless, we show
in the next figure that substantial speedups can indeed be achieved when
Anderson acceleration is employed.

Next, Fig.~\ref{fig:2dheat_strong_scalability_d} quantifies the error incurred
by the network formulation. Recall from Section \ref{sec:error} that the
network formulation can incur error if the uncertainty-propagation operator
$\propagatorRV$ constitutes an approximation of an underlying ``truth''
operator $\propagatorTruthRV$. This is precisely the case in this context,
where the truth operator $\propagatorTruthRV$ associates with performing
uncertainty-propagation with infinite-dimensional PCE representations for
output random variables and endogenous-input random variables; this ``no edge
truncation'' case is mathematically equivalent to simply solving the global
uncertainty-propagation problem described in Section \ref{sec:globUPprob}.
Thus, to assess these errors, we compute the relative error in the PCE
coefficients of the field random variable at several spatial locations
computed using the NetUQ approach (executed with 3rd-order total-degree
truncation and satisfying a relative residual of $10^{-10}$) with respect to
the same PCE coefficients computed by solving the global
uncertainty-propagation problem. This relative error is computed as the
$\ell^2$-norm of the difference between these PCE coefficients, divided by the
$\ell^2$-norm of the PCE coefficients computed via global uncertainty
propagation.  Fig.~\ref{fig:2dheat_strong_scalability_d} shows that this error
increases as the number of components increases; this is sensible, as a larger
number of components implies more spatial locations where truncation is
imposed. The figure also shows that this error is larger for variables located
further from the domain boundary; this is also intuitive, as it suggests that
errors grow as the variable of interest is located further from from
one of the driving exogenous inputs.


\begin{figure}[h!]
\centering
\begin{subfigure}{0.4\textwidth}
	\centering
	\includegraphics[width=0.95\textwidth]{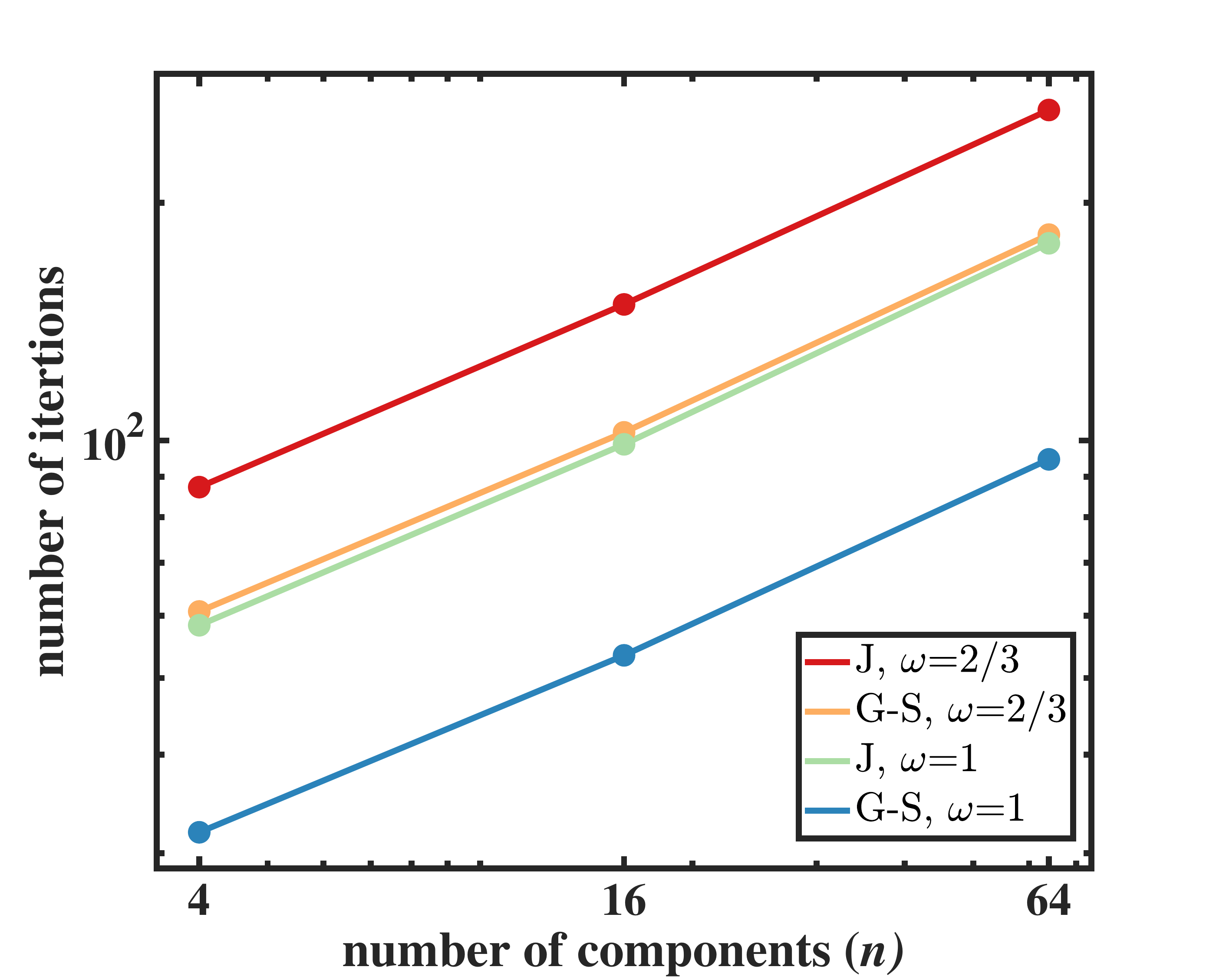}
	\caption{Number of iterations v.\ number of components}
	\label{strongNoAAIt}
\end{subfigure}
\begin{subfigure}{0.4\textwidth}
\centering
\includegraphics[width=0.95\textwidth]{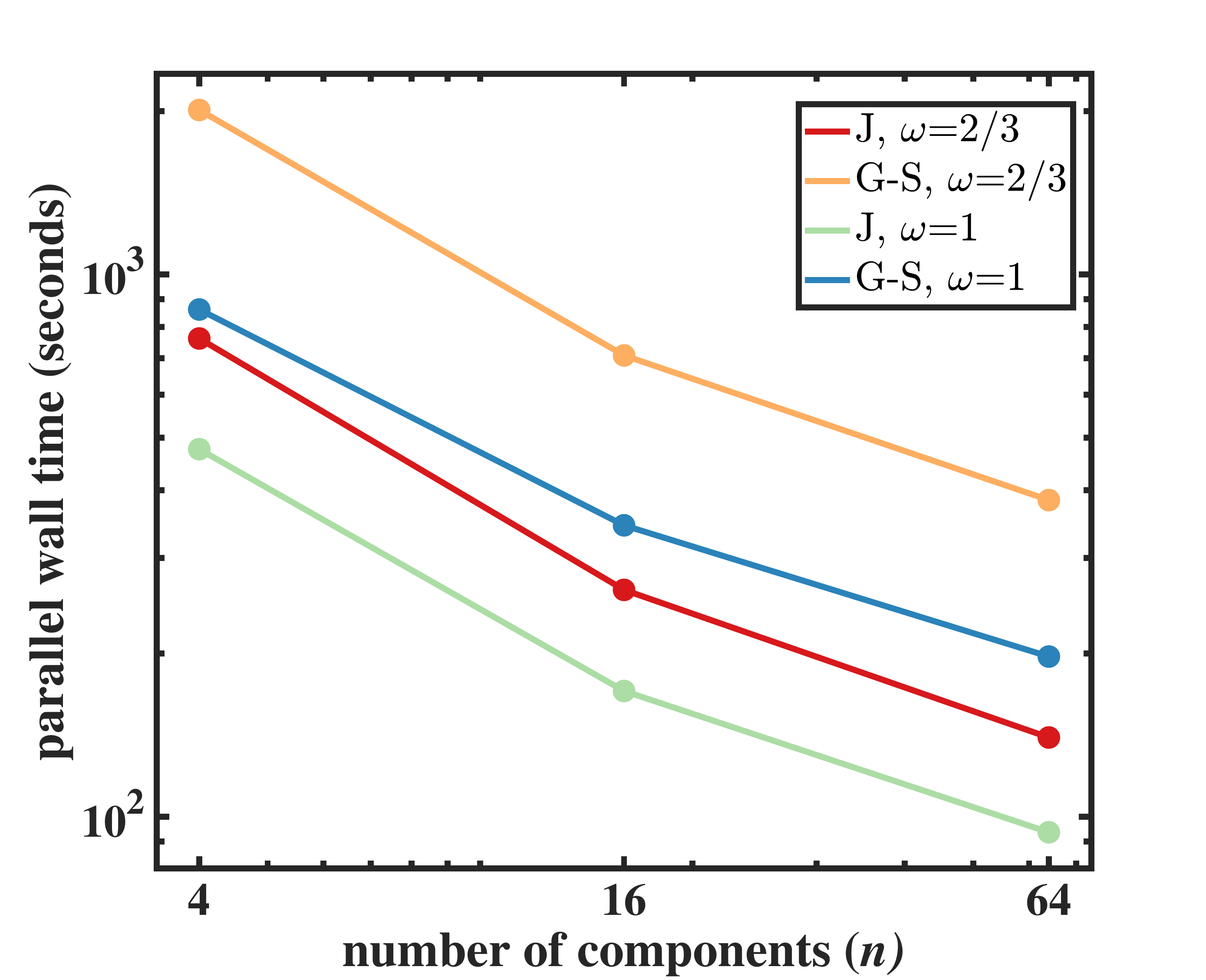}
\caption{Parallel wall time v.\ number of components}
\label{fig:2dheat_strong_scalability_b}
\end{subfigure}
\begin{subfigure}{0.4\textwidth}
\centering
\includegraphics[width=0.95\textwidth]{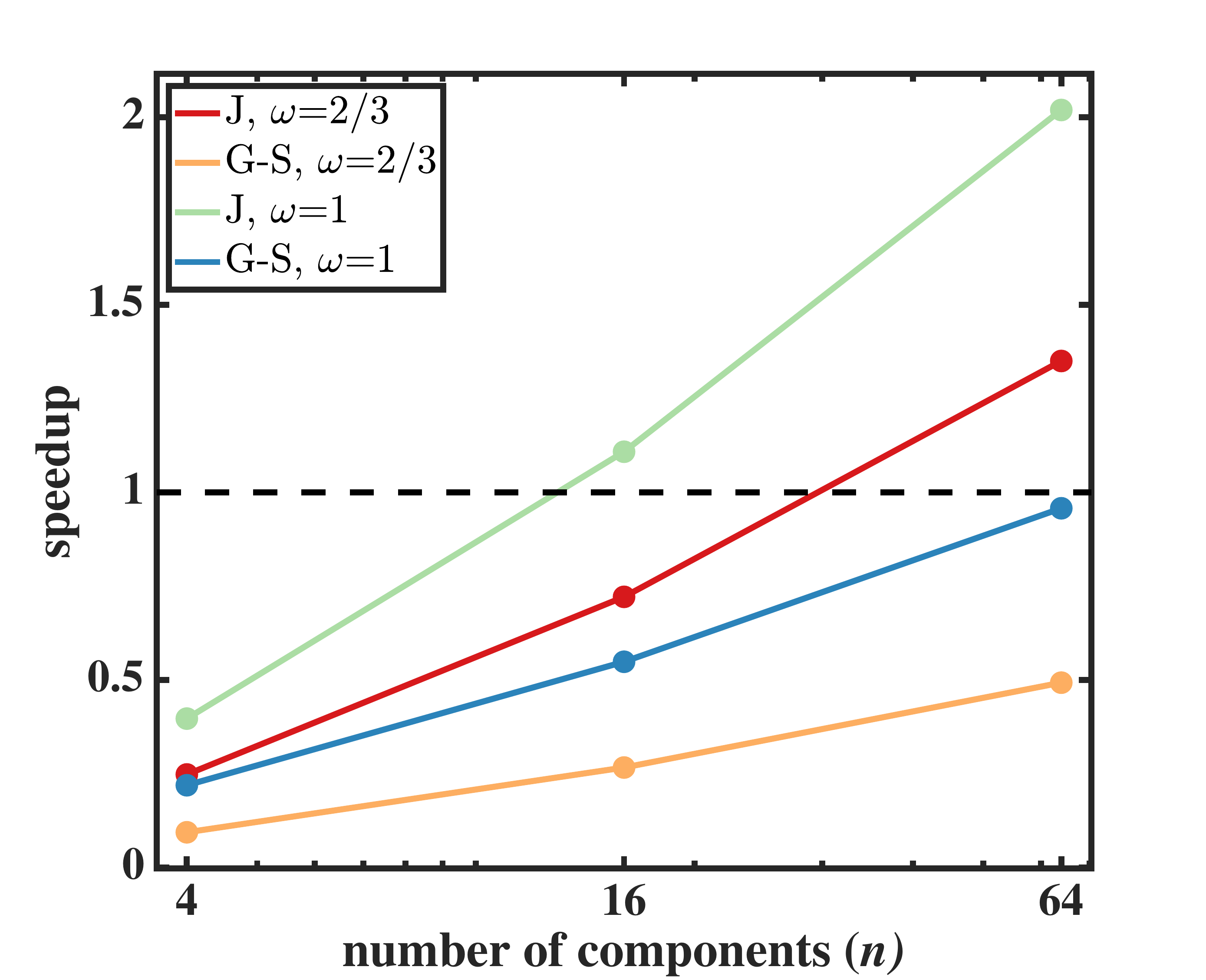}
\caption{Speedup v.\ number of components}
\label{fig:2dheat_strong_scalability_c}
\end{subfigure}
\begin{subfigure}{0.4\textwidth}
\centering
\includegraphics[width=0.95\textwidth]{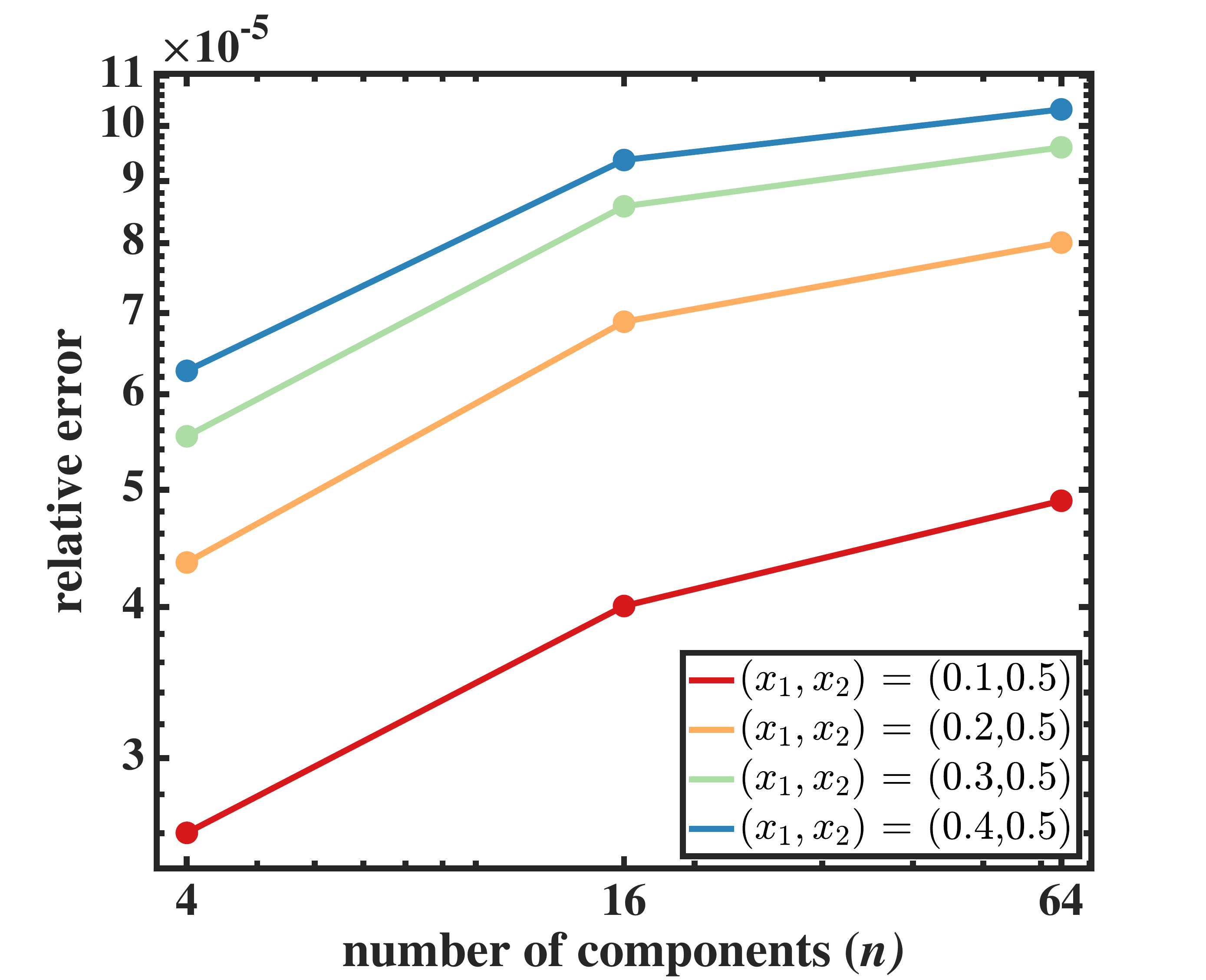}
\caption{Relative error v.\ number of components}
\label{fig:2dheat_strong_scalability_d}
\end{subfigure}
\caption{
\captionFirst	Strong-scaling study without
	Anderson acceleration. Iterations are terminated when the relative
	residual reaches a value of $10^{-3}$.
J and G-S
	labels refer to the Jacobi and Gauss--Seidel methods, respectively.
	Relative error refers to the
	error in the PCE
coefficients of the field random variable at several spatial locations
computed using the NetUQ approach (executed with 3rd-order total-degree
truncation and satisfying a relative residual of $10^{-10}$) with respect to
the same PCE coefficients computed by solving the global
uncertainty-propagation problem.}
\label{fig:2dheat_strong_scalability}
\end{figure}

Fig.~\ref{fig:2dheat_strong_scalability_AA} repeats the same study, but with
all relaxation methods using Anderson acceleration with $\memory=5$. Note that
we need not repeat the error-quantification study, as these results are
related to
the network formulation itself, and are independent of the relaxation method
used to numerically solve the network uncertainty-propagation problem.
Through a comparison with the corresponding figures in
Fig.~\ref{fig:2dheat_strong_scalability}, it is clear that Anderson acceleration
yields roughly an order-of-magnitude reduction in the number of iterations
required for convergence.  This in turn implies substantially lower parallel
wall times, as well as significantly larger speedups. Indeed, when Anderson
acceleration is employed, the NetUQ method realizes speedups as high as nearly
20 in the case of $\nsubsystems=16$ components and the Jacobi method.  One
noteworthy observation is that the maximum speedup corresponds to
the 16-component network. Since this problem is relatively small (i.e., an FEM
discretization with 1681 nodes), decomposing the domain into
smaller domains beyond a certain threshold (16 subdomains in this case) does not sufficiently reduce the
time required for each component uncertainty propagation to overcome the
increase in iteration count for the algorithm to converge. We also note that
Anderson acceleration has the effect of reducing the performance discrepancy
obtained using relaxation factors of $\relaxation=2/3$ and $\relaxation = 1$,
especially for the Jacobi method.

\begin{figure}[h!]
\centering
\begin{subfigure}{0.4\textwidth}
	\centering
	\includegraphics[width=0.95\textwidth]{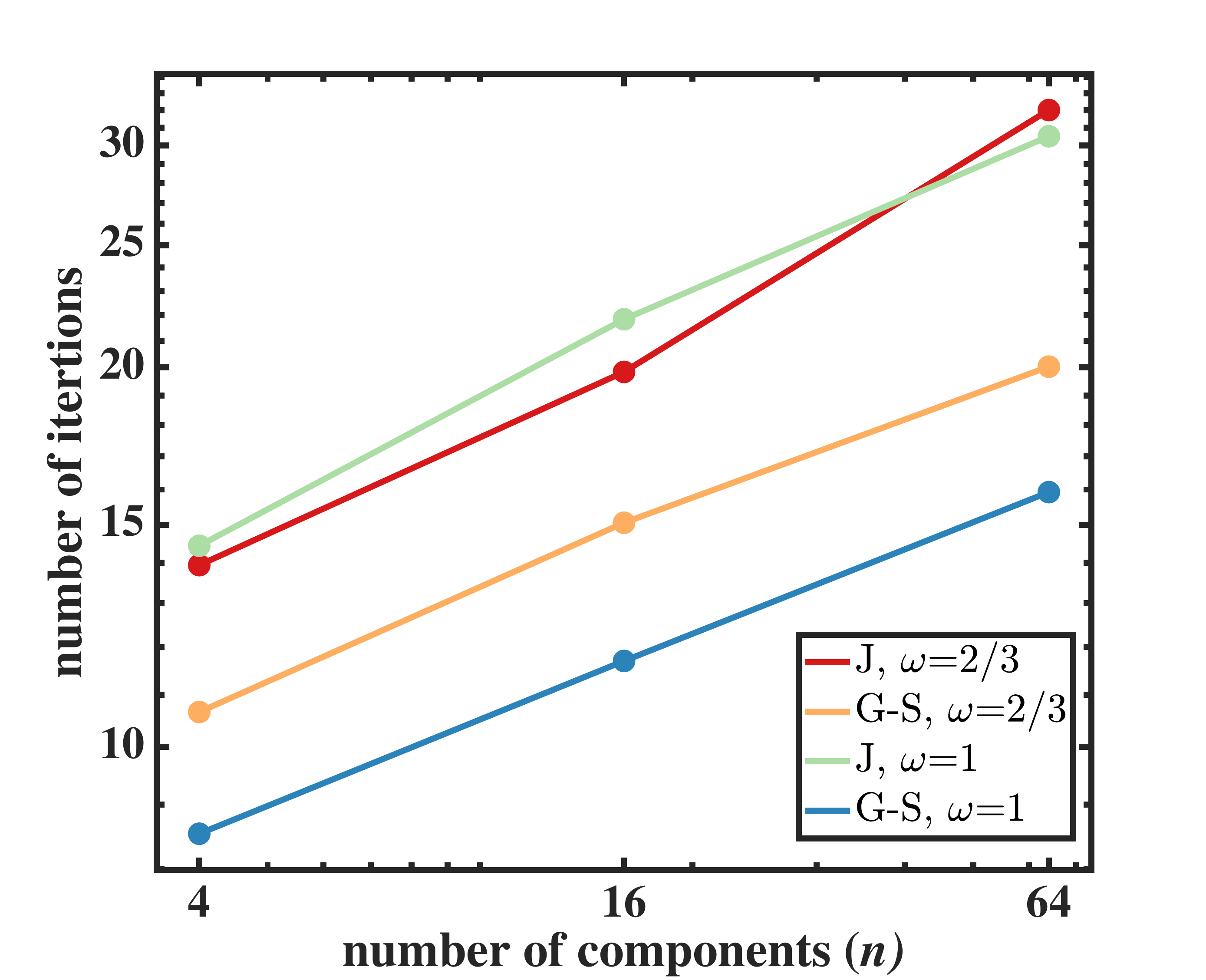}
	\caption{Number of iterations v.\ number of components}
\end{subfigure}
\begin{subfigure}{0.4\textwidth}
	\centering
	\includegraphics[width=0.95\textwidth]{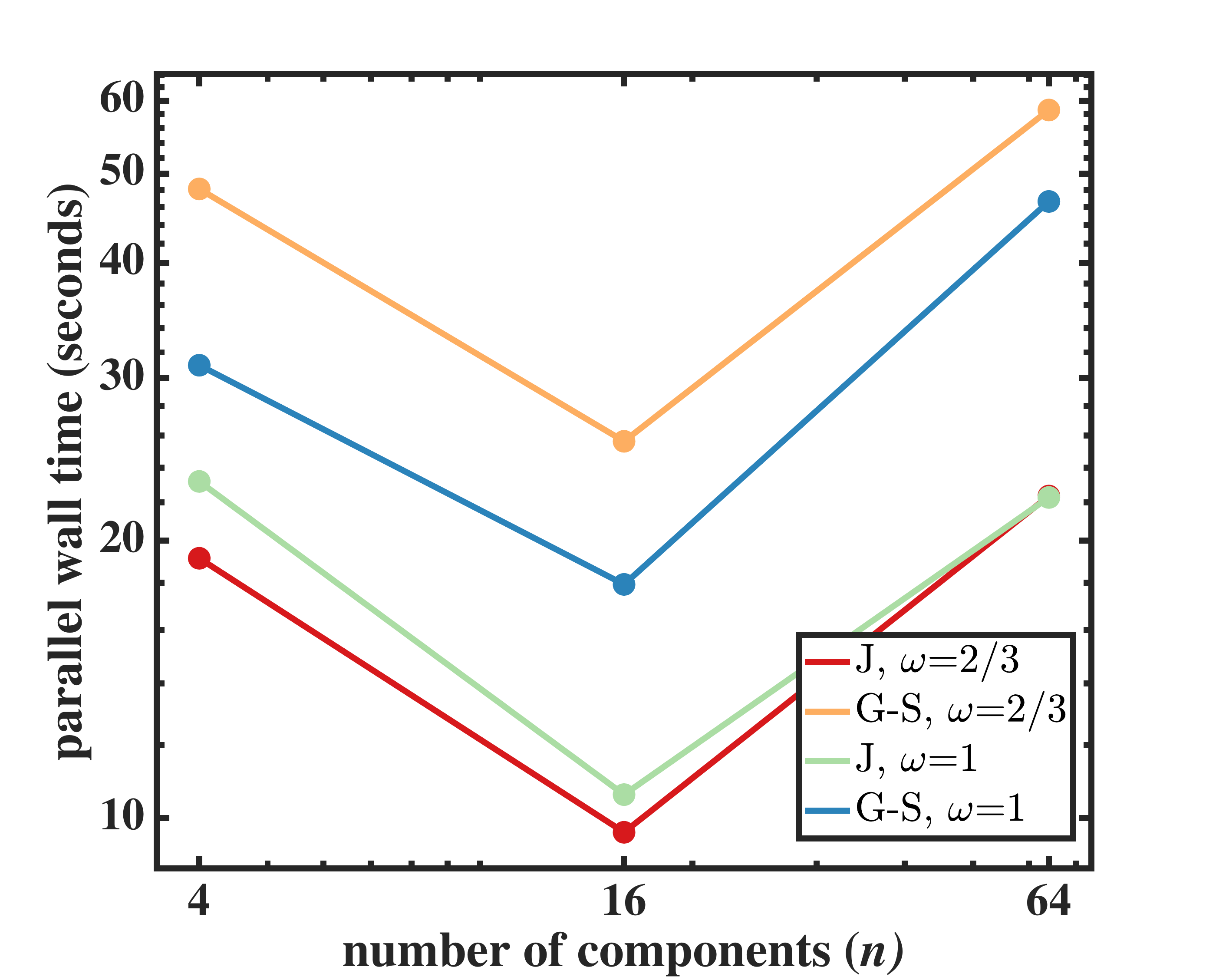}
	\caption{Parallel wall time v.\ number of components}
\end{subfigure}
\begin{subfigure}{0.4\textwidth}
	\centering
	\includegraphics[width=0.95\textwidth]{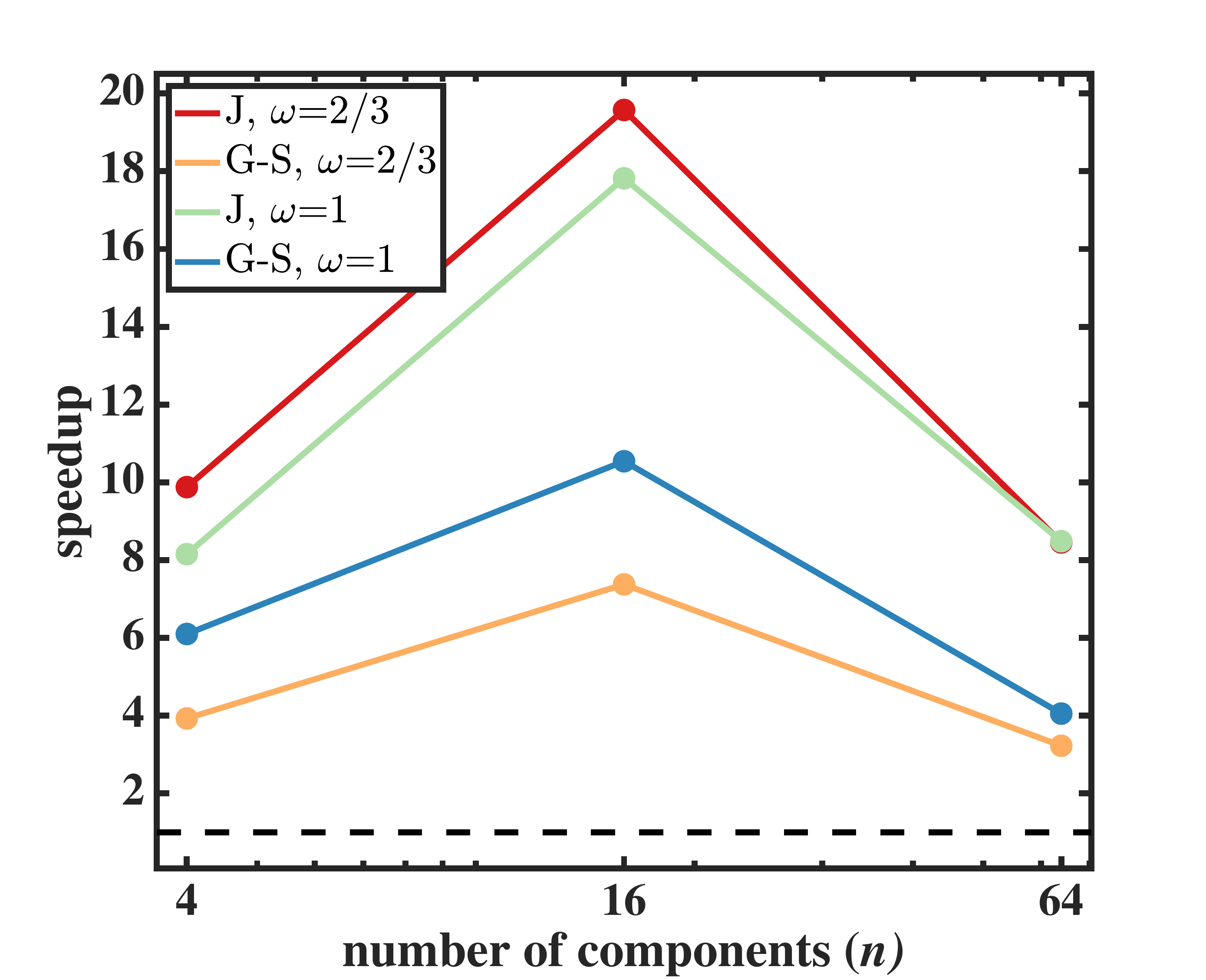}
	\caption{Speedup v.\ number of components}
\end{subfigure}
\caption{
\captionFirst	 Strong-scaling study using
	Anderson acceleration with $\memory = 5$.
Iterations are terminated when the relative
	residual reaches a value of $10^{-3}$.
J and G-S
	labels refer to the Jacobi and Gauss--Seidel methods, respectively.}
\label{fig:2dheat_strong_scalability_AA}
\end{figure}

Finally, to assess the effect of the permutation $\permutationTuple$ on
the performance of the Gauss--Seidel method,
Fig.~\ref{fig:2dheat_strong_scalability_perms_AA} reports strong-scaling
results for Gauss--Seidel with $\relaxation=1$, Anderson acceleration, and ten
random permutations $\permutationTuple$. Recall from Section \ref{sec:gaussseidel} that the permutation
effectively defines the DAG on which Gauss--Seidel propagates
uncertainties in a feed-forward manner within each iteration.  Figure
\ref{fig:2dheat_strong_scalability_perms_it_AA} shows that---for this
problem---the choice of permutation has essentially no effect on the
number of iterations required for convergence. However, Figure
\ref{fig:2dheat_strong_scalability_perms_nseq_AA} shows that the choice of
permutation does have a substantial effect on the number of sequential
sequential steps per iteration $\nsequential$ as computed using Algorithm
\ref{alg:DAG}, which in turn yields significant differences in
parallel wall time (Fig.~\ref{fig:2dheat_strong_scalability_perms_time_AA})
and speedup (Fig.~\ref{fig:2dheat_strong_scalability_perms_speedup_AA}). Thus,
for this problem, the permutation should be chosen to minimize the
number of sequential steps per iteration $\nsequential$.  Note that virtually
no speedup is observed even when using Anderson acceleration for the worst
performing permutation. Thus further emphasizes the need to carefully
consider the choice of permutation in practice.

\begin{figure}[h!]
\centering
\begin{subfigure}{0.4\textwidth}
	\centering
	\includegraphics[width=0.95\textwidth]{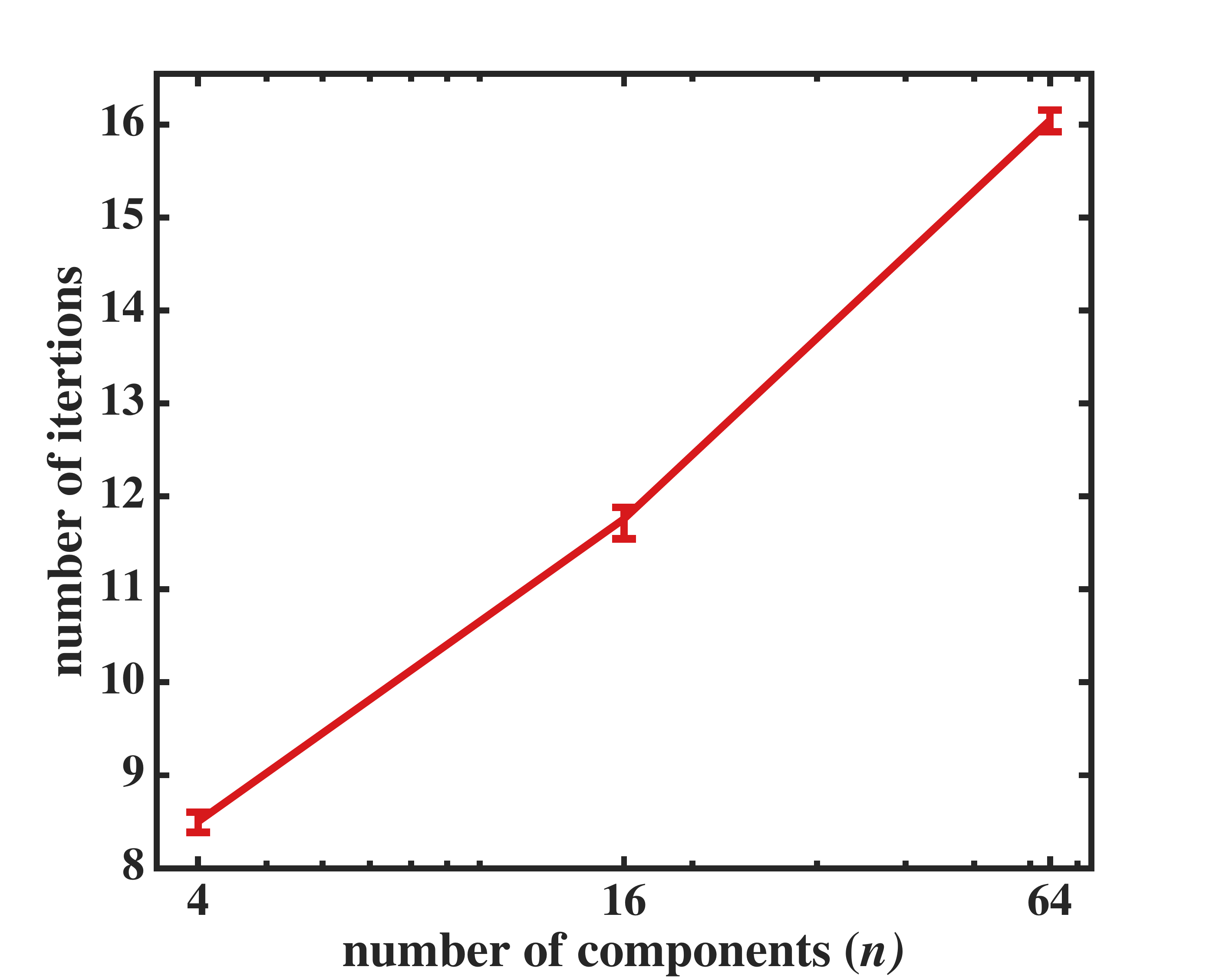}
	\caption{Number of iterations v.\ number of components}
\label{fig:2dheat_strong_scalability_perms_it_AA}
\end{subfigure}
\begin{subfigure}{0.4\textwidth}
	\centering
	\includegraphics[width=0.95\textwidth]{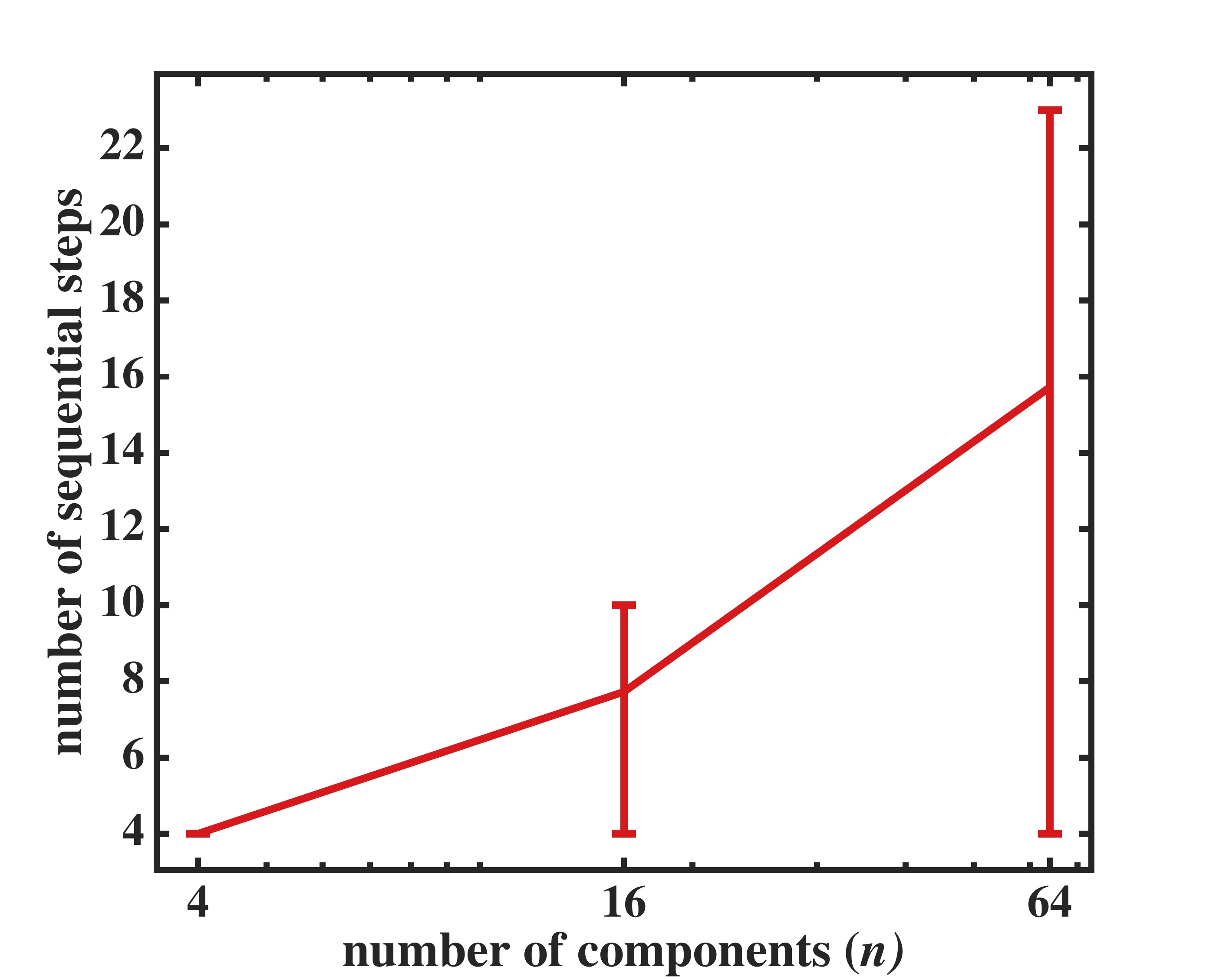}
	\caption{Number of sequential steps per iteration v.\ number of components}
\label{fig:2dheat_strong_scalability_perms_nseq_AA}
\end{subfigure}
\begin{subfigure}{0.4\textwidth}
	\centering
	\includegraphics[width=0.95\textwidth]{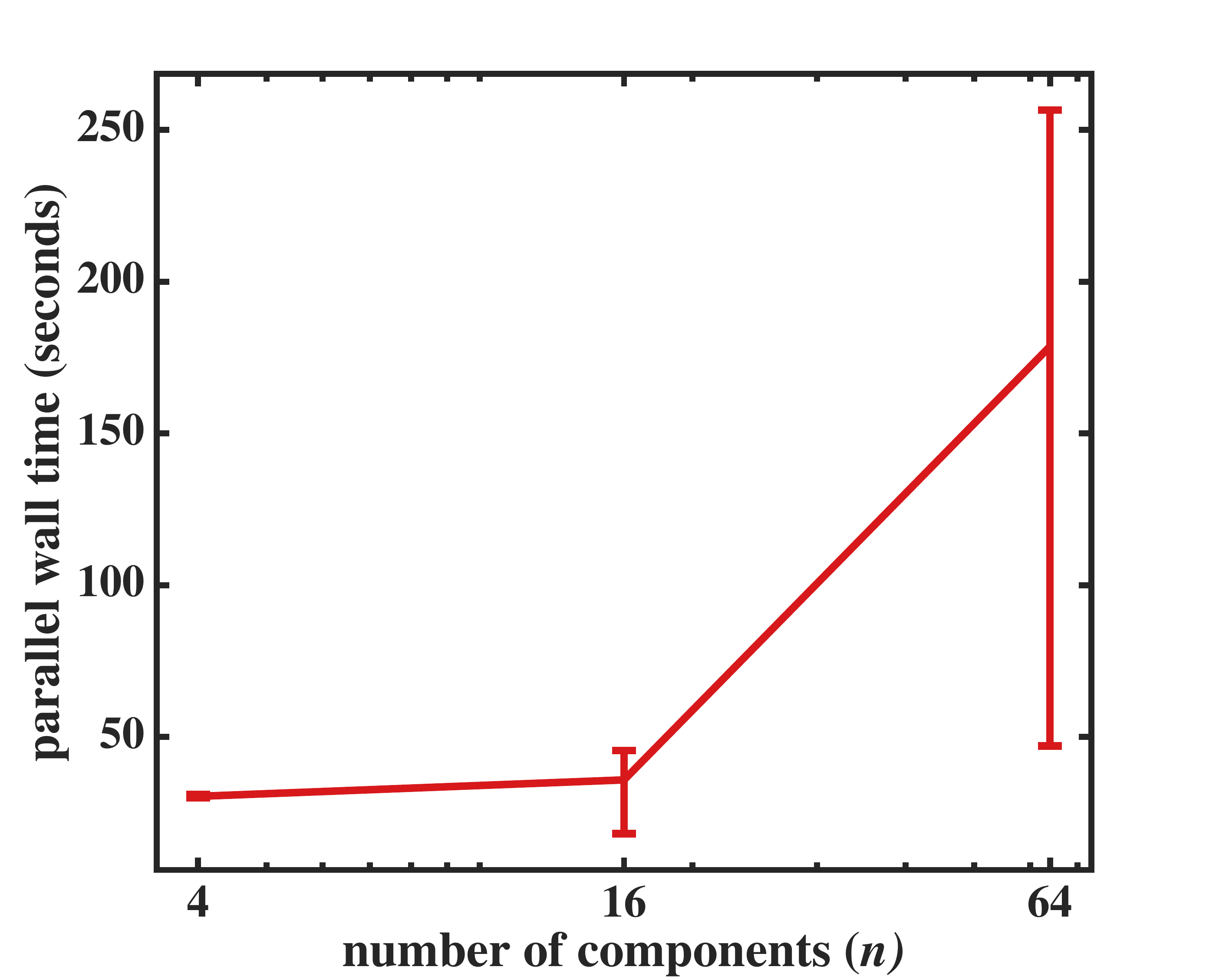}
	\caption{Parallel wall time v.\ number of components}
\label{fig:2dheat_strong_scalability_perms_time_AA}
\end{subfigure}
\begin{subfigure}{0.4\textwidth}
	\centering
	\includegraphics[width=0.95\textwidth]{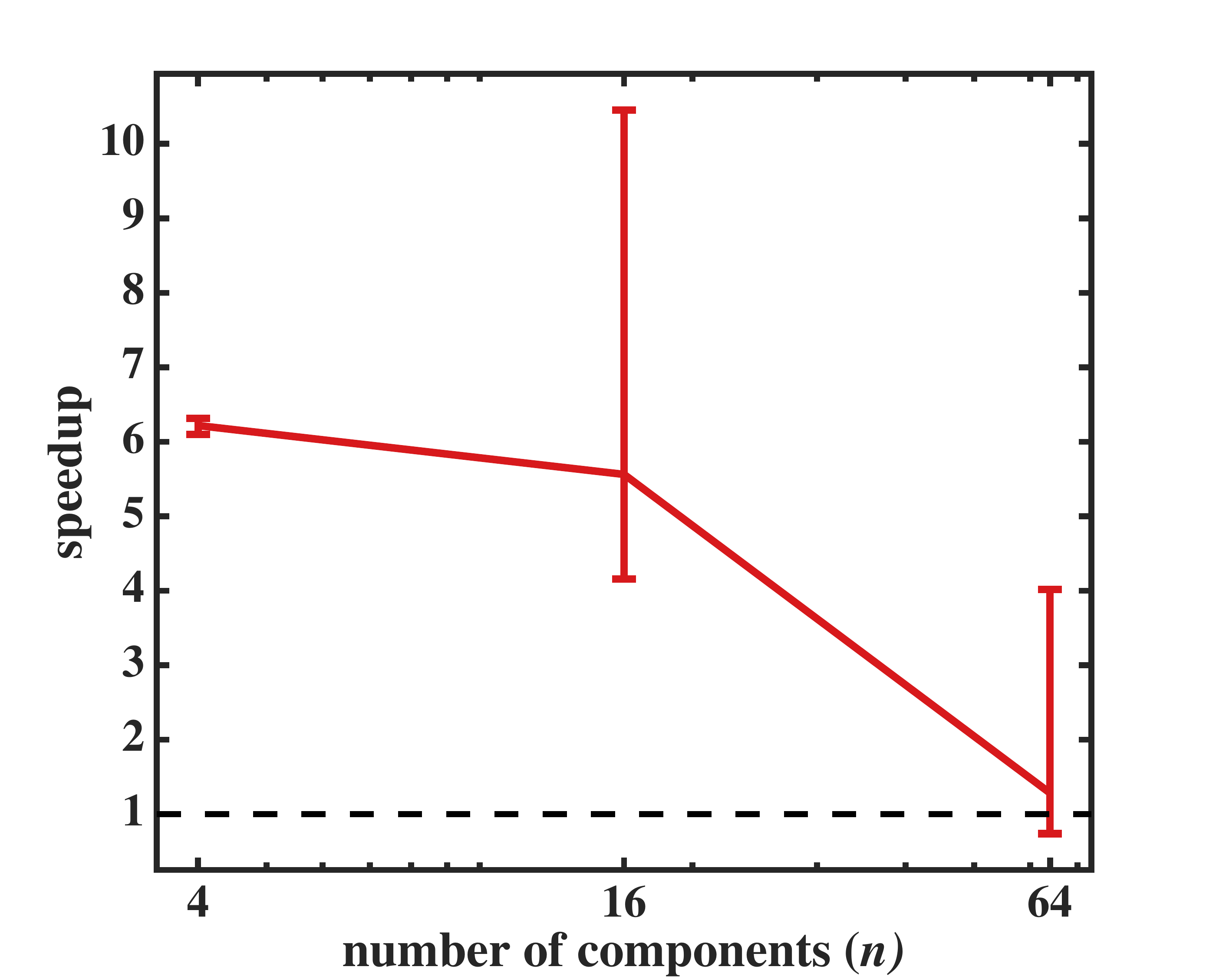}
	\caption{Speedup v.\ number of components}
\label{fig:2dheat_strong_scalability_perms_speedup_AA}
\end{subfigure}
\caption{\captionFirst Strong-scaling study using
	Anderson acceleration with $\memory = 5$.
	Gauss--Seidel with $\relaxation=1$ performance for ten random permutations
	$\permutationTuple$.
	Red curve represents the mean value, and vertical lines span the
	minimum and maximum values computed over these permutations.
Note that the permutation has minimal effect on the number of
	iterations required for convergence, and thus performance is driven
	by the number of sequential steps per iteration $\nsequential$
	that the permutation admits.
	}
\label{fig:2dheat_strong_scalability_perms_AA}
\end{figure}


\subsection{Weak-scaling study}\label{sec:weak}

We now investigate the weak-scaling performance of NetUQ.  In the context of
NetUQ, weak scaling associates with the case where a fixed library of
components on which uncertainty propagation can be performed is assembled into
larger and larger full systems.
In particular, we
adopt the same strategy for constructing networks as in Section
\ref{sec:strong}, with one exception: each component deterministic BVP is
discretized using a mesh composed of 36  grid points, regardless of the total
number of components. As such, the size of each component deterministic 
problem remains fixed for all networks. However, as the number of components
increases, the size of the overlap region between components decreases, and size of the
deterministic global problem increases.  We note that the weak-scaling case
with $\nsubsystems =64$ is equivalent to the strong-scaling case with
$\nsubsystems=64$, and thus the results for these cases are identical.

Fig.~\ref{fig:2dheat_weak_conv} reports convergence results, where---as in the
strong-scaling case---we use a permutation $\permutationTuple$ for
Gauss--Seidel that yields the minimum number of sequential steps per iteration
of $\nsequential=4$.  Comparing with Fig.~\ref{fig:2dheat_strong_conv}, we see
that---as in the strong-scaling case---convergence is faster for smaller
networks, Gauss--Seidel, $\relaxation=1$, and Anderson
acceleration.	The explanation behind these trends is the same as in the
strong-scaling case. However, we observe that convergence is substantially
faster (as a function of iteration number) for $\nsubsystems=4$ and
$\nsubsystems=16$ in the weak-scaling case compared with the
strong-scaling case. This is due to the fact that---for these network
sizes---the weak-scaling problem corresponds to a larger amount of overlap between
neighboring components than the strong-scaling problem. Increasing the amount
of overlap is known to promote convergence in domain decomposition, which is
an effect we observe here.

\begin{figure}[h!]
	\centering
	\begin{subfigure}{0.4\textwidth}
		\centering
		\includegraphics[width=0.95\textwidth]{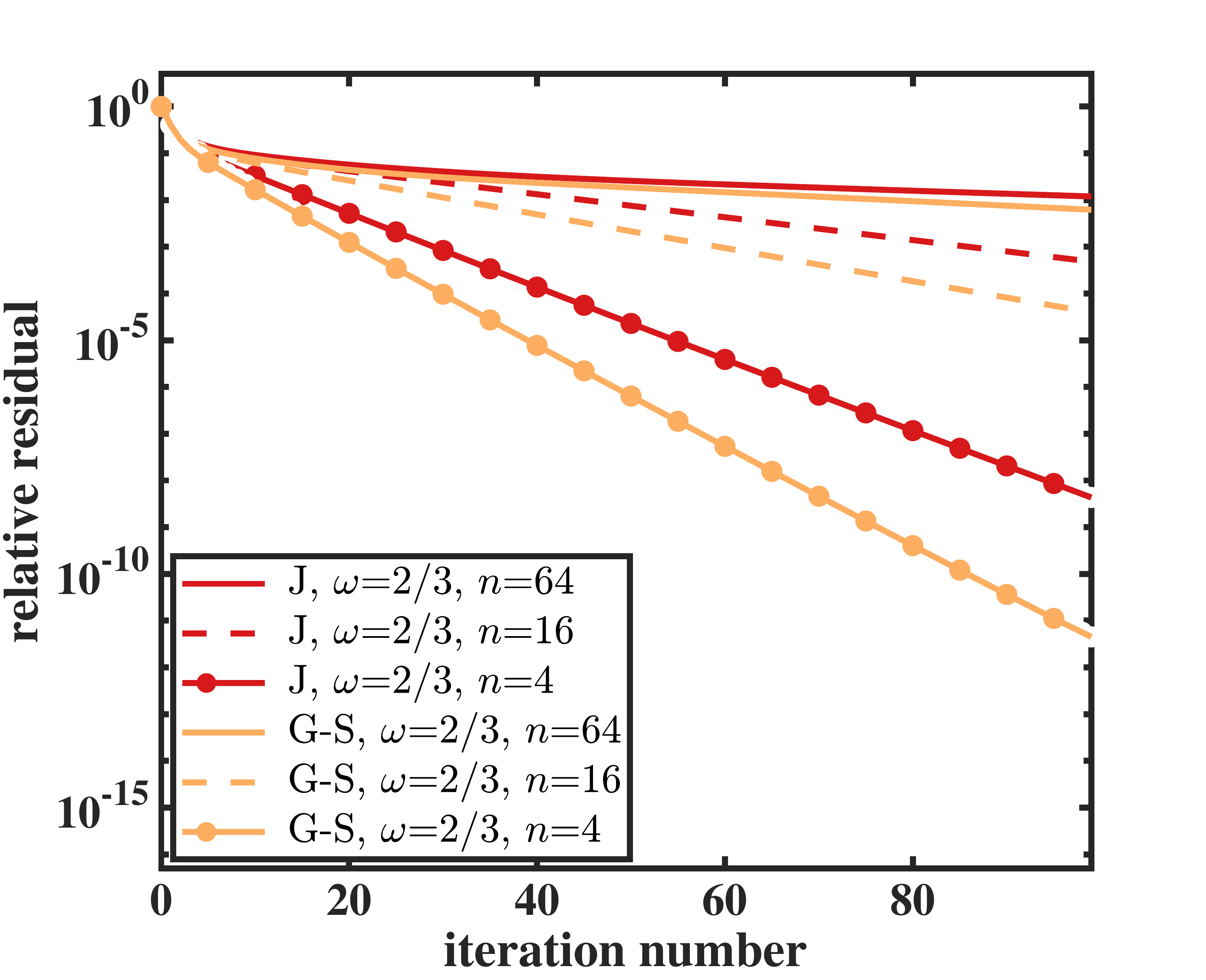}
		\caption{$\relaxation = 2/3$; no Anderson acceleration}
	\end{subfigure}
	\begin{subfigure}{0.4\textwidth}
		\centering
		\includegraphics[width=0.95\textwidth]{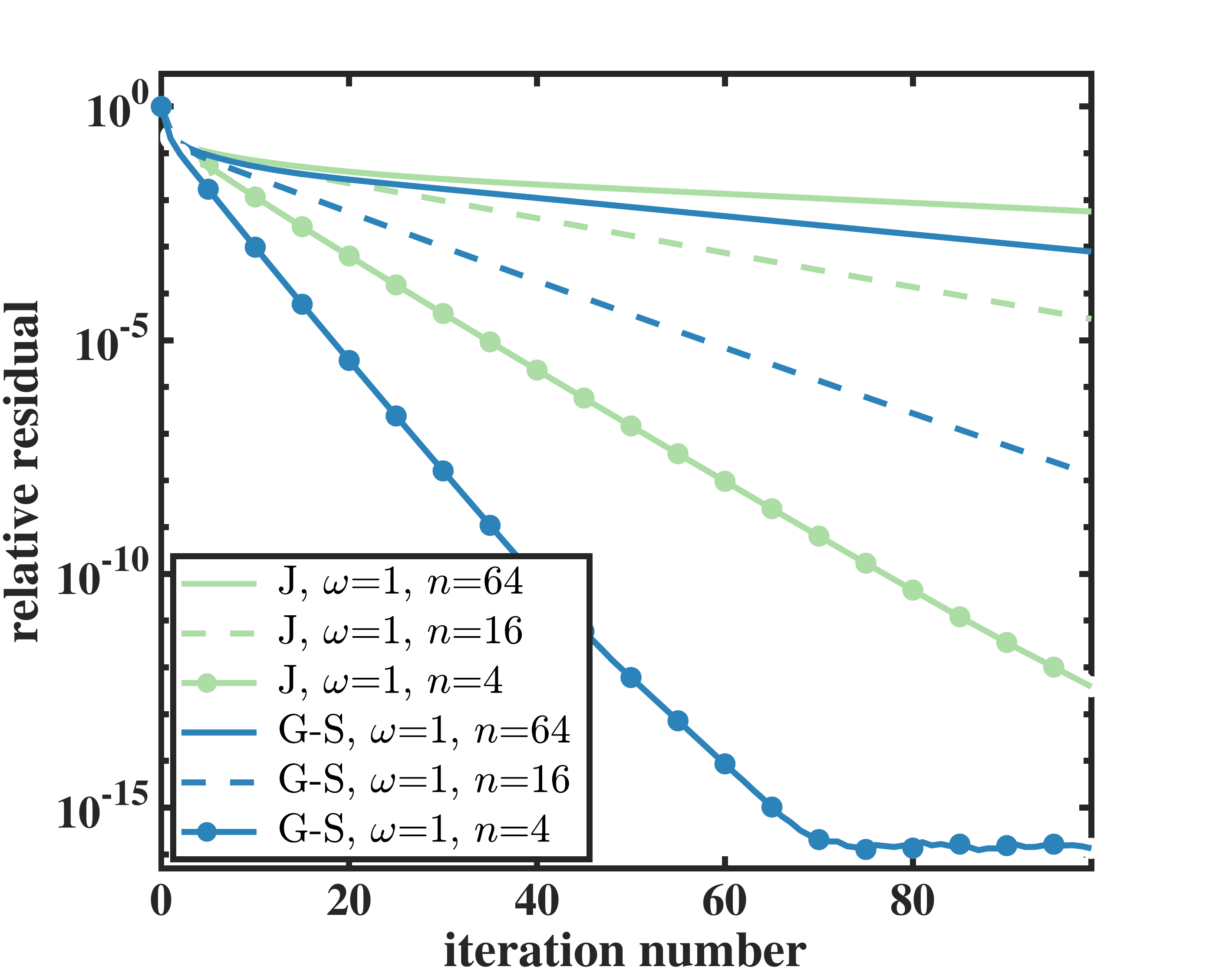}
		\caption{$\relaxation = 1$; no Anderson acceleration}
	\end{subfigure}
	\begin{subfigure}{0.4\textwidth}
		\centering
		\includegraphics[width=0.95\textwidth]{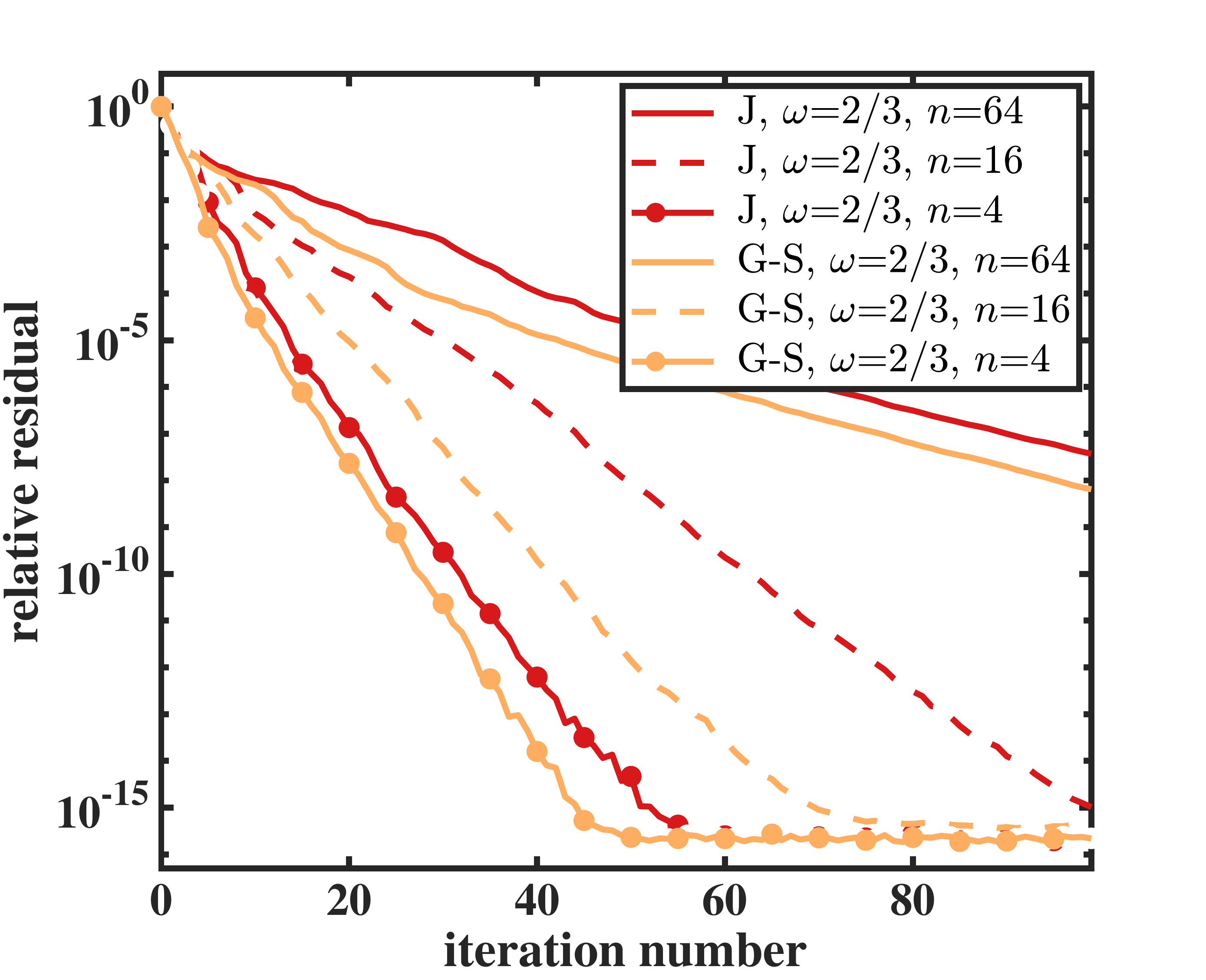}
		\caption{$\relaxation = 2/3$; Anderson acceleration, $\memory=5$}
	\end{subfigure} \begin{subfigure}{0.4\textwidth}
		\centering
		\includegraphics[width=0.95\textwidth]{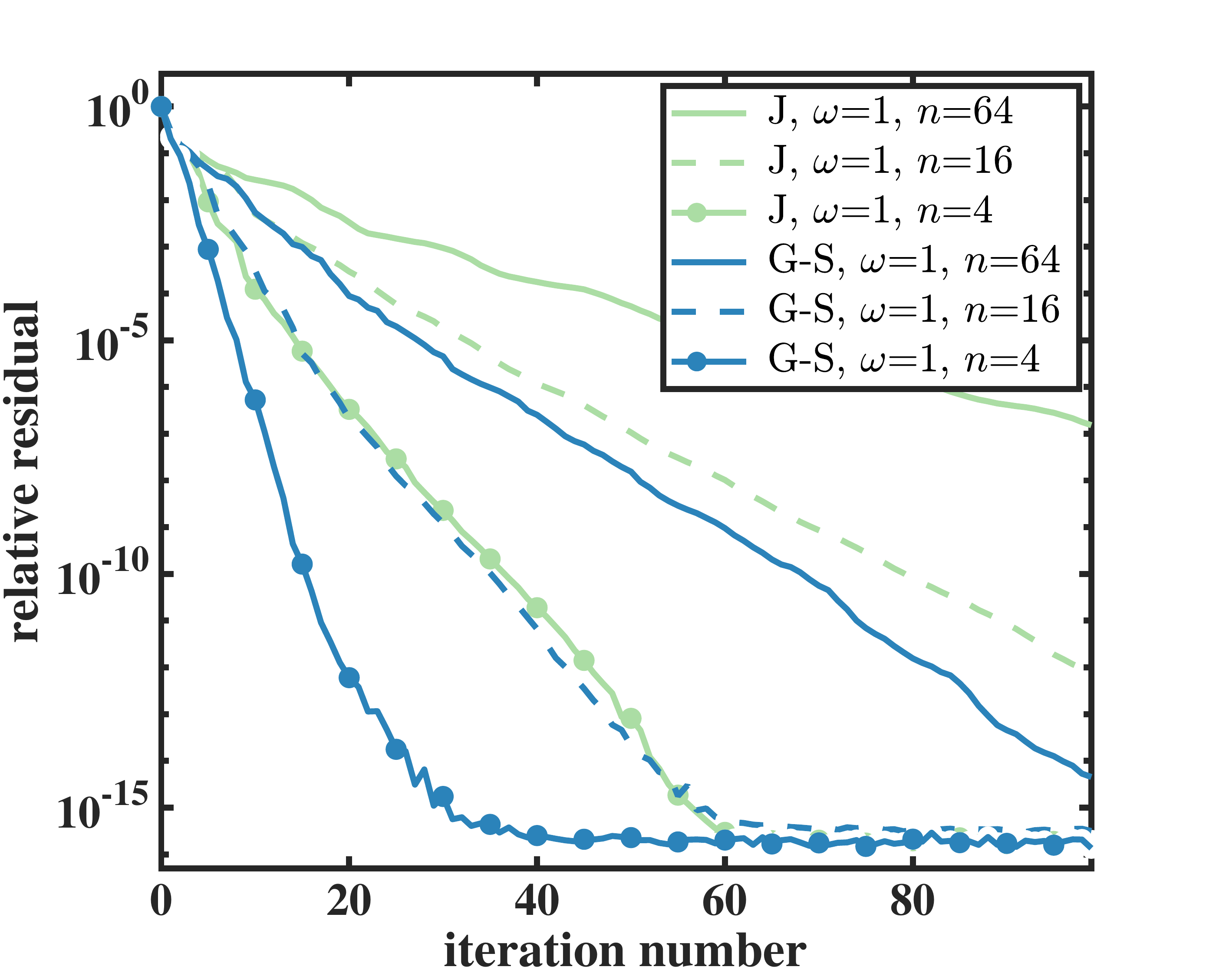}
		\caption{$\relaxation = 1$; Anderson acceleration, $\memory=5$}
	\end{subfigure}
	\caption{
	\captionFirst Weak-scaling study: convergence results. J and G-S
	labels refer to the Jacobi and Gauss--Seidel methods, respectively. Note
	that convergence is faster when NetUQ employs smaller networks,
	Gauss--Seidel (rather than Jacobi), $\relaxation=1$ (rather than
	$\relaxation=2/3$), and Anderson acceleration.
	}
	\label{fig:2dheat_weak_conv}
\end{figure}

Next, Fig.~\ref{fig:2dheat_weak_scalability} reports performance as a function
of the number of components. As in Fig.~\ref{fig:2dheat_strong_scalability},
results correspond to NetUQ without Anderson acceleration, and for the case
where iterations are terminated when the relative residual reaches a value of
$10^{-3}$.  Comparing Figs.~\ref{strongNoAAIt} and \ref{strongNoAAIt_weak}
shows that---in both cases---the number of iterations increases as the number
of components increases. This is due to the fact that more iterations are
needed for information to propagate throughout the domain when it is
decomposed into more components. However, we can see that the weak-scaling
case yields substantially fewer iterations than the strong-scaling case for
$\nsubsystems = 4$ and $\nsubsystems=16$; this occurs due to the increased
amount of overlap in the weak-scaling case as previously discussed.

Fig.\ \ref{fig:2dheat_weak_scalability_b} reports wall times, again using a
one-to-one component-to-processor map. Comparing with
Fig.~\ref{fig:2dheat_strong_scalability_b} shows that weak and strong scaling
exhibit opposite trends. This is sensible because---unlike in strong scaling
where the component deterministic problem decreases in size for larger
networks---the deterministic component-problem size remains constant for all
network sizes. Thus, the parallel wall time is driven purely by iteration
count in the case of weak scaling. However, as in the strong-scaling case, we
again see that the Jacobi method yields superior parallel wall-time
performance than the Gauss--Seidel method despite consuming more iterations;
this is due to the embarrassingly parallel nature of Jacobi iterations.

Fig.~\ref{fig:2dheat_weak_scalability_c} reports the speedup as a function of
the number of components, where the speedup is defined as the ratio of serial
execution time to parallel execution time, where the serial execution time
corresponds to the cost of solving the global uncertainty-propagation
problem using the same (global) mesh as that corresponding to the NetUQ
problem.  As with the strong-scaling case reported in
Fig.~\ref{fig:2dheat_strong_scalability_c}, we again see that modest speedups
are achieved for Jacobi with the largest network size. While this might seem
surprising given the trends shown in Fig.\
\ref{fig:2dheat_weak_scalability_b}, it can be explained by the fact that the
global uncertainty-propagation problem incurs lower parallel wall time as the
number of components decreases in this case. This also explains
why the speedup increases at a slower rate with increasing network size as
compared with the strong-scaling case.

Finally, Fig.~\ref{fig:2dheat_weak_scalability_d} elucidates the same trends
as were observed in Fig.~\ref{fig:2dheat_strong_scalability_d} for the
strong-scaling case: the errors incurred by the NetUQ formulation are
extremely small, and increase slightly as the number of components increases.

\begin{figure}[h!]
	\centering
	\begin{subfigure}{0.4\textwidth}
		\centering
		\includegraphics[width=0.95\textwidth]{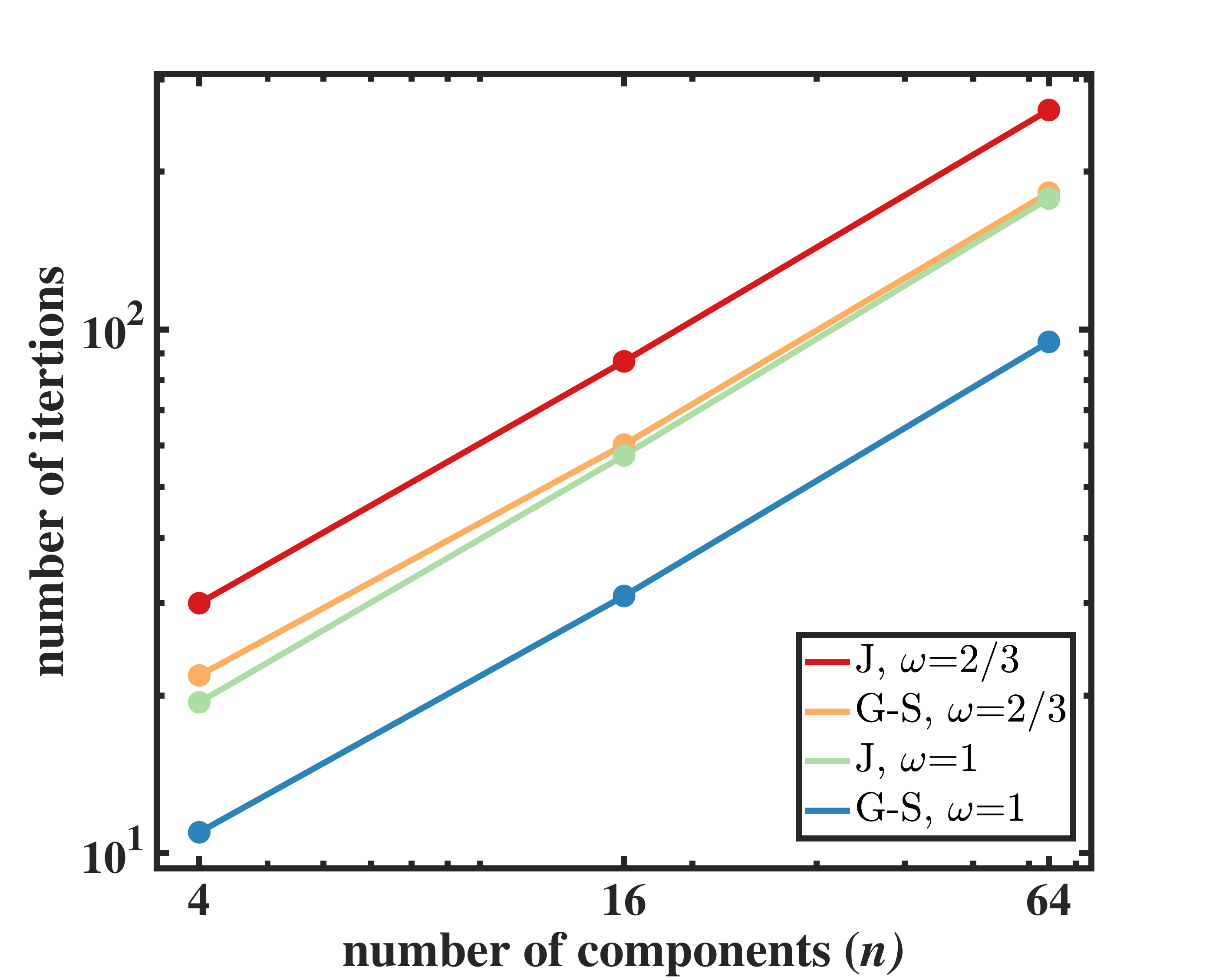}
		\caption{Number of iterations v.\ number of components}
	\label{strongNoAAIt_weak}
	\end{subfigure}
	\begin{subfigure}{0.4\textwidth}
		\centering
		\includegraphics[width=0.95\textwidth]{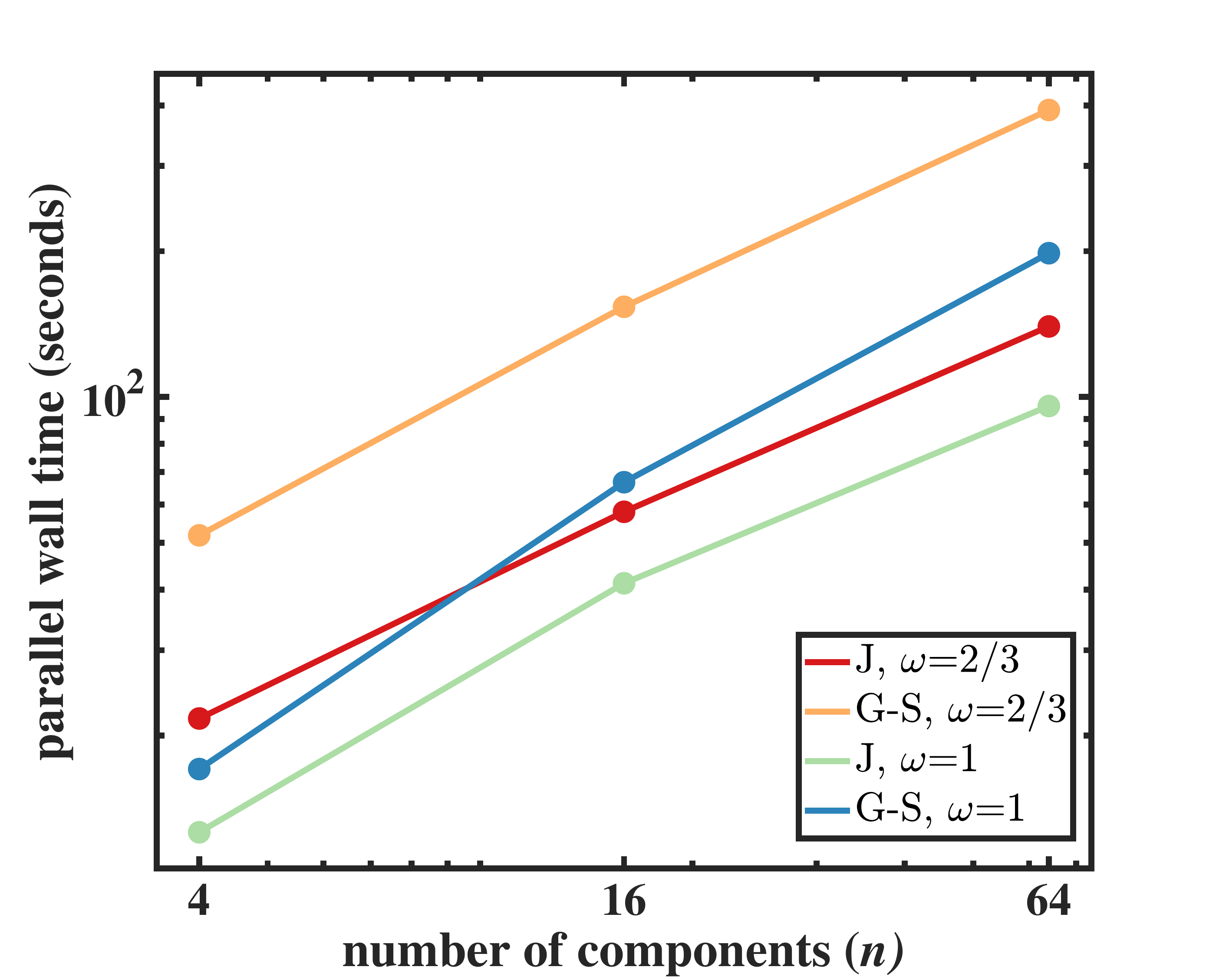}
		\caption{Parallel wall time v.\ number of components}
\label{fig:2dheat_weak_scalability_b}
	\end{subfigure}
	\begin{subfigure}{0.4\textwidth}
		\centering
		\includegraphics[width=0.95\textwidth]{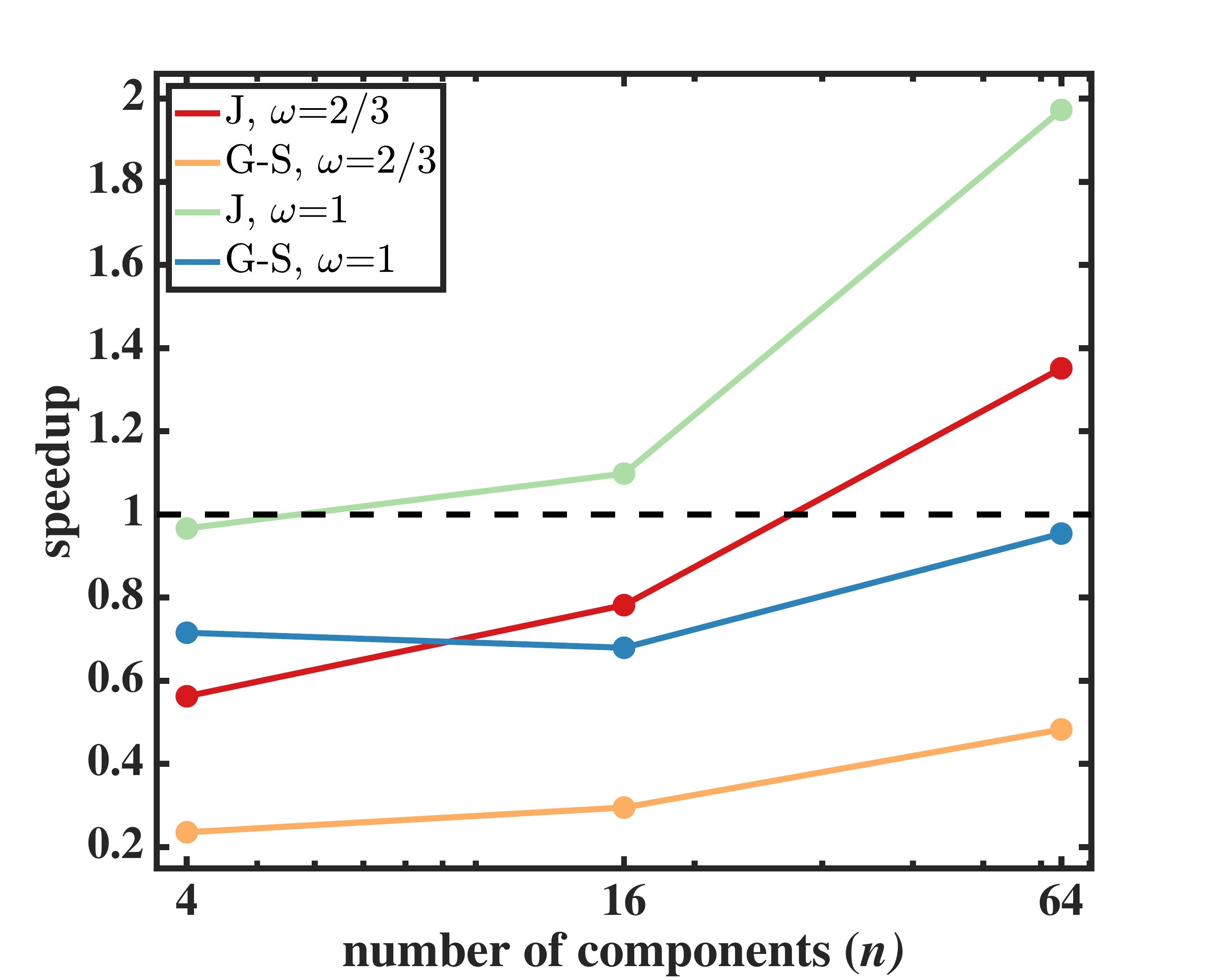}
		\caption{Speedup v.\ number of components}
\label{fig:2dheat_weak_scalability_c}
	\end{subfigure}
	\begin{subfigure}{0.4\textwidth}
		\centering
		\includegraphics[width=0.95\textwidth]{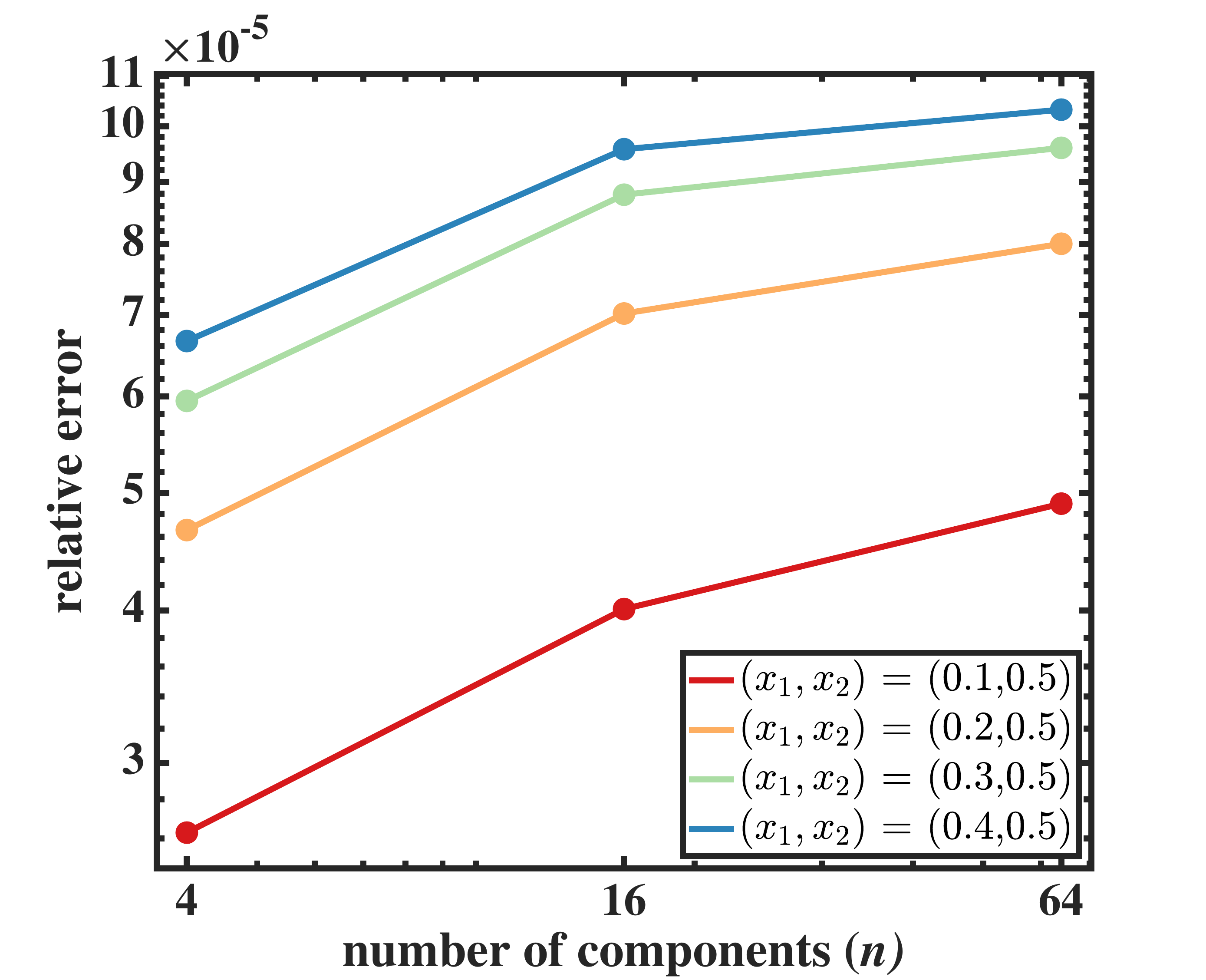}
		\caption{Relative error v.\ number of components}
\label{fig:2dheat_weak_scalability_d}
	\end{subfigure}
	\caption{
\captionFirst	Weak-scaling study without
	Anderson acceleration. Iterations are terminated when the relative
	residual reaches a value of $10^{-3}$.
J and G-S
	labels refer to the Jacobi and Gauss--Seidel methods, respectively.
	Relative error refers to the
	error in the PCE
coefficients of the field random variable at several spatial locations
computed using the NetUQ approach (executed with 3rd-order total-degree
truncation and satisfying a relative residual of $10^{-10}$) with respect to
the same PCE coefficients computed by solving the global
uncertainty-propagation problem.
}
	\label{fig:2dheat_weak_scalability}
\end{figure}

Fig.~\ref{fig:2dheat_weak_scalability_AA} repeats the same study, but with
all relaxation methods using Anderson acceleration with $\memory=5$.
By comparing with the corresponding results without Anderson acceleration
shown in Fig.~\ref{fig:2dheat_weak_scalability}, it is again clear that
employing Anderson acceleration substantially improves the performance of
NetUQ. As in the strong-scaling case, we again observe speedup saturation at
$\nsubsystems=16$. As before, we also observe that Anderson acceleration
reduces the performance discrepancy arising from different values of the
relaxation factor $\relaxation$.

\begin{figure}[h!]
	\centering
	\begin{subfigure}{0.4\textwidth}
		\centering
		\includegraphics[width=0.95\textwidth]{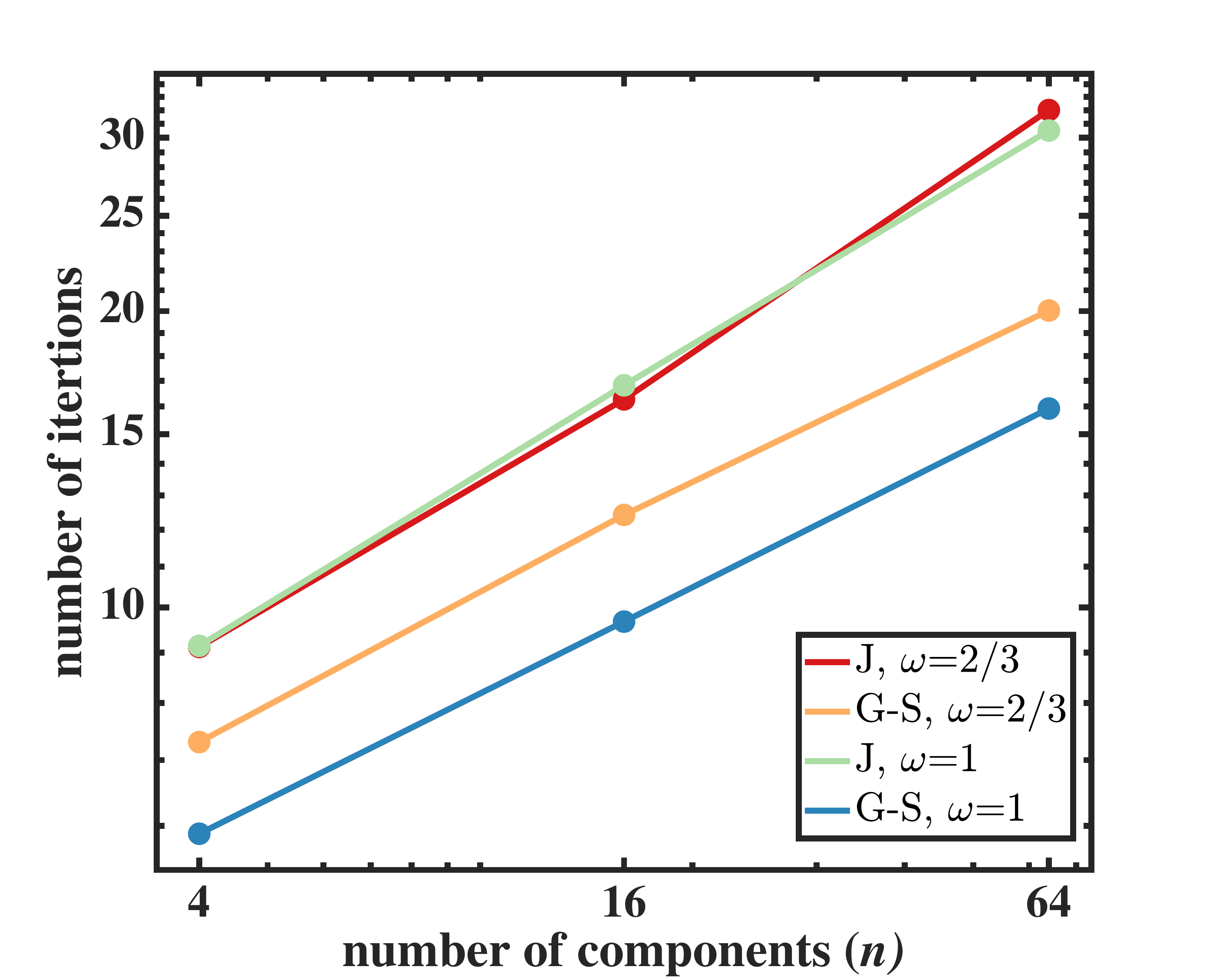}
		\caption{Number of iterations v.\ number of components}
	\end{subfigure}
	\begin{subfigure}{0.4\textwidth}
		\centering
		\includegraphics[width=0.95\textwidth]{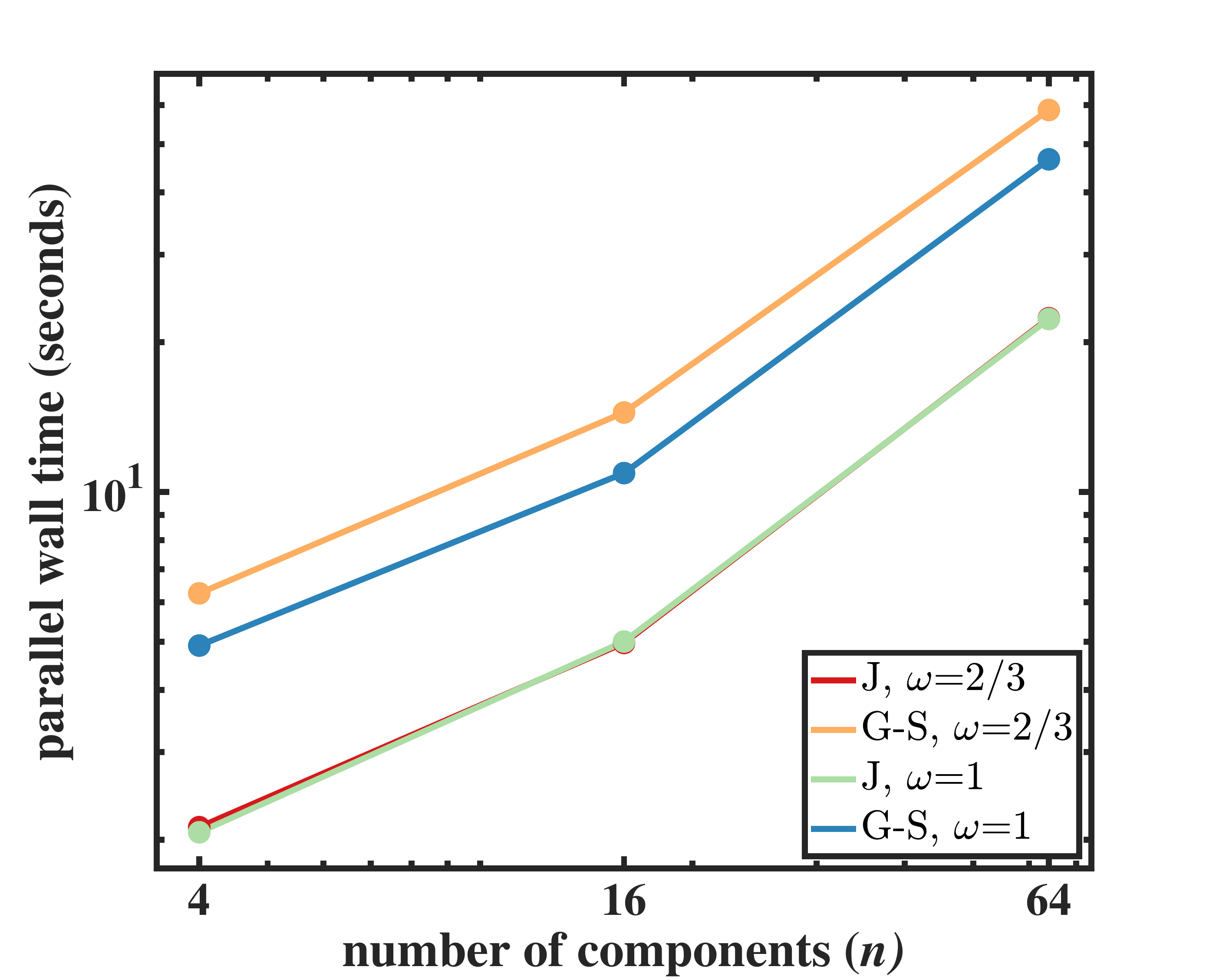}
		\caption{Parallel wall time v.\ number of components}
	\end{subfigure}
	\begin{subfigure}{0.4\textwidth}
		\centering
		\includegraphics[width=0.95\textwidth]{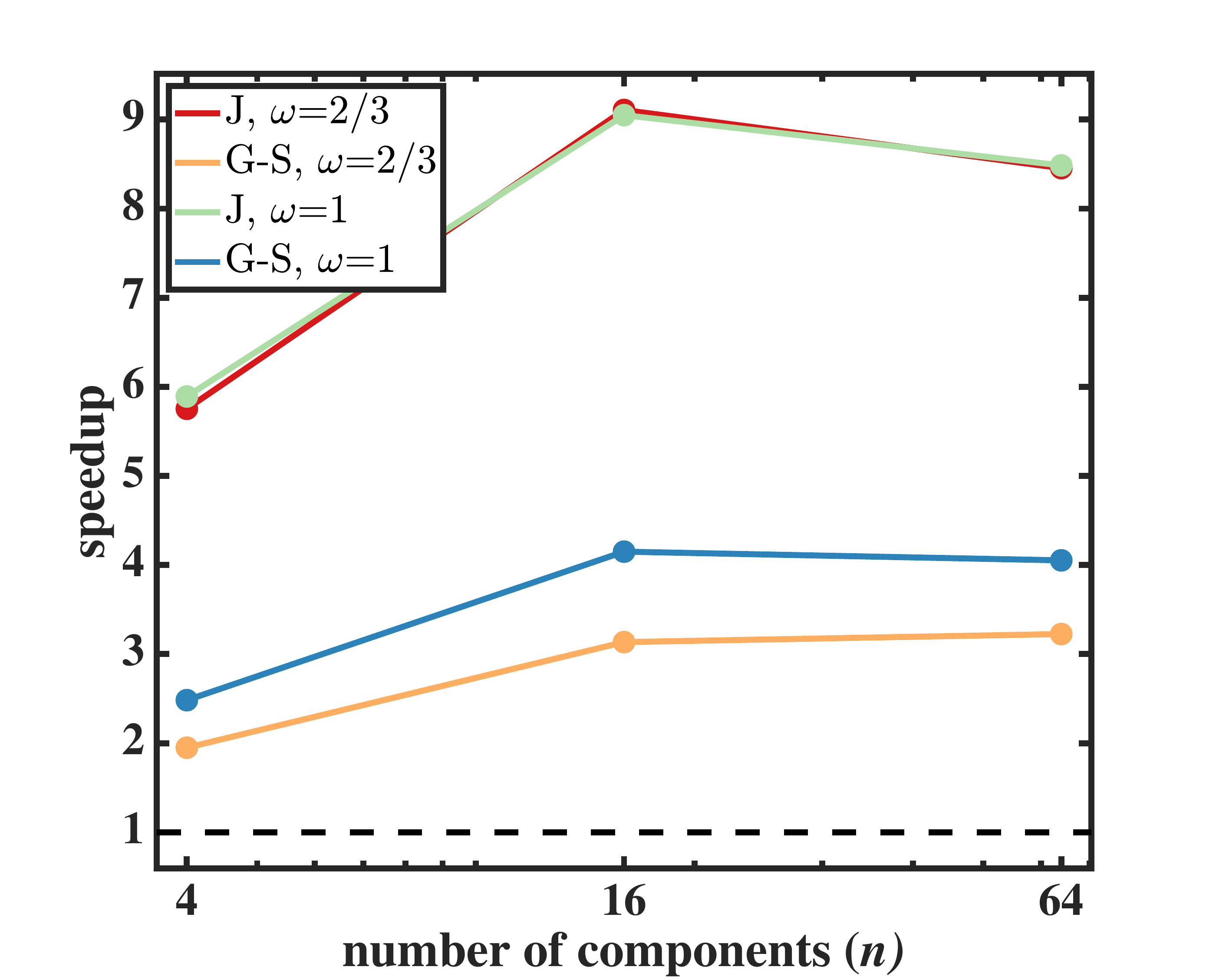}
		\caption{Speedup v.\ number of components}
	\end{subfigure}
	\caption{
	\captionFirst	Weak-scaling study
	using
	Anderson acceleration with $\memory = 5$.
Iterations are terminated when the relative
	residual reaches a value of $10^{-3}$.
J and G-S
	labels refer to the Jacobi and Gauss--Seidel methods, respectively.
}
	\label{fig:2dheat_weak_scalability_AA}
\end{figure}

Finally, Fig.~\ref{fig:2dheat_weak_scalability_perms_AA}
assesses the effect of the permutation $\permutationTuple$ on
the performance of the Gauss--Seidel method in weak scaling. Results are
similar to the strong-scaling case: the choice of permutation tuple has
essentially no effect on the number of iterations required for convergence, so
performance is driven by the number of sequential steps per iteration
$\nsequential$ exposed by the choice of permutation.

%

\begin{figure}[h!]
	\centering
	\begin{subfigure}{0.4\textwidth}
		\centering
		\includegraphics[width=0.95\textwidth]{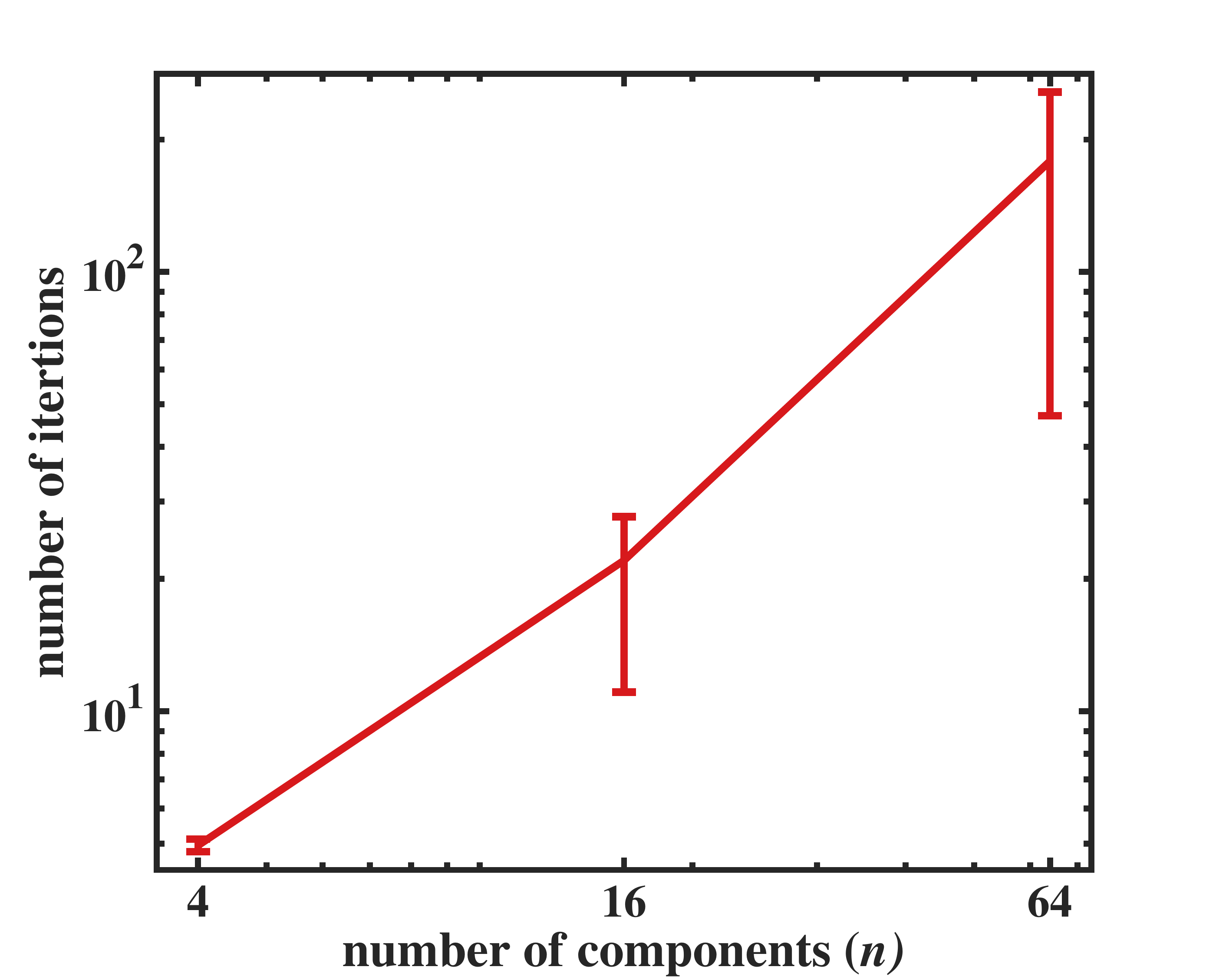}
		\caption{Number of iterations v.\ number of components}
	\end{subfigure}
	\begin{subfigure}{0.4\textwidth}
		\centering
		\includegraphics[width=0.95\textwidth]{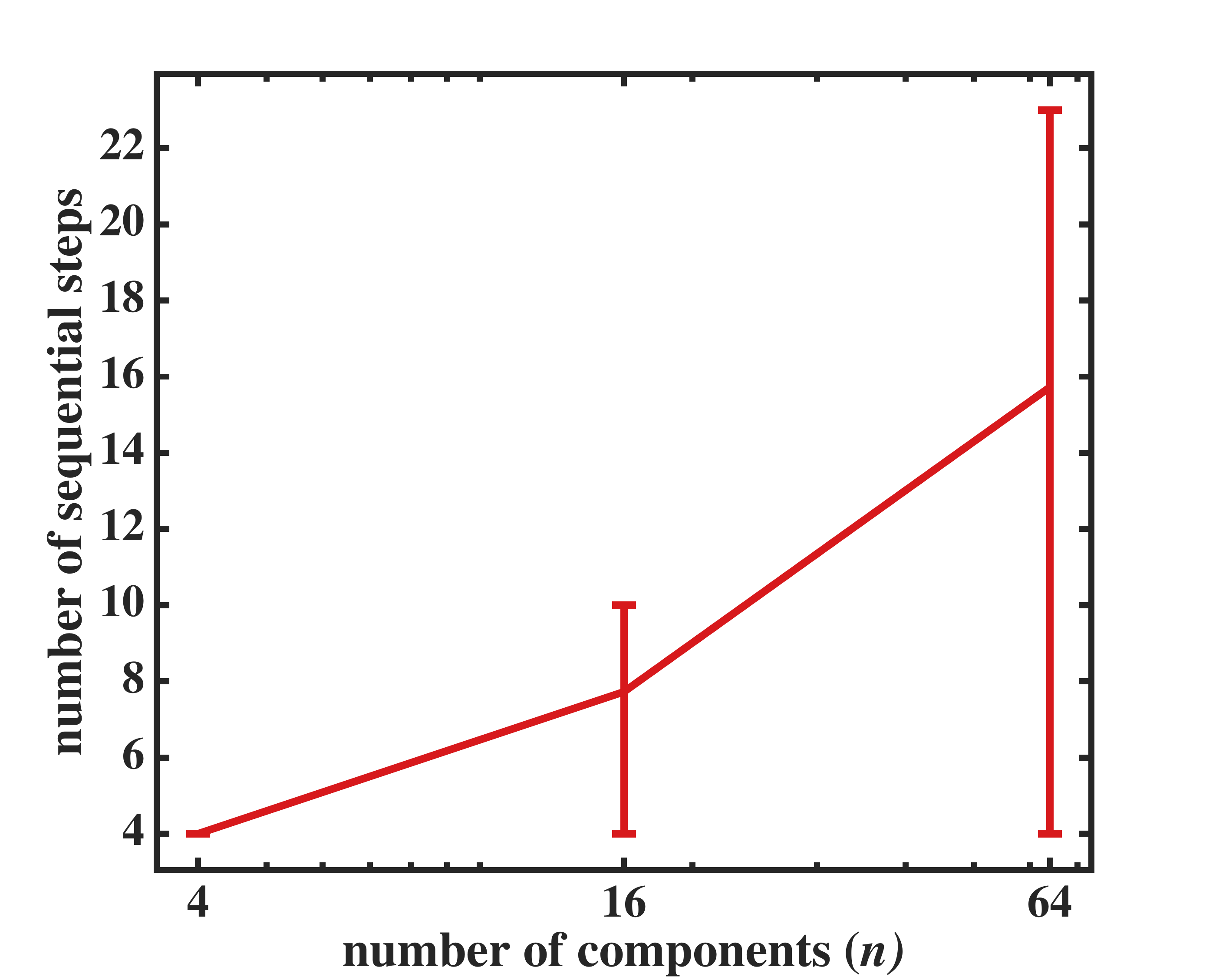}
		\caption{Number of sequential steps per iteration v.\ number of components}
	\end{subfigure}
	\begin{subfigure}{0.4\textwidth}
		\centering
		\includegraphics[width=0.95\textwidth]{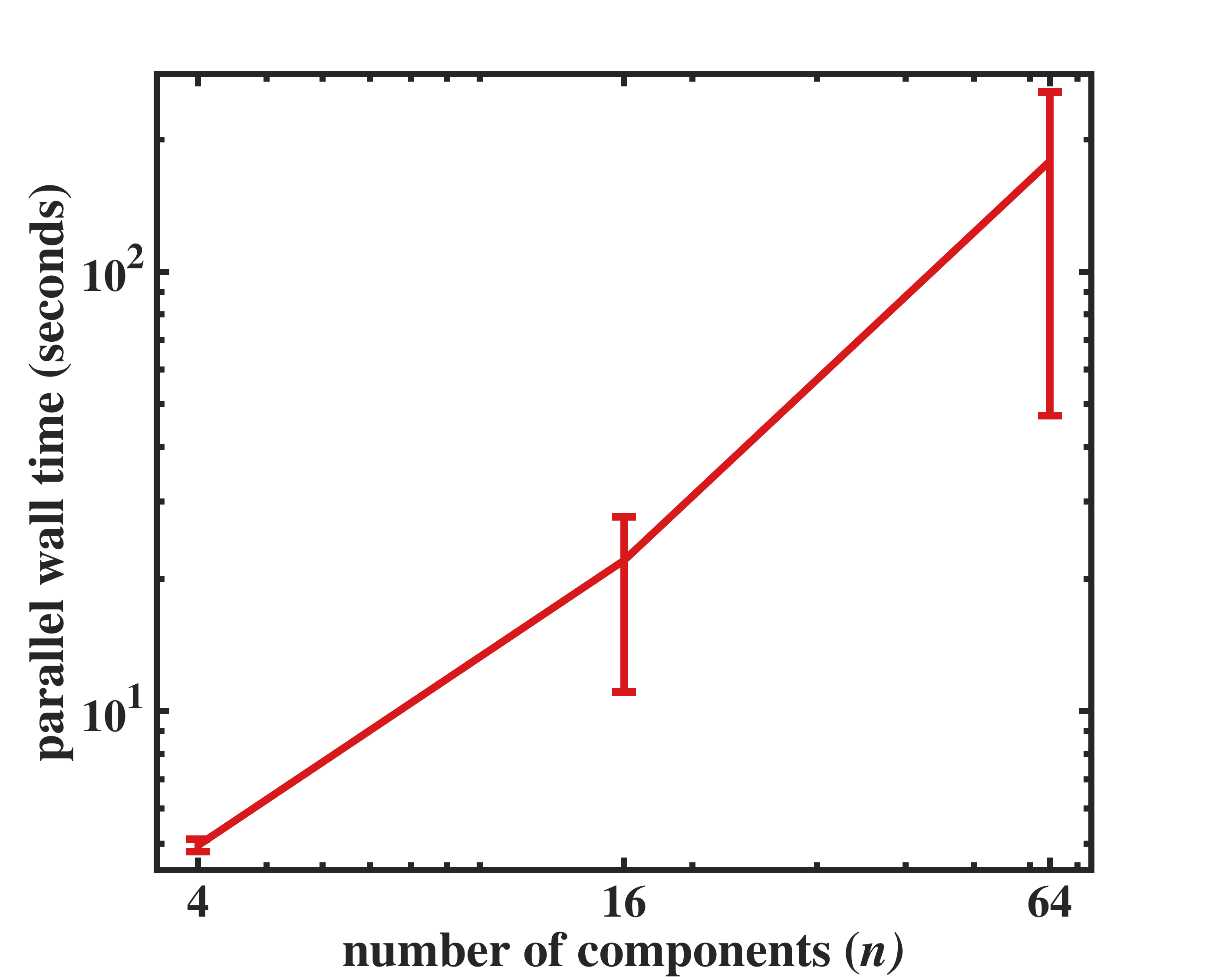}
		\caption{Parallel wall time v.\ number of components}
	\end{subfigure}
	\begin{subfigure}{0.4\textwidth}
		\centering
		\includegraphics[width=0.95\textwidth]{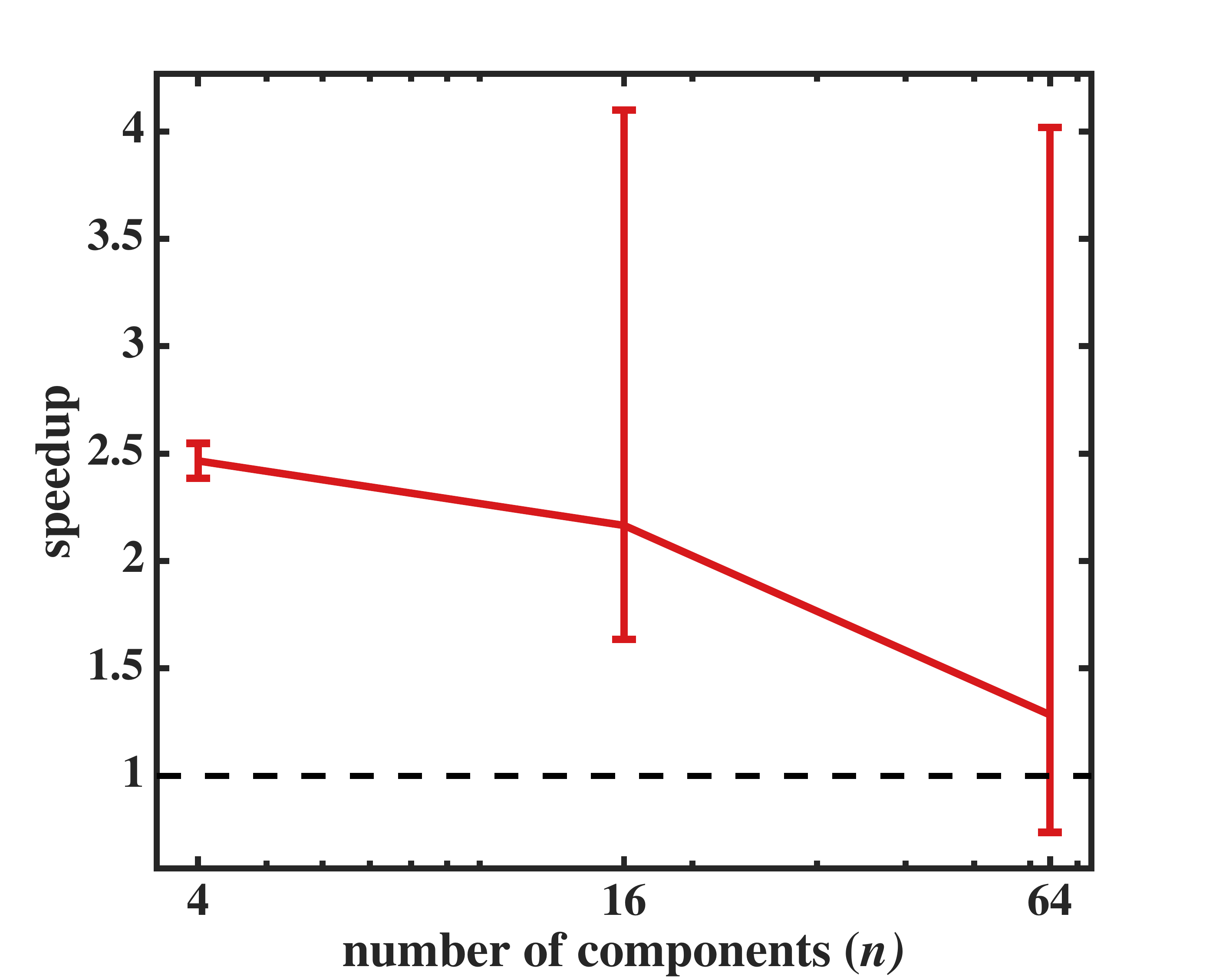}
		\caption{Speedup v.\ number of components}
	\end{subfigure}
	\caption{\captionFirst
	Weak-scaling study using
	Anderson acceleration with $\memory = 5$.
	Gauss--Seidel with $\relaxation=1$ performance for ten random permutations
	$\permutationTuple$.
	Red curve represents the mean value, and vertical lines span the
	minimum and maximum values computed over these permutations.
 }
	\label{fig:2dheat_weak_scalability_perms_AA}
\end{figure}

\section{Conclusions}\label{sec:conclusions}

This work proposed the network uncertainty quantification (NetUQ) method for
propagating uncertainties in large-scale networks. The approach simply assumes
the existence of a collection of components, each of which is characterized by
exogenous-input random variables, endogenous-input random variables, output
random variables, and an uncertainty-propagation operator. The method
constructs the network simply by associating endogenous-input random variables
for each component with output random variables from other components; no
other inter-component compatibility conditions are required. This formulation
promotes component independence by enabling different components to use
different functional representations of random variables and different
uncertainty-propagation operators.

To resolve the resulting network uncertainty-propagation problem, the method
employs the classical relaxation methods of Jacobi and Gauss--Seidel iteration
equipped with Anderson acceleration. Each Jacobi iteration entails
embarrassingly parallel component uncertainty propagation, while each
Gauss--Seidel iteration incurs feed-forward uncertainty propagation
in a DAG created by ``splitting'' selected edges in the network. Critically,
no two-way coupled solves between components are required within a given
relaxation iteration.

In addition to proposing the NetUQ method, this work provided supporting error
analysis (Section \ref{sec:error}), which is applicable to the case where the
component uncertainty-propagation operators comprise an approximation of an
underlying ``truth'' uncertainty-propagation operator. These results include
\textit{a priori} (Proposition \ref{prop:apriori}) and \textit{a
posteriori} (Proposition \ref{prop:apost}) error bounds for the general case.
In addition, this section performed error analysis for the case where the
uncertainty-propagation operator restricts the solution to lie in a subspace
of the space considered by the truth uncertainty-propagation operator,
including conditions for zero in-plane error (Propositions \ref{cor:LHSzero}
and \ref{cor:LHSzeroSecond}), and \textit{a priori} (Proposition
\ref{prop:alternativeApriori}) and \textit{a posteriori} (Proposition
\ref{prop:alternativeApost}) in-plane error bounds.

Numerical experiments performed in Section \ref{sec:numericalExperiments}
studied the strong-scaling (Section \ref{sec:strong}) and weak-scaling
(Section \ref{sec:weak}) performance of the method on a benchmark
parameterized diffusion problem. These results illustrate that convergence
occurs in the fewest number of iterations for smaller networks, Gauss--Seidel
iteration, and a relaxation factor of $\relaxation=1$. However, due to the
embarrassingly  parallel nature of Jacobi iterations, the Jacobi method yields
superior parallel wall-time performance compared with the Gauss--Seidel
method. Critically, numerical experiments also demonstrated that Anderson
acceleration is essential for generating substantial speedups relative to
monolithic full-system uncertainty propagation.

Future work entails demonstrating the methodology on problems characterized by
greater discrepancies between components, investigating other mechanisms to
accelerate convergence of the fixed-point iterations, and integrating
surrogate models to reduce the wall-time incurred by component deterministic
simulations used for component uncertainty propagation.  Preliminary results
in this direction have been presented in Ref.~\cite{guzzettiNS}, which applied
NetUQ to propagate uncertainties in large-scale 3D hemodynamics models, and
employed reduced-order models to reduce the computational cost of solving the
component deterministic problems.

\section*{Acknowledgements}
The authors gratefully acknowledge Jiahua Jiang for helpful initial
experiments performed with Anderson acceleration.  This work was performed at
Sandia National Laboratories and was supported by the LDRD program (project
190968).  This paper describes objective technical results and analysis. Any
subjective views or opinions that might be expressed in the paper do not
necessarily represent the views of the U.S. Department of Energy or the United
States Government. Sandia National Laboratories is a multimission laboratory
managed and operated by National Technology \& Engineering Solutions of
Sandia, LLC, a wholly owned subsidiary of Honeywell International Inc., for
the U.S.\ Department of Energy's National Nuclear Security Administration
under contract DE-NA0003525.

\bibliographystyle{siam}
\bibliography{references}
\end{document}